\documentclass[a4paper,10pt,reqno]{amsart}%
\usepackage[utf8]{inputenc}%
\usepackage[T1]{fontenc}%
\usepackage[a4paper,left=3cm,right=3cm,top=3cm,bottom=3.5cm]{geometry}%
\usepackage[pagebackref]{hyperref}%
\usepackage{amssymb,color}%
\usepackage[english]{babel}
\usepackage{marginnote}%
\usepackage{draftwatermark}%
\SetWatermarkText{work in progress}%
\SetWatermarkScale{.5}%
\SetWatermarkColor[gray]{.98}%
\newtheorem{theorem}{Theorem}[section]%
\newtheorem{lemma}[theorem]{Lemma}%
\newtheorem{remark}[theorem]{Remark}%
\newtheorem{corollary}[theorem]{Corollary}%
\DeclareMathOperator{\G}{G}
\allowdisplaybreaks
\numberwithin{equation}{section}%
\title{Universal cutoff for Dyson Ornstein Uhlenbeck process}%
\author{Jeanne Boursier}
\address{
(JB) Universit\'e Paris-Dauphine, PSL University, UMR 7534, CNRS, CEREMADE, 75016 Paris, France}
\email{boursier@ceremade.dauphine.fr}
\author{Djalil Chafaï}
\address{
(DC) \'Ecole Normale Sup\'erieure, UMR 8553, CNRS, DMA, 75005 Paris, France \& 
Universit\'e Paris-Dauphine, PSL University, UMR 7534, CNRS, CEREMADE, 75016 Paris, France}
\email{djalil@chafai.net}
\urladdr{https://djalil.chafai.net/}
\author{Cyril Labbé}
\address{
(CL) Université de Paris, Laboratoire de Probabilités, Statistiques et Modélisation, UMR 8001, F-75205 Paris, France}
\email{clabbe@lpsm.paris}%
\date{Summer 2021, revised Winter 2022, revised Summer 2022, compiled \today}%
\begin{document}
\begin{abstract}
  We study the Dyson--Ornstein--Uhlenbeck diffusion process, an evolving gas
  of interacting particles. Its invariant law is the beta Hermite ensemble of
  random matrix theory, a non-product log-concave distribution. We explore the
  convergence to equilibrium of this process for various distances or
  divergences, including total variation, relative entropy, and transportation
  cost. When the number of particles is sent to infinity, we show that a
  cutoff phenomenon occurs: the distance to equilibrium vanishes abruptly at a
  critical time. A remarkable feature is that this critical time is
  independent of the parameter beta that controls the strength of the
  interaction, in particular the result is identical in the non-interacting
  case, which is nothing but the Ornstein--Uhlenbeck process. We also provide
  a complete analysis of the non-interacting case that reveals some new
  phenomena. Our work relies among other ingredients on convexity and
  functional inequalities, exact solvability, exact Gaussian formulas,
  coupling arguments, stochastic calculus, variational formulas and
  contraction properties. This work leads, beyond the specific process that we
  study, to questions on the high-dimensional analysis of heat kernels of
  curved diffusions.
\end{abstract}
\maketitle
\setcounter{tocdepth}{1}
\tableofcontents

\section{Introduction and main results}

Let us consider a Markov process $X={(X_t)}_{t\geq0}$ with state space $S$ and
invariant law $\mu$ for which
\[
  \lim_{t\to\infty}\mathrm{dist}(\mathrm{Law}(X_t)\mid\mu)=0
\]
where $\mathrm{dist}(\cdot\mid\cdot)$ is a distance or divergence on the
probability measures on $S$. Suppose now that $X=X^n$ depends on a dimension,
size, or complexity parameter $n$, and let us set $S=S^n$, $\mu=\mu^n$, and
$X_0=x^n_0\in S^n$. For example $X^n$ can be a random walk on the symmetric
group of permutations of $\{1,\ldots,n\}$, Brownian motion on the group of
$n\times n$ unitary matrices, Brownian motion on the $n$-dimensional sphere,
etc. In many of such examples, it has been proved that when $n$ is large
enough, the supremum over some set of initial conditions $x^n_0$ of the
quantity $\mathrm{dist}(\mathrm{Law}(X_t^n)\mid\mu^n)$ collapses abruptly to
$0$ when $t$ passes a critical value $c=c_n$ which may depend on $n$. This is
often referred to as a \emph{cutoff phenomenon}. More precisely, if
$\mathrm{dist}$ ranges from $0$ to $\max$, then, for some subset
$S^n_0 \subset S^n$ of initial conditions, some critical value $c=c_n$ and for
all $\varepsilon\in(0,1)$,
\[
  \lim_{n\to\infty}%
   \sup_{x^n_0 \in S^n_0}
  \mathrm{dist}(\mathrm{Law}(X^n_{t_n})\mid\mu^n)%
  =\begin{cases}%
    \max & \text{if $t_n=(1-\varepsilon)c_n$}\\%
    0 & \text{if $t_n=(1+\varepsilon)c_n$}%
  \end{cases}.
\]
It is standard to introduce, for an arbitrary small threshold $\eta>0$, the
quantity
$\inf\{t\geq0:\sup_{x_0\in
  S_0^n}\mathrm{dist}(\mathrm{Law}(X^n_t)\mid\mu^n)\leq\eta\}$ known as the
\emph{mixing time} in the literature. Of course such a definition fully makes
sense as soon as
$t\mapsto{\sup_{x_0\in S_0^n}}\mathrm{dist}(\mathrm{Law}(X^n_t)\mid\mu^n)$ is
non-increasing.

When $S^n$ is finite, it is customary to take $S^n_0 = S^n$. When $S^n$ is
infinite, it may happen that the supremum over the whole set $S^n$ of the
distance to equilibrium remains equal to $\max$ at all times, in which case
one has to consider strict subspaces of initial conditions. For some
processes, it is possible to restrict $S^n_0$ to a single state in which case
one obtains a very precise description of the convergence to equilibrium
starting from this initial condition. Note that the constraint over the
initial condition can be made compatible with a limiting dynamics, for
instance a mean-field limit when the process describes an exchangeable
interacting particle system.

The \emph{cutoff phenomenon} was put forward by Aldous and Diaconis at the
origin for random walks on finite sets, see for instance
\cite{zbMATH03973935,MR1374011,MR2375599,zbMATH06813269} and references
therein. The analysis of the cutoff phenomenon is the subject of an important
activity, still seeking for a complete theory: let us mention that, for the
total variation distance, Peres proposed the so-called product condition (the
mixing time must be much larger than the inverse of the spectral gap) as a
necessary and sufficient condition for a cutoff phenomenon to hold, but
counter-examples were exhibited~\cite[Sec.~18.3]{zbMATH06813269} and the
product condition is only necessary.

The study of the cutoff phenomenon for Markov diffusion processes goes back at
least to the works of Saloff-Coste \cite{saloff1994precise,MR2111426} in
relation notably with Nash--Sobolev type functional inequalities, heat kernel
analysis, and Diaconis--Wilson probabilistic techniques. We also refer to the
more recent work \cite{meliot2014cut} for the case of diffusion processes on
compact groups and symmetric spaces, in relation with group invariance and
representation theory, a point of view inspired by the early works of Diaconis
on Markov chains and of Saloff-Coste on diffusion processes. Even if most of
the available results in the literature on the cutoff phenomenon are related
to compact state spaces, there are some notable works devoted to non-compact
spaces such as
\cite{MR2203823,CLL2,10.1214/19-AAP1526,barrera-hogele-pardo-langevin,barrera-hogele-pardo-small-levy,MR4073676}.

Our contribution is an exploration of the cutoff phenomenon for the
Dyson--Ornstein--Uhlenbeck diffusion process, for which the state space is
$\mathbb{R}^n$. This process is an interacting particle system. When the
interaction is turned off, we recover the Ornstein--Uhlenbeck process, a
special case that has been considered previously in the literature but for
which we also provide new results.

\subsection{Distances}

As for $\mathrm{dist}$ we use several standard distances or divergences
between probability measures: Wasserstein, total variation (TV), Hellinger,
Entropy, $\chi^2$ and Fisher, surveyed in Appendix \ref{ap:distances}. We take
the following convention for probability measures $\mu$ and $\nu$ on the same
space:
\begin{equation}\label{eq:dist}
  \mathrm{dist}(\mu\mid\nu)
  =\begin{cases}
    \mathrm{Wasserstein}(\mu,\nu) &\text{when
      $\mathrm{dist}=\mathrm{Wasserstein}$}\\
    \left\|\mu-\nu\right\|_{\mathrm{TV}} &\text{when
      $\mathrm{dist}=\mathrm{TV}$}\\
    \mathrm{Hellinger}(\mu,\nu)
    &\text{when $\mathrm{dist}=\mathrm{Hellinger}$}\\
    \mathrm{Kullback}(\mu\mid\nu)
    &\text{when $\mathrm{dist}=\mathrm{Kullback}$}\\
    \chi^2(\mu\mid\nu)
    &\text{when $\mathrm{dist}=\chi^2$}\\
    \mathrm{Fisher}(\mu\mid\nu)
    &\text{when $\mathrm{dist}=\mathrm{Fisher}$}
  \end{cases},
\end{equation}
see Appendix \ref{ap:distances} for precise definitions. The maximal value $\max$ taken
by $\mathrm{dist}$ is given by
\begin{equation}\label{eq:max}
  \max = \begin{cases}
    1
    &\text{ if } \mathrm{dist} \in \{ \mathrm{TV},\mathrm{Hellinger}\},\\
    +\infty
    &\text{ if } \mathrm{dist} \in \{\mathrm{Wasserstein},\mathrm{Kullback}, \chi^2, \mathrm{Fisher}\}.
  \end{cases}
\end{equation}

\subsection{The Dyson--Ornstein--Uhlenbeck (DOU) process and preview of main
  results} The DOU process is the solution $X^n={(X^n_t)}_{t\geq0}$ on
$\mathbb{R}^n$ of the stochastic differential equation
\begin{equation}\label{eq:DOU}
  X^n_0=x^n_0\in\mathbb{R}^n,\quad
  \mathrm{d}X_t^{n,i}
  =\sqrt{\frac{2}{n}}\mathrm{d}B^i_t
  -V'(X_t^{n,i})\mathrm{d}t
  +\frac{\beta}{n}\sum_{j\neq i}\frac{\mathrm{d}t}{X_t^{n,i}-X_t^{n,j}},\quad 1\leq
  i\leq n,
\end{equation}
where ${(B_t)}_{t\geq0}$ is a standard $n$-dimensional Brownian motion (BM), and where
\begin{itemize}
\item $V(x)=\frac{x^2}{2}$ is a ``confinement potential'' acting through the
  drift $-V'(x)=-x$
\item $\beta\geq0$ is a parameter tuning the interaction strength.
\end{itemize}
The notation $X^{n,i}_t$ stands for the $i$-th coordinate of the vector
$X^n_t$. The process $X^n$ can be thought of as an interacting particle system
of $n$ one-dimensional Brownian particles $X^{n,1},\ldots,X^{n,n}$, subject to
confinement and singular pairwise repulsion when $\beta>0$ (respectively first
and second term in the drift). We take an inverse temperature of order $n$ in
\eqref{eq:DOU} in order to obtain a mean-field limit without time-changing the
process, see Section \ref{se:meanfield}. The spectral gap is $1$ for all
$n\geq1$, see Section \ref{ss:Lp}. We refer to Section \ref{se:altparam} for
other parametrizations or choices of inverse temperature.

In the special cases $\beta\in\{0,1,2\}$, the cutoff phenomenon for the DOU
process can be established by using Gaussian analysis and stochastic calculus,
see Sections \ref{ss:OU} and \ref{ss:exact}. For $\beta = 0$, the process
reduces to the Ornstein--Uhlenbeck process (OU) and its behavior serves as a
benchmark for the interaction case $\beta\neq0$, while when $\beta\in\{1,2\}$,
the approach involves a lift to unitary invariant ensembles of random matrix
theory. For a general $\beta\geq1$, our main results regarding the cutoff
phenomenon for the DOU process are given in Sections \ref{ss:cutoffgen} and
\ref{ss:nonpoint}. We are able, in particular, to prove the following: for all
$\mathrm{dist}\in\{\mathrm{Wasserstein},\mathrm{TV},\mathrm{Hellinger}\}$,
$a>0$, $\varepsilon \in (0,1)$, we have
\[
  \lim_{n\to\infty}%
  \sup_{x^n_0 \in [-a,a]^n}
  \mathrm{dist}(\mathrm{Law}(X^n_{t_n})\mid P_n^\beta)
  =\begin{cases}
    \max & \text{if $t_n=(1-\varepsilon)c_n$}\\
    0 & \text{if $t_n=(1+\varepsilon)c_n$}
  \end{cases},
\]
where $P_n^\beta$ is the invariant law of the process, and where
\[
  c_n
  := \begin{cases}
    \log(\sqrt{n}a)
    &\text{ if $\mathrm{dist} = \mathrm{Wasserstein}$}\\
    \log(na)
    &\text{ if $\mathrm{dist} \in \{\mathrm{TV},\mathrm{Hellinger}\}$}
  \end{cases}.
\]
This result is stated in a slightly more general form in Corollary
\ref{th:DOUWTV}. Our proof relies crucially on an exceptional exact
solvability of the dynamics, notably the fact that we know explicitly the
optimal long time behavior in entropy and coupling distance, as well as the
eigenfunction associated to the spectral gap which turns out to be linear and
optimal. This comes from the special choice of $V$ as well as the special
properties of the Coulomb interaction. We stress that such an exact
solvability is no longer available for a general strongly convex $V$, even for
instance in the simple example $V(x)=\frac{x^2}{2}+x^4$ or for general linear
forces. Nevertheless, and as usual, two other special classical choices of $V$
could be explored, related to Laguerre and Jacobi weights, see Section
\ref{sec:genpot}.


\subsection{Analysis of the Dyson--Ornstein--Uhlenbeck process}

The process $X^n$ was essentially discovered by Dyson in \cite{MR148397}, in
the case $\beta\in\{1,2,4\}$, because it describes the dynamics of the
eigenvalues of $n\times n$ symmetric/Hermitian/symplectic random matrices with
independent Ornstein--Uhlenbeck entries, see Lemma \ref{le:DOU} and Lemma
\ref{le:DOUS} below for the cases $\beta=1$ and $\beta=2$ respectively.

\begin{itemize}
\item Case $\beta=0$ (interaction turned off). The particles become $n$
  independent one-dimensional Ornstein--Uhlenbeck processes, and the DOU
  process $X^n$ becomes exactly the $n$-dimensional Ornstein--Uhlenbeck
  process $Z^n$ solving \eqref{eq:OU}. The process lives in $\mathbb{R}^n$.
  The particles collide but since they do not interact, this does not raise
  any issue.
\item Case $0<\beta<1$. Then with positive probability the particles collide
  producing a blow up of the drift, see for instance
  \cite{MR1440140,zbMATH07238061} for a discussion. Nevertheless, it is
  possible to define the process for all times, for instance by adding a local
  time term to the stochastic differential equation, see \cite{zbMATH07238061}
  and references therein. It is natural to expect that the cutoff universality
  works as for $\beta\not\in(0,1)$, but for simplicity we do not consider this
  case here.
\item Case $\beta\geq1$. If we order the coordinates by defining the convex
  domain
  \[
    D_n=\{x\in\mathbb{R}^n:x_1<\cdots<x_n\},
  \]
  and if $x^n_0\in D_n$ then the equation \eqref{eq:DOU} admits a unique
  strong solution that never exits $D_n$, in other words the particles never
  collide and the order of the initial particles is preserved at all times,
  see~\cite{rogers1993interacting}. Moreover if
  \[
    \overline{D}_n=\{x\in\mathbb{R}^n:x_1\leq\cdots\leq x_n\}
  \]
  then it is possible to start the process from the boundary
  $\overline{D}_n\setminus D_n$, in particular from $x^n_0$ such that
  $x^{n,1}_0=\cdots=x^{n,n}_0$, and despite the singularity of the drift, it
  can be shown that with probability one, $X^n_t\in D_n$ for all $t>0$. We
  refer to \cite[Th.~4.3.2]{MR2760897} for a proof in the Dyson Brownian
  Motion case that can be adapted \emph{mutatis mutandis}.
\end{itemize}

In the sequel, we will only consider the cases $\beta=0$ with
$x_0^n\in\mathbb{R}^n$ and $\beta\geq1$ with $x^n_0\in \overline{D}_n$.

The drift in \eqref{eq:DOU} is the gradient of a function, and \eqref{eq:DOU}
rewrites
\begin{equation}\label{eq:DOU2}
  X_0^n=x_0^n\in D_n,\quad
  \mathrm{d}X_t^n=\sqrt{\frac{2}{n}}\mathrm{d}B_t-\frac1{n}\nabla E(X_t^n)\mathrm{d}t,
\end{equation}
where
\begin{equation}\label{eq:E}
  E(x_1,\ldots,x_n)
  =n\sum_{i=1}^nV(x_i)
  +{\beta}\sum_{i>j}\log\frac{1}{|x_i-x_j|}
\end{equation}
can be interpreted as the energy of the configuration of particles
$x_1,\ldots,x_n$.
\begin{itemize}
\item If $\beta=0$, then the Markov process $X^n$ is an Ornstein--Uhlenbeck
  process, irreducible with unique invariant law
  $P_n^0=\mathcal{N}(0,\frac{1}{n}I_n)$ which is reversible.
\item If $\beta\geq1$, then the Markov process $X^n$ is not irreducible, but
  $D_n$ is a recurrent class carrying a unique invariant law $P_n^\beta$,
  which is reversible and given by
  \begin{equation}\label{eq:P}
    P_n^\beta
    =\frac{\mathrm{e}^{-E(x_1,\ldots,x_n)}}{C_n^\beta}\mathbf{1}_{(x_1,\ldots,x_n)\in\overline{D}_n}
    \mathrm{d}x_1\cdots\mathrm{d}x_n,
  \end{equation}
  where $C_n^\beta$ is the normalizing factor given by
  \begin{equation}\label{eq:ZD}
    C_n^\beta=\int_{\overline{D}_n}
    \mathrm{e}^{-E(x_1,\ldots,x_n)}\mathrm{d}x_1\cdots\mathrm{d}x_n.
  \end{equation}
\end{itemize}

In terms of geometry, it is crucial to observe that since $-\log$ is convex on $(0,+\infty)$, the map
\[
 (x_1,\ldots,x_n)\in D_n\mapsto
 \mathrm{Interaction}(x_1,\ldots,x_n)
  ={\beta}\sum_{i>j}\log\frac{1}{x_i-x_j},
\]
is convex. Thus, since $V$ is convex on $\mathbb{R}$, it follows that $E$ is
convex on $D_n$. For all $\beta\geq0$, the law $P_n^\beta$ is log-concave with
respect to the Lebesgue measure as well as with respect to
$\mathcal{N}(0,\frac{1}{n}I_n)$.

\subsection{Non-interacting case and Ornstein--Uhlenbeck benchmark}
\label{ss:OU}

When we turn off the interaction by taking $\beta=0$ in \eqref{eq:DOU}, the
DOU process becomes an Ornstein--Uhlenbeck process (OU)
$Z^n={(Z^n_t)}_{t\geq0}$ on $\mathbb{R}^n$ solving the stochastic differential
equation
\begin{equation}\label{eq:OU}
  Z^n_0=z^n_0\in\mathbb{R}^n,\quad
  \mathrm{d}Z^n_t=\sqrt{\frac{2}{n}}\mathrm{d}B^n_t-Z^n_t\mathrm{d}t,
\end{equation}
where $B^n$ is a standard $n$-dimensional BM. The invariant law of $Z^n$ is
the product Gaussian law
$P_n^0=\mathcal{N}(0,\frac{1}{n}I_n)=\mathcal{N}(0,\frac{1}{n})^{\otimes n}$.
The explicit Gaussian nature of
$Z_t^n\sim\mathcal{N}(z_0^n\mathrm{e}^{-t},\frac{1-\mathrm{e}^{-2t}}{n}I_n)$,
valid for all $t\geq0$, allows for a fine analysis of convergence to
equilibrium, as in the following theorem.

\begin{theorem}[Cutoff for OU: mean-field regime]\label{th:OU1}
  Let $Z^n={(Z^n_t)}_{t\geq0}$ be the OU process \eqref{eq:OU} and let $P_n^0$
  be its invariant law. Suppose that
  \[
    \varliminf_{n\to\infty}\frac{|z^n_0|^2}{n}>0
    \quad\text{and}\quad
    \varlimsup_{n\to\infty}\frac{|z^n_0|^2}{n}<\infty
  \]
  where $|z|=\sqrt{z_1^2+\cdots+z_n^2}$ is the Euclidean norm. Then for all
  $\varepsilon\in(0,1)$,
  \[
    \lim_{n\to\infty}\mathrm{dist}(\mathrm{Law}(Z^n_{t_n})\mid P_n^0)
    =\begin{cases}
      \max & \text{if $t_n=(1-\varepsilon)c_n$}\\
      0 & \text{if $t_n=(1+\varepsilon)c_n$}
      \end{cases}
  \]
  where
  \[
    c_n
    =\begin{cases}
      \frac{1}{2}\log(n) & \text{if $\mathrm{dist}=\mathrm{Wasserstein}$},\\
      \log(n) & \text{if $\mathrm{dist}\in\{\mathrm{TV},\mathrm{Hellinger},\mathrm{Kullback},\chi^2\}$},\\
      \frac{3}{2}\log(n) & \text{if $\mathrm{dist}=\mathrm{Fisher}$}.
    \end{cases}
  \]
\end{theorem}

Theorem \ref{th:OU1} is proved in Section \ref{se:OU}. See Figure
\ref{fi:oudou} and Figure \ref{fi:hellinger} for a numerical experiment.

Theorem \ref{th:OU1} constitutes a very natural benchmark for the cutoff
phenomenon for the DOU process. Theorem \ref{th:OU1} is not a surprise, and
actually the TV and Hellinger cases are already considered in
\cite{MR2203823}, see also \cite{Barrera}. Let us mention that
in~\cite{barrera-hogele-pardo-small-levy}, a cutoff phenomenon for TV, entropy
and Wasserstein is proven for the OU process of \emph{fixed} dimension $d$ and
vanishing noise. This is to be compared with our setting where the dimension
is sent to infinity: the results (and their proofs) are essentially the same
in these two situations, however we will see below that if one considers more
general initial conditions, there are some substantial differences according
to whether the dimension is fixed or sent to infinity.%


The restriction over the initial condition in Theorem \ref{th:OU1} is spelled
out in terms of the second moment of the empirical distribution, a natural
choice suggested by the mean-field limit discussed in Section
\ref{se:meanfield}. It yields a mixing time of order $\log(n)$, just like for
Brownian motion on compact Lie groups, see \cite{MR2111426,meliot2014cut}. For
the OU process and more generally for overdamped Langevin processes, the
non-compactness of the space is replaced by the confinement or tightness due
to the drift.

Actually, Theorem \ref{th:OU1} is a particular instance of the following, much
more general result that reveals that, except for the Wasserstein distance, a
cutoff phenomenon \emph{always} occurs.

\begin{theorem}[General cutoff for OU]\label{th:OU2}
  Let $Z^n={(Z^n_t)}_{t\geq0}$ be the OU process \eqref{eq:OU} and let $P_n^0$
  be its invariant law. Let
  $\mathrm{dist}\in\{\mathrm{TV},\mathrm{Hellinger},\mathrm{Kullback},\chi^2,\mathrm{Fisher}\}$.
  Then, for all $\varepsilon\in(0,1)$,
  \[
    \lim_{n\to\infty}
    \mathrm{dist}(\mathrm{Law}(Z^n_{t_n})\mid P_n^0)
    =\begin{cases}
      \max & \text{if $t_n=(1-\varepsilon)c_n$,}\\
      0 & \text{if $t_n=(1+\varepsilon)c_n$}
    \end{cases}
  \]
  where
  \[
    c_n =
    \begin{cases}
      \log(\sqrt{n}|z_0^n|) \vee\frac{1}{4}\log(n)
      &\text{if $\mathrm{dist}\in\{\mathrm{TV},\mathrm{Hellinger},\mathrm{Kullback},\chi^2\}$},\\
      \log(n|z_0^n|) \vee\frac{1}{2}\log(n)
      &\text{if $\mathrm{dist} = \mathrm{Fisher}$}.
    \end{cases}
  \]
  Regarding the Wasserstein distance, the following dichotomy occurs:
  \begin{itemize}
  \item if $\lim_{n\to\infty}|z_0^n|=+\infty$, then for all
    $\varepsilon\in(0,1)$, with $c_n=\log|z_0^n|$,
    \[
      \lim_{n\to\infty}
      \mathrm{Wasserstein}(\mathrm{Law}(Z_{t_n}),P_n^0) =
      \begin{cases}
        +\infty & \text{if $t_n=(1-\varepsilon)c_n$,}\\
        0 & \text{if $t_n=(1+\varepsilon)c_n$},
      \end{cases}
    \]
  \item if $\lim_{n\to\infty}|z_0^n|=\alpha \in [0,\infty)$ then there is
    \emph{no cutoff phenomenon} namely for any $t>0$
    \[
      \lim_{n\to\infty}
      \mathrm{Wasserstein}^2(\mathrm{Law}(Z_{t}),P_n^0)
      =
      \alpha^2\mathrm{e}^{-2t} +2\Bigr(1-\sqrt{1-\mathrm{e}^{-2t}} - \tfrac12\mathrm{e}^{-2t}\Bigr).
    \]
  \end{itemize}
\end{theorem}

Theorem \ref{th:OU2} is proved in Section \ref{se:OU}.

The observation that for every distance or divergence, except for the
Wasserstein distance, a cutoff phenomenon occurs \emph{generically} seems to
be new.

Let us make a few comments. First, in terms of convergence to equilibrium the
relevant observable in Theorem \ref{th:OU2} appears to be the Euclidean norm
$|z_0^n|$ of the initial condition. This quantity differs from the
eigenfunction associated to the spectral gap of the generator, which is given
by $z_1+\cdots+z_n$ as we will recall later on. Note by the way that
\eqref{eq:OUL2} and \eqref{eq:l2expansion} are equal! Second, cutoff occurs at
a time that is \emph{independent} of the initial condition provided that its
Euclidean norm is small enough: this cutoff time appears as the time required
to regularize the initial condition (a Dirac mass) into a sufficiently spread
out absolutely continuous probability measure; in particular this cutoff
phenomenon would not hold generically if we allowed for spread out (non-Dirac)
initial conditions. Note that, for the OU process of \emph{fixed} dimension
and vanishing noise, we would not observe a cutoff phenomenon when starting
from initial conditions with small enough Euclidean norm: this is a high
dimensional phenomenon. In this respect, the Wasserstein distance is peculiar
since it is much less stringent on the local behavior of the measures at
stake: for instance
$\lim_{n\to\infty}\mathrm{Wasserstein}(\delta_0,\delta_{1/n})=0$ while for all
other distances or divergences considered here, the corresponding quantity
would remain equal to $\max$. This explains the absence of \emph{generic}
cutoff phenomenon for Wasserstein. Third, the explicit expressions provided in
our proof allow to extract the cutoff profile in each case, but we prefer not
to provide them in our statement and refer the interested reader to the end of
Section \ref{se:OU}.

\subsection{Exactly solvable intermezzo}
\label{ss:exact}

When $\beta\neq0$, the law of the DOU process is no longer Gaussian nor
explicit. However several exactly solvable aspects are available. Let us
recall that a Cox--Ingersoll--Ross process (CIR) of parameters $a,b,\sigma$ is
the solution $R = (R_t)_{t\ge 0}$ on $\mathbb{R}_+$ of
\begin{equation}\label{eq:CIR}
  R_0 = r_0 \in \mathbb{R}_+,\quad %
  \mathrm{d}R_t = \sigma\sqrt{R_t} \mathrm{d}W_t + (a-bR_t) \mathrm{d}t,
\end{equation}
where $W$ is a standard BM. Its invariant law is
$\mathrm{Gamma}(2a/\sigma^2,2b/\sigma^2)$ with density proportional to
$r\ge 0 \mapsto r^{2a/\sigma^2-1}\mathrm{e}^{-2br/\sigma^2}$, with mean $a/b$,
and variance $a\sigma^2/(2b^2)$. It was proved by William Feller in
\cite{zbMATH03069442} that the density of $R_t$ at an arbitrary $t$ can be
expressed in terms of special functions.

If ${(Z_t)}_{t\geq0}$ is a $d$-dimensional OU process of parameters
$\theta\geq0$ and $\rho\in\mathbb{R}$, weak solution of
\begin{equation}\label{eq:OU2}
  \mathrm{d}Z_t=\theta\mathrm{d}W_t-\rho Z_t\mathrm{d}t
\end{equation}
where $W$ is a $d$-dimensional BM, then $R={(R_t)}_{t\geq0}$, $R_t:=|Z_t|^2$,
is a CIR process with parameters $a=\theta^2d$, $b=2\rho$, $\sigma = 2\theta$.
When $\rho=0$ then $Z$ is a BM while $R=|Z|^2$ is a squared Bessel process.

The following theorem gathers some exactly solvable aspects
of the DOU process for general $\beta \ge 1$, which are largely already in the
statistical physics folklore, see \cite{zbMATH07261321}. It is based on our knowledge of eigenfunctions associated to the
first spectral values of the dynamics, see \eqref{eq:eigenfunctions}, and
their remarkable properties. As in \eqref{eq:eigenfunctions}, we set
$\pi(x):=x_1+\cdots+x_n$ when $x\in\mathbb{R}^n$.

\begin{theorem}[From DOU to OU and CIR]
  \label{th:DOUOU}
  Let ${(X^n_t)}_{t\geq0}$ be the DOU process \eqref{eq:DOU}, with $\beta=0$
  or $\beta\geq1$, and let $P_n^\beta$ be its invariant law.
  Then:\begin{itemize}
  \item ${(\pi(X^n_t))}_{t\geq0}$ is a one-dimensional OU process weak
    solution of \eqref{eq:OU} with $\theta=\sqrt{2}$, $\rho=1$. Its invariant
    law is $\mathcal{N}(0,1)$. It does not depend on $\beta$, and
    $\pi(X^n_t)\sim\mathcal{N}(\pi(x^n_0)\mathrm{e}^{-t},1-\mathrm{e}^{-2t})$,
    $t\geq0$. Furthermore $\pi(X^n_t)^2$ is a CIR process of parameters $a=2$,
    $b=2$, $\sigma = 2\sqrt{2}$.
  \item ${(|X^n_t|^2)}_{t\geq0}$ is a CIR process, weak solution of
    \eqref{eq:CIR} with $a=2+\beta(n-1)$, $b=2$, $\sigma = \sqrt{8/n}$. Its
    invariant law is
    $\mathrm{Gamma}(\frac{1}{2}(n+\beta\frac{n(n-1)}{2}),\frac{n}{2})$ of mean
    $1+\frac{\beta}{2}(n-1)$ and variance $\beta+\frac{2-\beta}{n}$.
    Furthermore, if $d=n+\beta\frac{n(n-1)}{2}$ is a positive integer,
    then ${(|X^n_t|^2)}_{t\geq0}$ has the law of ${(|Z_t|^2)}_{t\geq0}$ where
    ${(Z_t)}_{t\geq0}$ is a $d$-dimensional OU process, weak solution of
    \eqref{eq:OU} with $\theta=\sqrt{2/n}$, $\rho=1$, and $Z_0=z^n_0$ for an
    arbitrary $z^n_0\in\mathbb{R}^d$ such that $|z^n_0|=|x^n_0|$.
  \end{itemize}
\end{theorem}

At this step it is worth noting that Theorem \ref{th:DOUOU} gives in
particular, denoting $\beta_n:=1+\frac{\beta}{2}(n-1)$,
\begin{equation}\label{eq:betan}
  \mathbb{E}[\pi(X^n_t)]=\pi(x^n_0)\mathrm{e}^{-t}
  \underset{t\to\infty}{\longrightarrow}0
  \quad\text{and}\quad
  \mathbb{E}[|X^n_t|^2]
  =\beta_n+(|x^n_0|^2-\beta_n)\mathrm{e}^{-2t}
  \underset{t\to\infty}{\longrightarrow}
  \beta_n.
\end{equation}
Following \cite[Sec.~2.2]{zbMATH07238061}, the limits can also be deduced
from the Dumitriu--Edelman tridiagonal random matrix model
\cite{zbMATH02122403} isospectral to $\beta$-Hermite. These formulas for the
``transient'' first two moments $\mathbb{E}[\pi(X_t^n)]$ and
$\mathbb{E}[|X_t^n|^2]$ reveal an abrupt convergence to their equilibrium
values~:
\begin{itemize}
\item If $\lim_{n\to\infty}\frac{\pi(x^n_0)}{n}=\alpha\neq0$ then
  for all $\varepsilon\in(0,1)$,
  \begin{equation}\label{eq:cutoffmom1}
    \lim_{n\to\infty} \vert\mathbb{E}[\pi(X^n_{t_n})]\vert
    =
    \begin{cases}
      +\infty&\text{if $t_n=(1-\varepsilon)\log(n)$}\\
      0&\text{if $t_n=(1+\varepsilon)\log(n)$}
    \end{cases}.
  \end{equation}
\item If $\lim_{n\to\infty}\frac{|x^n_0|^2}{n} = \alpha \ne \frac{\beta}{2}$ then for all
  $\varepsilon\in(0,1)$, denoting $\beta_n:=1+\frac{\beta}{2}(n-1)$,
  \begin{equation}\label{eq:cutoffmom2}
    \lim_{n\to\infty} \left|\mathbb{E}[|X^n_{t_n}|^2]-\beta_n \right|
    =
    \begin{cases}
      +\infty&\text{if $t_n=(1-\varepsilon)\frac{1}{2}\log(n)$}\\
      0&\text{if $t_n=(1+\varepsilon)\frac{1}{2}\log(n)$}
    \end{cases}.
  \end{equation}
\end{itemize}
\noindent These critical times are universal with respect to $\beta$. The
first two transient moments are related to the eigenfunctions
\eqref{eq:eigenfunctions} associated to the first two non-zero eigenvalues of
the dynamics. Higher order transient moments are related to eigenfunctions
associated to higher order eigenvalues. Note that $\mathbb{E}[\pi(X^n_t)]$ and
$\mathbb{E}[|X^n_t|^2]$ are the first two moments of the non-normalized mean
empirical measure $\mathbb{E}[\sum_{i=1}^n\delta_{X^{n,i}_t}]$, and this lack
of normalization is responsible of the critical times of order $\log(n)$. In
contrast, the first two moments of the normalized mean empirical measure
$\mathbb{E}[\frac{1}{n}\sum_{i=1}^n\delta_{X^{n,i}_t}]$, given by
$\frac{1}{n}\mathbb{E}[\pi(X^n_t)]$ and $\frac{1}{n}\mathbb{E}[|X^n_t|^2]$
respectively, do not exhibit a critical phenomenon. This is related to the
exponential decay of the first two moments in the mean-field limit
\eqref{eq:mom12}, as well as the lack of cutoff for Wasserstein already
revealed for OU by Theorem \ref{th:OU2}. This also reminds the high dimension
behavior of norms in the field of the asymptotic geometric analysis of convex
bodies. In another direction, this elementary observation on the moments also
illustrates that the cutoff phenomenon for a given quantity is not stable
under rather simple transformations of this quantity.

\medskip

From the first part of Theorem \ref{th:DOUOU} and contraction properties
available for \emph{some} distances or divergences, see Lemma
\ref{le:contraction}, we obtain the following lower bound on the mixing time
for the DOU, which is independent of $\beta$:

\begin{corollary}[Lower bound on the mixing time]\label{cor:LowerBd}
  Let ${(X^n_t)}_{t\geq0}$ be the DOU process \eqref{eq:DOU} with $\beta=0$ or
  $\beta\geq1$, and invariant law $P_n^\beta$. Let
  $\mathrm{dist}\in\{\mathrm{TV},\mathrm{Hellinger},\mathrm{Kullback}, \chi^2,\mathrm{Wasserstein}\}$. Set
  \[
    c_n :=
      \begin{cases}
        \log(|\pi(x_0^n)|)
        &\text{if $\mathrm{dist}\in\{ \mathrm{TV},\mathrm{Hellinger},\mathrm{Kullback},\chi^2\}$}\\        
        \log\Big(\frac{\vert \pi(x_0^n) \vert}{\sqrt{n}}\Big)
        &\text{if $\mathrm{dist} = \mathrm{Wasserstein}$}
    \end{cases},
  \]
  and assume that $\lim_{n\to\infty}c_n=\infty$. Then, for all
  $\varepsilon\in(0,1)$, we have
  \[
    \lim_{n\to\infty}
    \mathrm{dist}(\mathrm{Law}(X^n_{(1-\varepsilon)c_n})\mid P_n^\beta)
    =\max.
  \]
\end{corollary}

\noindent Theorem \ref{th:DOUOU} and Corollary \ref{cor:LowerBd} are proved in
Section \ref{se:proofs:DOUOU:DOUCIR:moments}.

\medskip

The derivation of an upper bound on the mixing time is much more delicate:
once again recall that the case $\beta = 0$ covered by Theorem \ref{th:OU2} is
specific as it relies on exact Gaussian computations which are no longer
available for $\beta \ge 1$. In the next subsection, we will obtain results
for general values of $\beta \ge 1$ via more elaborate arguments.

In the specific cases $\beta \in \{1,2\}$, there are some exactly solvable
aspects that one can exploit to derive, in particular, precise upper bounds on
the mixing times. Indeed, for these values of $\beta$, the DOU process is the
process of eigenvalues of the matrix-valued OU process:
\[
  M_0 = m_0,\quad \mathrm{d}M_t = \sqrt{\frac{2}{n}} \mathrm{d}B_t - M_t
  \mathrm{d}t,
\]
where $B$ is a BM on the symmetric $n\times n$ matrices if $\beta = 1$ and on
Hermitian $n\times n$ matrices if $\beta =2$, see \eqref{Eq:Herm} and
\eqref{Eq:Sym} for more details. Based on this observation, we can deduce an
upper bound on the mixing times by contraction (for \emph{most} distances or
divergences).

\begin{theorem}[Upper bound on mixing time in matrix case]\label{th:DOU12}
  Let ${(X^n_t)}_{t\geq0}$ be the DOU process \eqref{eq:DOU} with
  $\beta\in\{0,1,2\}$, and invariant law $P_n^\beta$, and
  $\mathrm{dist}\in\{\mathrm{TV},\mathrm{Hellinger},\mathrm{Kullback},\chi^2,\mathrm{Wasserstein}\}$.
  Set
  \[
    c_n :=
    \begin{cases}
      \log(\sqrt{n} |x_0^n|) \vee \log(\sqrt{n})
      &\text{if $\mathrm{dist}\in\{
        \mathrm{TV},\mathrm{Hellinger},\mathrm{Kullback},\chi^2\}$}\\
       \log(|x_0^n|)
      &\text{if $\mathrm{dist} = \mathrm{Wasserstein}$}      
    \end{cases},
  \]
  and assume that $\lim_{n\to\infty}c_n=\infty$ if
  $\mathrm{dist} = \mathrm{Wasserstein}$. Then, for all $\varepsilon\in (0,1)$,
  we have
  \[
      \lim_{n\to\infty}\mathrm{dist}(\mathrm{Law}(X^n_{(1+\varepsilon)c_n})\mid
      P_n^\beta)
      =0.
  \]
\end{theorem}

Combining this upper bound with the lower bound already obtained above, we
derive a cutoff phenomenon in this particular matrix case.

\begin{corollary}[Cutoff for DOU in the matrix case]\label{cor:DOU12}
  Let ${(X^n_t)}_{t\geq0}$ be the DOU process \eqref{eq:DOU}, with
  $\beta\in\{0,1,2\}$, and invariant law $P_n^\beta$. Let
  $\mathrm{dist}\in\{\mathrm{TV},\mathrm{Hellinger},\mathrm{Kullback},\chi^2,\mathrm{Wasserstein}\}$.
  Let ${(a_n)}_n$ be a real sequence satisfying $\inf_n \sqrt{n} a_n > 0$, and
  assume further that $\lim_{n\to\infty}\sqrt{n} a_n=\infty$ if
  $\mathrm{dist}=\mathrm{Wasserstein}$. Then, for all $\varepsilon\in (0,1)$,
  we have
  \[
    \lim_{n\to\infty}%
    \sup_{x^n_0 \in [-a_n,a_n]^n}
    \mathrm{dist}(\mathrm{Law}(X^n_{t_n})\mid P_n^\beta)
    =\begin{cases}
      \max & \text{if $t_n=(1-\varepsilon)c_n$}\\
      0 & \text{if $t_n=(1+\varepsilon)c_n$}
    \end{cases}
  \]
  where
  \[
    c_n
    :=\begin{cases}
      \log(na_n)
      &\text{ if $\mathrm{dist} \in \{\mathrm{TV},\mathrm{Hellinger},\mathrm{Kullback},\chi^2\}$}\\
      \log(\sqrt{n} a_n)
      &\text{ if $\mathrm{dist} = \mathrm{Wasserstein}$}
    \end{cases}.
  \]
\end{corollary}

\noindent Theorem \ref{th:DOU12} and Corollary \ref{cor:DOU12} are proved in Section \ref{se:DOU12}.

\medskip

It is worth noting that $d=n+\beta\frac{n(n-1)}{2}$ in Theorem \ref{th:DOUOU}
is indeed an integer in the ``random matrix'' 
cases $\beta\in\{1,2\}$, and corresponds then exactly to the degree of freedom
of the Gaussian random matrix models GOE and GUE respectively. More precisely,
if we let $X^n_\infty\sim P_n^\beta$ then:
\begin{itemize}
\item If $\beta=1$ then $P_n^\beta$ is the law of the eigenvalues of
  $S\sim\mathrm{GOE}_n$, and $|X^n_\infty|^2=\sum_{j,k=1}^nS_{jk}^2$ which is
  the sum of $n$ squared Gaussians of variance $v=1/n$ (diagonal) plus twice the sum of
  $\frac{n^2-n}{2}$ squared Gaussians of variance $\frac{v}{2}$
  (off-diagonal) all being independent. The duplication has the effect of
  renormalizing the variance from $\frac{v}{2}$ to $v$. All in all we have the
  sum of $d=\frac{n^2+n}{2}$ independent squared Gaussians of same variance
  $v$. See Section \ref{se:DOU12}.
\item If $\beta=2$ then $P_n^\beta$ is the law of the eigenvalues of
  $H\sim\mathrm{GUE}_n$, and $|X^n_\infty|^2=\sum_{j,k=1}^n|H_{jk}|^2$ is the
  sum of $n$ squared Gaussians of variance $v=1/n$ (diagonal) plus twice the sum of $n^2-n$
  squared Gaussians of variance $\frac{v}{2}$ (off-diagonal) all being independent. All
  in all we have the sum of $d=n^2$ independent squared Gaussians of same
  variance $v$. See Section \ref{se:DOU12}.
\end{itemize}

\medskip

Another manifestation of exact solvability lies at the level of functional
inequalities. Indeed, and following \cite{zbMATH07238061}, the optimal
Poincaré constant of $P_n^\beta$ is given by $1/n$ and does not depend on
$\beta$, and the extremal functions are tranlations/dilations of $x\mapsto \pi(x)=x_1+\cdots+x_n$. This corresponds to a spectral gap
of the dynamics equal to $1$ and its associated eigenfunction. Moreover, the
optimal logarithmic Sobolev inequality of $P_n^\beta$ (Lemma \ref{le:lsi}) is
given by $2/n$ and does not depend on $\beta$, and the extremal functions are
of the form $x\mapsto\mathrm{e}^{c(x_1+\cdots+x_n)}$, $c\in\mathbb{R}$. This knowledge
of the optimal constants and extremal functions and their independence with
respect to $\beta$ is truly remarkable. It plays a crucial role in the results presented in this article. More precisely, the optimal Poincaré inequality is used for the lower
bound via the first eigenfunctions while the optimal logarithmic Sobolev
inequality is used for the upper bound via exponential decay of the entropy.

\subsection{Cutoff in the general interacting case}
\label{ss:cutoffgen}

Our main contribution consists in deriving an upper bound on the mixing times
in the general case $\beta \ge 1$: the proof relies on the logarithmic Sobolev
inequality, some coupling arguments and a regularization procedure.

\begin{theorem}[Upper bound on the mixing time: the general case]\label{th:DOUWTV}
  Let ${(X^n_t)}_{t\geq0}$ be the DOU process \eqref{eq:DOU}, with $\beta=0$
  or $\beta\ge 1$ and invariant law $P_n^\beta$.
  Take  $\mathrm{dist}\in\{\mathrm{TV},\mathrm{Hellinger},\mathrm{Wasserstein}\}$.
  Set
  \[
    c_n :=
    \begin{cases}
      \log(\sqrt{n} |x_0^n|) \vee \log({n})
      &\text{if $\mathrm{dist}\in\{ \mathrm{TV},\mathrm{Hellinger}\}$}\\
      \log(|x_0^n|) \vee \log(\sqrt{n})
      &\text{if $\mathrm{dist} = \mathrm{Wasserstein}$}
    \end{cases}.
  \]
  Then, for all $\varepsilon\in (0,1)$, we have
  \[
      \lim_{n\to\infty}\mathrm{dist}(\mathrm{Law}(X^n_{(1+\varepsilon)c_n})\mid
      P_n^\beta)
      =0.
  \]
\end{theorem}

Combining this upper bound with the general lower bound that we obtained in
Corollary \ref{cor:LowerBd}, we deduce the following cutoff phenomenon.
Observe that it holds both for $\beta =0$ and $\beta \geq 1$, and that the
expression of the mixing time does not depend on $\beta$.

\begin{corollary}[Cutoff for DOU in the general case]\label{cor:DOUWTV}
  Let ${(X^n_t)}_{t\geq0}$ be the DOU process \eqref{eq:DOU} with $\beta=0$ or
  $\beta\geq1$ and invariant law $P_n^\beta$. Take
  $\mathrm{dist}\in\{\mathrm{TV},\mathrm{Hellinger},\mathrm{Wasserstein}\}$.
  Let ${(a_n)}_n$ be a real sequence satisfying $\inf_na_n > 0$. Then, for all
  $\varepsilon \in (0,1)$, we have
  \[
    \lim_{n\to\infty}%
    \sup_{x^n_0 \in [-a_n,a_n]^n}
    \mathrm{dist}(\mathrm{Law}(X^n_{t_n})\mid P_n^\beta)
    =\begin{cases}
      \max & \text{if $t_n=(1-\varepsilon)c_n$}\\
      0 & \text{if $t_n=(1+\varepsilon)c_n$}
    \end{cases}
  \]
  where
  \[
    c_n
    := \begin{cases}
      \log(na_n)
      &\text{ if $\mathrm{dist} \in \{\mathrm{TV},\mathrm{Hellinger}\}$}\\
      \log(\sqrt{n} a_n)
      &\text{ if $\mathrm{dist} = \mathrm{Wasserstein}$}
    \end{cases}.
  \]
\end{corollary}

\noindent The proofs of Theorem \ref{th:DOUWTV} and Corollary \ref{cor:DOUWTV}
for the TV and Hellinger distances are presented in Section
\ref{se:DOUWTV:ub}. The Wasserstein distance is treated in Section
\ref{se:DOUW}. Let us make a comment on the assumptions made on $a_n$ in
Corollaries \ref{cor:DOU12} and \ref{cor:DOUWTV}. They are dictated by the
upper bounds established in Theorems \ref{th:DOU12} and \ref{th:DOUWTV}, which
take the form of maxima of two terms: one that depends on the initial
condition, and another one which is a power of a logarithm of $n$. The
logarithmic term is an upper bound on the time required to regularize a
pointwise initial condition, its precise expression varies according to the
method of proof we rely on: in the matrix case, it is the time required to
regularize a larger object, the matrix-valued OU process; in the general case,
it is related to the time it takes to make the entropy of a pointwise initial
condition small. These bounds are not optimal for $\beta=0$ (compare with
Theorem \ref{th:OU2}), and probably neither for $\beta \ge 1$.

\medskip

A natural, but probably quite difficult, goal would be to establish a cutoff
phenomenon in the situation where the set of initial conditions is reduced to
\emph{any} given singleton, as in Theorem \ref{th:OU2} for the case
$\beta = 0$. Recall that in that case, the asymptotic of the mixing time is
dictated by the Euclidean norm of the initial condition. In the case
$\beta \ge 1$, this \emph{cannot} be the right observable since the Euclidean
norm does not measure the distance to equilibrium. Instead one should probably
consider the Euclidean norm $|x_0^n-\rho_n|$, where $\rho_n$ is the vector of
the quantiles of order $1/n$ of the semi-circle law that arises in the
mean-field limit equilibrium (see Subsection \ref{se:meanfield}). More
precisely
\begin{equation}\label{eq:qn}
    \rho_{n,i} =\inf\left\{t\in \mathbb{R}:\int_{-\infty}^t   \frac{\sqrt{2\beta-x^2}}{\beta\pi}\mathbf{1}_{x\in[-\sqrt{2\beta},\sqrt{2\beta}]}\mathrm{d}x\geq \frac{i}{n} \right\}, \quad i\in \{1,\ldots,n\}.
\end{equation}
Note that $\rho_n = 0$ when $\beta = 0$.

\smallskip

A first step in this direction is given by the following result:

\begin{theorem}[DOU in the general case and pointwise initial condition]\label{th:DOUW}
  Let ${(X^n_t)}_{t\geq0}$ be the DOU process \eqref{eq:DOU} with $\beta=0$ or
  $\beta\geq1$, and invariant law $P_n^\beta$. There hold
  \begin{itemize}
  \item If $\lim_{n\to\infty}|x^n_0-\rho_n|=+\infty$, then, denoting
    $t_n = \log(|x_0^n-\rho_n|)$, for all $\varepsilon \in (0,1)$,
    \[
      \lim_{n\to\infty}
      \mathrm{Wasserstein}(\mathrm{Law}(X_{(1+\varepsilon)t_n}),P_n^\beta) = 0.
    \]
  \item If $\lim_{n\to\infty}|x^n_0-\rho_n|=\alpha\in [0,\infty)$, then, for all $t>0$,
    \[
      \varlimsup_{n\to\infty}
      \mathrm{Wasserstein}(\mathrm{Law}(X_{t}),P_n^\beta)^2
      \le
      \alpha^2\mathrm{e}^{-2t}.
    \]
  \end{itemize}
\end{theorem}

\noindent Theorem \ref{th:DOUW} is proved in Section \ref{se:DOUW}.

\subsection{Non-pointwise initial conditions}
\label{ss:nonpoint}

It is natural to ask about the cutoff phenomenon when the initial conditions
$X^n_0$ is not pointwise. Even if we turn off the interaction by taking
$\beta=0$, the law of the process at time $t$ is then no longer Gaussian in
general, which breaks the method of proof used for Theorem \ref{th:OU1} and
Theorem \ref{th:OU2}. Nevertheless, Theorem \ref{th:DOUK} below provides a
universal answer, that is both for $\beta=0$ and $\beta \ge 1$, at the price
however of introducing several objects and notations. More precisely, for any
probability measure $\mu$ on $\mathbb{R}^n$, we introduce
\begin{equation}\label{eq:S}
  S(\mu)
  =\begin{cases}
    \displaystyle\int\frac{\mathrm{d}\mu}{\mathrm{d}x}\log\frac{\mathrm{d}\mu}{\mathrm{d}x}\mathrm{d}x
    =\text{``}\mathrm{Kullback}(\mu\mid\mathrm{d}x)\text{''}
    &\text{if
      $\displaystyle\frac{\mathrm{d}\mu}{\mathrm{d}x}\log\frac{\mathrm{d}\mu}{\mathrm{d}x}\in
      L^1(\mathrm{d}x)$}\\
    +\infty&\text{otherwise}
  \end{cases}.
\end{equation}
Note that $S$ takes its values in the whole $(-\infty,+\infty]$, and when
$S(\mu)<+\infty$ then $-S(\mu)$ is the \emph{Boltzmann--Shannon entropy} of
the law $\mu$. For all $x\in\mathbb{R}^n$ with $x_i \ne x_j$ for all $i\ne j$,
we have
\begin{equation}\label{eq:Phi}
  E(x_1,\ldots,x_n)
  =n^2\iint\Phi(x,y) \mathbf{1}_{\{x\neq y\}} L_n(\mathrm{d}x)L_n(\mathrm{d}y)
\end{equation}
where $\displaystyle L_n:=\frac{1}{n}\sum_{i=1}^n\delta_{x_i}$ and where
$\displaystyle\Phi(x,y):=\frac{n}{n-1}\frac{V(x)+V(y)}{2}+\frac{\beta}{2}\log\frac{1}{|x-y|}$.

Let us define the map $\Psi:\mathbb{R}^n\mapsto \overline{D}_n$ by
\begin{equation}\label{eq:Psi}
  \Psi(x_1,\ldots,x_n):=(x_{\sigma(1)},\ldots,x_{\sigma(n)}).
\end{equation}
where $\sigma$ is any permutation of $\{1,\ldots,n\}$ that reorders the
particles non-decreasingly.

\begin{theorem}[Cutoff for DOU with product smooth initial
  conditions]\label{th:DOUK}
  Let ${(X^n_t)}_{t\geq0}$ be the DOU process \eqref{eq:DOU} with $\beta=0$ or
  $\beta\geq1$, and invariant law $P_n^\beta$. Let $S$, $\Phi$, and $\Psi$
  be as in \eqref{eq:S}, \eqref{eq:Phi}, and \eqref{eq:Psi}. Let us assume
  that $\mathrm{Law}(X_0^n)$ is the image law or push forward of a product law
  $\mu_1\otimes\cdots\otimes\mu_n$ by $\Psi$ where $\mu_1,\ldots,\mu_n$ are
  laws on $\mathbb{R}$. Then:
  \begin{enumerate}
  \item If 
    $\displaystyle\varliminf_{n\to \infty} \Bigr|\frac{1}{n}\sum_{i=1}^n \int x
    \mu_i(\mathrm{d}x)\Bigr| \neq 0$ then, for all $\varepsilon\in(0,1)$,
    \[
      \lim_{n\to \infty}\mathrm{Kullback}(\mathrm{Law}(X_{(1-\varepsilon)\log(n)})\mid P_n^\beta)=+\infty.
    \]
  \item If
    $\displaystyle\varlimsup_{n\to \infty}\frac{1}{n^2}\sum_{i=1}^n
    S(\mu_i)<\infty$ and
    $\displaystyle\varlimsup_{n\to \infty}
    \frac{1}{n^2}\sum_{i\ne j}\iint\Phi\,\mathrm{d}\mu_i \otimes \mathrm{d}\mu_j <\infty$,
    then, for all $\varepsilon\in(0,1)$,
    \[
      \lim_{n\to \infty}\mathrm{Kullback}(\mathrm{Law}(X_{(1+\varepsilon)\log(n)})\mid P_n^\beta)=0.
    \]
  \end{enumerate}
\end{theorem}

Theorem \ref{th:DOUK} is proved in Section \ref{se:DOUK}.

It is likely that Theorem \ref{th:DOUK} can be extended to the case
$\mathrm{dist}\in\{\mathrm{Wasserstein},\mathrm{Hellinger},\mathrm{Fisher}\}$.

\subsection{Structure of the paper}
\begin{itemize}
\item Section \ref{se:commopen} provides additional comments and open problems.
\item Section \ref{se:OU} focuses on the OU process ($\beta = 0$) and gives
  the proofs of Theorems \ref{th:OU1} and \ref{th:OU2}.
\item Section \ref{se:proofs:DOUOU:DOUCIR:moments} concerns the exact
  solvability of the DOU process for all $\beta$, and provides the proofs of Theorem \ref{th:DOUOU} and
  Corollary \ref{cor:LowerBd}.
\item Section \ref{se:DOU12} is about random matrices and gives the proofs of Theorem \ref{th:DOU12} and
  Corollary \ref{cor:DOU12}.
\item Section \ref{se:DOUWTV:ub} deals with the DOU process for all $\beta$
  with the TV and Hellinger distances, and provides the proofs of Theorem
  \ref{th:DOUWTV} and Corollary \ref{cor:DOUWTV}.
\item Section \ref{se:DOUW} gives the Wasserstein counterpart of Section
  \ref{se:DOUWTV:ub} and the proof of Theorem \ref{th:DOUW}.
\item Appendix \ref{ap:distances} provides a survey on distances and
  divergences, with new results.
\item Appendix \ref{ap:convexity} gathers useful dynamical consequences of
  convexity.
\end{itemize}

\section{Additional comments and open problems}
\label{se:commopen}

\subsection{About the results and proofs}

The proofs of our results rely among other ingredients on convexity and
optimal functional inequalities, exact solvability, exact Gaussian formulas,
coupling arguments, stochastic calculus, variational formulas, contraction
properties and regularization.

The proofs of Theorems \ref{th:OU1} and \ref{th:OU2} are based on the explicit
Gaussian nature of the OU process, which allows to use Gaussian formulas for
all the distances and divergences that we consider (the Gaussian formula for
$\mathrm{Fisher}$ seems to be new). Our analysis of the convergence to
equilibrium of the OU process seems to go beyond what is already known, see
for instance~\cite{MR2203823} and
\cite{10.1214/19-AAP1526,barrera-hogele-pardo-langevin,barrera-hogele-pardo-small-levy,MR4073676}.

Theorem \ref{th:DOUOU} is a one-dimensional analogue of
\cite[Th.~1.2]{MR3847984}. The proof exploits the explicit knowledge of
eigenfunctions of the dynamics \eqref{eq:eigenfunctions}, associated with the
first two non-zero spectral values, and their remarkable properties. The first
one is associated to the spectral gap and the optimal Poincaré inequality. It
implies Corollary \ref{cor:LowerBd}, which is the provider of all our lower
bounds on the mixing time for the cutoff.

The proof of Theorem \ref{th:DOU12} is based on a contraction property and the
upper bound for matrix OU processes. It is not available beyond the matrix
cases. All the other upper bounds that we establish are related to an optimal
exponential decay which comes from convexity and involves sometimes coupling,
the simplest instance being Theorem \ref{th:DOUWTV} about the Wasserstein
distance. The usage of the Wasserstein metrics for Dyson dynamics is quite
natural, see for instance \cite{bertucci-debbah-lasry-lions}.

The proof of Theorem \ref{th:DOUWTV} for the $\mathrm{TV}$ and
$\mathrm{Hellinger}$ relies on the knowledge of the optimal exponential decay
of the entropy (with respect to equilibrium) related to the optimal
logarithmic Sobolev inequality. Since pointwise initial conditions have
infinite entropy, the proof proceeds in three steps: first we regularize the
initial condition to make its entropy finite, second we use the optimal
exponential decay of the entropy of the process starting from this regularized
initial condition, third we control the distance between the processes
starting from the initial condition and its regularized version. This last
part is inspired by a work of Lacoin \cite{MR3474475} for the simple exclusion
process on the segment, subsequently adapted to continuous
state-spaces~\cite{CLL1,CLL2}, where one controls an \emph{area} between two
versions of the process.

The (optimal) exponential decay of the entropy (Lemma \ref{le:expdec}) is
equivalent to the (optimal) logarithmic Sobolev inequality (Lemma
\ref{le:lsi}). For the DOU process, the optimal logarithmic Sobolev inequality
provided by Lemma \ref{le:lsi} achieves also the universal bound with respect
to the spectral gap, just like for Gaussians. This sharpness between the best
logarithmic Sobolev constant and the spectral gap also holds for instance for
the random walk on the hypercube, a discrete process for which a cutoff
phenomenon can be established with the optimal logarithmic Sobolev inequality,
and which can be related to the OU process, see for instance
\cite{zbMATH04001076,DiaSal} and references therein. If we generalize the DOU
process by adding an arbitrary convex function to $V$, then we will still have
a logarithmic Sobolev inequality -- see \cite{zbMATH07238061} for several
proofs including the proof via the Bakry--Émery criterion -- however the
optimal logarithmic Sobolev constant will no longer be explicit nor sharp with
respect to the spectral gap, and the spectral gap will no longer be explicit.

The proof of Theorem \ref{th:DOUK} relies crucially on the tensorization
property of $\mathrm{Kullback}$ and on the asymptotics on the normalizing
constant $C_n^\beta$ at equilibrium.

\subsection{Analysis and geometry of the equilibrium}
\label{Subsec:Analysis}

The full space $\mathbb{R}^n$ is, up to a bunch of hyperplanes, covered with
$n!$ disjoint isometric copies of the convex domain $D_n$ obtained by
permuting the coordinates (simplices or Weyl chambers). Following
\cite{zbMATH07238061}, for all $\beta\geq0$ let us define the law
$P_{*n}^\beta$ on $\mathbb{R}^n$ with density proportional to
$\mathrm{e}^{-E}$, just like for $P_n^\beta$ in \eqref{eq:P} but without the
$\mathbf{1}_{(x_1,\ldots,x_n)\in\overline{D}_n}$.

If $\beta=0$ then $P_{*n}^0=P_n^0=\mathcal{N}(0,\frac{1}{n}I_n)$ according to
our definition of $P_n^0$.

If $\beta > 0$ then $P_{*n}^\beta$ has density
$(C_{*n}^\beta)^{-1}\mathrm{e}^{-E}$ with $C_{*n}^\beta=n!C_n^\beta$ where
$C_n^\beta$ is the normalization of $P_n^\beta$. Moreover $P_{*n}^\beta$ is a
mixture of $n!$ isometric copies of $P_n^\beta$, while $P_n^\beta$ is the
image law or push forward of $P_{*n}^\beta$ by the map
$\Psi_n:\mathbb{R}^n\to\overline{D}_n$ defined in \eqref{eq:Psi}. Furthermore
for all bounded measurable $f:\mathbb{R}^n\to\mathbb{R}$, denoting $\Sigma_n$
the symmetric group of permutations of $\{1,\ldots,n\}$,
\[
  \int f\mathrm{d}P_{*n}^\beta
  =\int f_{\mathrm{sym}}\mathrm{d} P_n^\beta
  \quad\text{with}\quad
  f_{\mathrm{sym}}(x_1,\ldots,x_n)
  :=\frac{1}{n!}\sum_{\sigma\in\Sigma_n}f(x_{\sigma(1)},\ldots,x_{\sigma(n)}).
\]

Regarding log-concavity, it is important to realize that if $\beta=0$ then $E$
is convex on $\mathbb{R}^n$, while if $\beta>0$ then $E$ is convex on $D_n$
but is not convex on $\mathbb{R}^n$ and has $n!$ isometric local minima.

\begin{itemize}
\item The law $P_{*n}^\beta$ is centered but is not log-concave when $\beta>0$
  since $E$ is not convex on $\mathbb{R}^n$.\\
  As $\beta\to0^+$ the law $P_{*n}^\beta$ tends to
  $P_{*n}^0=P_n^0=\mathcal{N}(0,\frac{1}{n}I_n)$ which is log-concave.
\item The law $P_n^\beta$ is not centered but is log-concave for all
  $\beta\geq0$.\\Its density vanishes at the boundary of $D_n$ if $\beta>0$.\\
  As $\beta\to0^+$ the law $P_n^\beta$ tends to the law of the order
  statistics of $n$ i.i.d.~$\mathcal{N}(0,\frac{1}{n})$.
\end{itemize}

\subsection{Spectral analysis of the generator: the non-interacting case}

This subsection and the next deal with analytical aspects of our dynamics. We start with the OU process ($\beta=0$) for which everything is explicit; the next subsection deals with the DOU process ($\beta \ge 1$).

The infinitesimal generator of the OU process is given by
\begin{equation}\label{eq:G0}
  \G f
  =\frac{1}{n}\Bigr(\Delta-\nabla E\cdot\nabla\Bigr)
  =\frac{1}{n}\sum_{i=1}^n\partial_i^2
  - \sum_{i=1}^nV'(x_i)\partial_i.
\end{equation}
It is a self-adjoint operator on $L^2(\mathbb{R}^n, P_n^0)$ that leaves
globally invariant the set of polynomials. Its spectrum is the set of all
non-positive integers, that is,
$\lambda_0 = 0 > \lambda_1 = - 1 > \lambda_2 = -2 > \cdots$. The corresponding
eigenspaces $F_0,F_1,F_2,\cdots$ are finite dimensional: $F_m$ is spanned by
the multivariate Hermite polynomials of degree $m$, in other words tensor
products of univariate Hermite polynomials. In particular, $F_0$ is the vector
space of constant functions while $F_1$ is the $n$-dimensional vector space of
all linear functions.

Let us point out that $\G$ can be restricted to the set of $P_n^0$ square
integrable \emph{symmetric} functions: it leaves globally invariant the set of
\emph{symmetric} polynomials, its spectrum is unchanged but the associated
eigenspaces $E_m$ are the restrictions of the vector spaces $F_m$ to the set
of symmetric functions, in other words, $E_m$ is spanned by the multivariate
\emph{symmetrized} Hermite polynomials of degree $m$. Note that $E_1$ is the
one-dimensional space generated by $\pi(x) =x_1+\cdots+x_n$.

The Markov semigroup ${(\mathrm{e}^{t\G})}_{t\geq0}$ generated by $\G$ admits
$P_n^0$ as a reversible invariant law since $\G$ is self-adjoint in
$L^2(P_n^0)$. Following \cite{saloff1994precise}, let us introduce the
\emph{heat kernel} $p_t(x,y)$ which is the density of
$\mathrm{Law}(X^n_t\mid X^n_0=x)$ with respect to the invariant law $P_n^0$.
The long-time behavior reads $\lim_{t\to\infty}p_t(x,\cdot)=1$ for all
$x\in \mathbb{R}^n$. Let $\left\|\cdot\right\|_p$ be the norm of
$L^p=L^p(P_n^0)$. For all $1\leq p\leq q$, $t\geq0$, $x\in \mathbb{R}^n$, we
have
\begin{equation}
  2\|\mathrm{Law}(X^n_t\mid X^n_0=x)-P_n^0\|_{\mathrm{TV}}
  =\|p_t(x,\cdot)-1\|_1
  \leq \|p_t(x,\cdot)-1\|_p
  \le\|p_t(x,\cdot)-1\|_q.
\end{equation}
In the particular case $p=2$ we can write
\begin{equation}\label{eq:l2expansion}
  \|p_t(x,\cdot)-1\|_2^2
  =\sum_{m=1}^\infty\mathrm{e}^{-2mt}\sum_{\psi\in B_m}|\psi(x)|^2.
\end{equation}
where $B_m$ is an orthonormal basis of $F_m\subset L^2(P_n^0)$, hence
\begin{equation}\label{eq:L2pt}
  \|p_t(x,\cdot)-1\|_2^2
  \geq \mathrm{e}^{-2t}\sum_{\psi\in B_1}|\psi(x)|^2,
\end{equation}
which leads to a lower bound on the $\chi^2$ (in other words $L^2$) cutoff,
provided one can estimate $\sum_{\psi\in B_1}|\psi(x)|^2$ which is the square
of the norm of the projection of $\delta_x$ on $B_1$.

Following \cite[Th.~6.2]{saloff1994precise}, an upper bound would follow from
a Bakry--Émery curvature--dimension criterion $\mathrm{CD}(\rho,d)$ with a
finite dimension $d$, in relation with Nash--Sobolev inequalities and
dimensional pointwise estimates on the heat kernel $p_t(x,\cdot)$ or
ultracontractivity of the Markov semigroup, see for instance
\cite[Sec.~4.1]{zbMATH01680886}. The OU process satisfies to
$\mathrm{CD}(\rho,\infty)$ but never to $\mathrm{CD}(\rho,d)$ with $d$ finite
and is not ultracontractive. Actually the OU process is a critical case, see
\cite[Ex.~2.7.3]{zbMATH01633816}.

\subsection{Spectral analysis of the generator: the interacting case}

We now assume that $\beta \ge 1$. The infinitesimal generator of the DOU process is the operator
\begin{equation}\label{eq:G}
  \G f
  =\frac{1}{n}\Bigr(\Delta-\nabla E\cdot\nabla\Bigr)
  =\frac{1}{n}\sum_{i=1}^n\partial_i^2
  - \sum_{i=1}^nV'(x_i)\partial_i
  +\frac{\beta}{2n}\sum_{j\neq i}\frac{\partial_i-\partial_j}{x_i-x_j}.
\end{equation}
Despite the interaction term, the operator leaves globally invariant the set
of \emph{symmetric} polynomials. Following Lassalle in
\cite{MR1133488,MR1456121}, see also \cite{zbMATH07238061}, the operator $\G$
is a self-adjoint operator on the space of $P_{*n}^\beta$ square integrable
\emph{symmetric} functions of $n$ variables, its spectrum does not depend on
$\beta$ and matches the spectrum of the OU process case $\beta=0$. In
particular the spectral gap is $1$. The eigenspaces $E_m$ are spanned by the
generalized symmetrized Hermite polynomials of degree $m$. For instance, $E_1$
is the one-dimensional space generated by $ \pi(x)=x_1+\cdots+x_n$ while $E_2$
is the two-dimensional space spanned by
\begin{equation}\label{eq:eigenfunctions}
  (x_1+\cdots+x_n)^2-1
  \quad\text{and}\quad
  x_1^2+\cdots+x_n^2-1-\frac{\beta}{2}(n-1).
\end{equation}

From the isometry between $L^2(\overline{D}_n,P_n^\beta)$ and
$L^2_{\mathrm{sym}}(\mathbb{R}^n,P_{*n}^\beta)$, the above explicit spectral
decomposition applies to the semigroup of the DOU on $\overline{D}_n$.
Formally, the discussion presented at the end of the previous subsection still
applies. However, in the present interacting case the integrability properties
of the heat kernel are not known: in particular, we do not know whether
$p_t(x,\cdot)$ lies in $L^p(P_n^\beta)$ for $t>0$, $x\in \overline{D}_n$ and
$p>1$. This leads to the question, of independent interest, of pointwise upper
and lower Gaussian bounds for heat kernels similar to the OU process, with
explicit dependence of the constants over the dimension. We refer for example
to \cite{zbMATH05122355,zbMATH05072820,zbMATH05643486} for some results in
this direction.

\subsection{Mean-field limit}\label{se:meanfield}

The measure $P_n^\beta$ is log-concave since $E$ is convex, and its density
writes
\begin{equation}\label{eq:HBE}
  x\in\mathbb{R}^n\mapsto
  \frac{\mathrm{e}^{-\frac{n}{2}|x|^2}}{C_n^\beta}\prod_{i>j}(x_i-x_j)^\beta\mathbf{1}_{x_1\leq\cdots\leq x_n}.
\end{equation}
See \cite[Sec.~2.2]{zbMATH07238061} for a high-dimensional analysis. The
Boltzmann--Gibbs measure $P_n^\beta$ is known as the $\beta$-Hermite ensemble
or H$\beta$E. When $\beta=2$, it is better known as the Gaussian Unitary
Ensemble (GUE). If $X^n\sim P_n^\beta$ then the Wigner theorem states that the
empirical measure with atoms distributed according to $P_n^\beta$ converges in
distribution to a semi-circle law, namely
\begin{equation}\label{eq:SC}
  \frac{1}{n}\sum_{i=1}^n\delta_{X^{n,i}}
  \underset{n\to\infty}{\overset{\text{weak}}{\longrightarrow}}
  \frac{\sqrt{2\beta-x^2}}{\beta\pi}\mathbf{1}_{x\in[-\sqrt{2\beta},\sqrt{2\beta}]}\mathrm{d}x,
\end{equation}
and this can be deduced in this Coulomb gas context from a large deviation
principle as in \cite{MR1465640}.

Let ${(X^n)}_{t\geq0}$ be the process solving \eqref{eq:DOU} with $\beta\geq0$
or $\beta\geq1$, and let
\begin{equation}\label{eq:munt}
  \mu^n_t=\frac{1}{n}\sum_{k=1}^n\delta_{X^{n,i}_t}
\end{equation}
be the empirical measure of the particles at time $t$. Following notably
\cite{rogers1993interacting,MR1851716,MR2209130,MR2053570,zbMATH07327443,zbMATH06996550},
if the sequence of initial conditions ${(\mu_0^n)}_{n\geq1}$ converges weakly
as $n\to\infty$ to a probability measure $\mu_0$, then the sequence of measure
valued processes ${({(\mu_t^n)}_{t\geq0})}_{n\geq1}$ converges weakly to the
unique probability measure valued deterministic process ${(\mu_t)}_{t\geq0}$
satisfying the evolution equation
\begin{equation}\label{eq:MKV}
  \langle\mu_t,f\rangle
  =\langle\mu_0,f\rangle
  -\int_0^t \int V'(x)f'(x)\mu_s(\mathrm{d}x) \mathrm{d}s
  +\frac{\beta}{2}
  \int_0^t\int_{\mathbb{R}^2}\frac{f'(x)-f'(y)}{x-y}\mu_s(\mathrm{d}x)\mu_s(\mathrm{d}y)\mathrm{d}s 
  \quad 
\end{equation}
for all $t\geq0$ and $f\in\mathcal{C}^3_b(\mathbb{R},\mathbb{R})$. The
equation \eqref{eq:MKV} is a weak formulation of a McKean--Vlasov equation or
free Fokker--Planck equation associated to a free OU process. Moreover, if
$\mu_0$ has all its moments finite, then for all $t\geq0$, we have the free
Mehler formula
\begin{equation}\label{eq:free}
  \mu_t = \mathrm{dil}_{\mathrm{e}^{-2t}}\mu_0\boxplus\mathrm{dil}_{\sqrt{1-\mathrm{e}^{-2t}}}\mu_\infty,
\end{equation}
where $\mathrm{dil}_\sigma\mu$ is the law of $\sigma X$ when $X\sim\mu$, where
``$\boxplus$'' stands for the free convolution of probability measures of
Voiculescu free probability theory, and where $\mu_\infty$ is the semi-circle
law of variance $\frac{\beta}{2}$. In particular, if $\mu_0$ is a semi-circle
law then $\mu_t$ is a semi-circle law for all $t\geq0$.

Let us introduce the $k$-th moment
$m_k(t):=\displaystyle\int x^k\mu_t(\mathrm{d}x)$ of $\mu_t$. The first and
second moments satisfy the differential equations $m_1'=-m_1$ and
$m_2'=-2m_2+\beta$ respectively, which give
\begin{equation}\label{eq:mom12}
  m_1(t)=\mathrm{e}^{-t}m_1(0)
  \underset{t\to\infty}{\longrightarrow}0
  \quad\text{and}\quad
  m_2(t)=m_2(0)\mathrm{e}^{-2t}+\frac{\beta}{2}(1-\mathrm{e}^{-2t})
  \underset{t\to\infty}{\longrightarrow}\frac{\beta}{2}.
\end{equation}
More generally, beyond the first two moments, the Cauchy--Stieltjes transform
\begin{equation}\label{eq:st}
  z\in\mathbb{C}_+ %
  =\{z\in\mathbb{C}:\Im z>0\}\mapsto s_t(z) %
  =\int_{\mathbb{R}}\frac{\mu_t(\mathrm{d}x)}{x-z}
\end{equation}
of $\mu_t$ is the solution of the following complex Burgers equation
\begin{equation}\label{eq:B}
  \partial_ts_t(z)=s_t(z)+z\partial_zs_t(z)+ \beta s_t(z)\partial_zs_t(z),\quad t\geq0, z\in\mathbb{C}_+.
\end{equation}
The semi-circle law on $[-c,c]$ has density
$\frac{2\sqrt{c^2-x^2}}{\pi c^2}\mathbf{1}_{x\in[-c,c]}$, mean $0$, second
moment or variance $\frac{c^2}{4}$, and Cauchy--Stieltjes transform
$s_t(z)=\frac{\sqrt{4z^2-4c^2}-2z}{c^2}$, $t\geq0 , z\in\mathbb{C}_+$.

The cutoff phenomenon is in a sense a diagonal $(t,n)$ estimate, melting long
time behavior and high dimension. When $|z_0^n|$ is of order $n$, cutoff
occurs at a time of order $\approx\log(n)$: this informally corresponds to
taking $t\to\infty$ in $(\mu_t)_{t\ge 0}$.

When $\mu_0$ is centered with same second moment $\frac{\beta}{2}$ as
$\mu_\infty$, then there is a Boltzmann H-theorem interpretation of the
limiting dynamics as $n\to\infty$: the steady-state is the Wigner semi-circle
law $\mu_\infty$, the second moment is conserved by the dynamics, the
Voiculescu entropy is monotonic along the dynamics, grows exponentially, and
is maximized by the steady-state.

\subsection{$L^p$ cutoff}
\label{ss:Lp}

Following \cite{MR2375599}, we can deduce an $L^p$ cutoff started from $x$
from an $L^1$ cutoff by showing that the heat kernel $p_t(x,\cdot)$ is in
$L^p(P_n^\beta)$ for some $t>0$. Thanks to the Mehler formula, it can be checked
that this holds for the OU case, despite the lack of
ultracontractivity. The heat kernel of the DOU process
is less accessible.

In another exactly solvable direction, the $L^p$ cutoff phenomenon has been
studied for instance in \cite{saloff1994precise,MR2111426} for Brownian motion
on compact simple Lie groups, and in \cite{MR2111426,meliot2014cut} for
Brownian motion on symmetric spaces, in relation with representation theory,
an idea which goes back at the origin to the works of Diaconis on random
walks on groups.

\subsection{Cutoff window and profile}

Once a cutoff phenomenon is established, one can ask for a finer description
of the pattern of convergence to equilibrium. The \emph{cutoff window} is the
order of magnitude of the transition time from the value $\max$ to the value
$0$: more precisely, if cutoff occurs at time $c_n$ then we say that the
cutoff window is $w_n$ if
\begin{align*}
	\varlimsup_{b\to +\infty} \varlimsup_{n\to\infty} \mathrm{dist}(\mathrm{Law}(X_{c_n + bw_n})\mid P_n^\beta) &=0,\\
	\varliminf_{b\to -\infty} \varliminf_{n\to\infty} \mathrm{dist}(\mathrm{Law}(X_{c_n + bw_n})\mid P_n^\beta) &=\max,
\end{align*}
and for any $b\in \mathbb{R}$
\[
  0 <  \varliminf_{n\to\infty} \mathrm{dist}(\mathrm{Law}(X_{c_n + bw_n})\mid
  P_n^\beta)
  \le  \varlimsup_{n\to\infty} \mathrm{dist}(\mathrm{Law}(X_{c_n + bw_n})\mid
  P_n^\beta) < \max.
\]
Note that necessarily $w_n = o(c_n)$ by definition of the cutoff phenomenon. Note also that $w_n$ is unique in the following sense: $w'_n$ is a cutoff window if and only if $w_n/w'_n$ remains bounded from above and below as $n\to\infty$.\\
We say that the \emph{cutoff profile} is given by
$\varphi:\mathbb{R}\to [0,1]$ if
\[
  \lim_{n\to\infty} \mathrm{dist}(\mathrm{Law}(X_{c_n + bw_n})\mid P_n^\beta)
  = \varphi(b).
\]
The analysis of the OU process carried out in Theorems \ref{th:OU1} and \ref{th:OU2} can be pushed further to establish the so-called cutoff profiles, we refer to the end of Section \ref{se:OU} for details.\\

Regarding the DOU process, such a detailed description of the convergence to
equilibrium does not seem easily accessible. However it is straightforward to
deduce from our proofs that the cutoff window is of order $1$, in other words the inverse of the spectral gap, in the setting of Corollary \ref{cor:DOU12}. This is also the case in the setting of Corollary \ref{cor:DOUWTV} for the Wasserstein distance.\\
We believe that this remains true in the setting of Corollary \ref{cor:DOUWTV}
for the TV and Hellinger distances: actually, a lower bound of the required
order can be derived from the calculations in the proof of Corollary
\ref{cor:LowerBd}; on the other hand, our proof of the upper bound on the
mixing time does not allow to give a precise enough upper bound on the window.

\subsection{Other potentials}\label{sec:genpot}

It is natural to ask about the cutoff phenomenon for the process solving
\eqref{eq:DOU} when $V$ is a more general $\mathcal{C}^2$ function. The
invariant law $P_n^\beta$ of this Markov diffusion writes
\begin{equation}\label{eq:DOUVPn}
  \frac{\mathrm{e}^{-n\sum_{i=1}^nV(x_i)}}{C_n^\beta}\prod_{i>j}(x_i-x_j)^\beta
  \mathbf{1}_{(x_1,\ldots,x_n)\in\overline{D}_n}\mathrm{d}x_1\cdots\mathrm{d}x_n.
\end{equation}
The case where $V-\frac{\rho}{2}\left|\cdot\right|^2$ is convex for some
constant $\rho\geq0$ generalizes the DOU case and has exponential convergence
to equilibrium, see \cite{zbMATH07238061}. Three exactly solvable cases
are known:
\begin{itemize}
\item $\mathrm{e}^{-V(x)}=\mathrm{e}^{-\frac{x^2}{2}}$: the DOU
  process associated to the Gaussian law weight and the $\beta$-Hermite
  ensemble including HOE/HUE/HSE when $\beta\in\{1,2,4\}$,
\item $\mathrm{e}^{-V(x)}=x^{a-1}\mathrm{e}^{-x}\mathbf{1}_{x\in[0,\infty)}$:
  the Dyson--Laguerre process associated to the Gamma law weight and the
  $\beta$-Laguerre ensembles including LOE/LUE/LSE when $\beta\in\{1,2,4\}$,
\item $\mathrm{e}^{-V(x)}=x^{a-1}(1-x)^{b-1}\mathbf{1}_{x\in[0,1]}$: the
  Dyson--Jacobi process associated to the Beta law weight and the
  $\beta$-Jacobi ensembles including JOE/JUE/JSE when $\beta\in\{1,2,4\}$,
\end{itemize}
up to a scaling. Following Lassalle
\cite{MR1133488,MR1105634,MR1096625,MR1456121} and Bakry \cite{MR1417973}, in
these three cases, the multivariate orthogonal polynomials of the invariant
law $P_n^\beta$ are the eigenfunctions of the dynamics of the process. We
refer to \cite{zbMATH05136242,zbMATH02122403,zbMATH02122724} for more
information on (H/L/J)$\beta$E random matrix models.

The contraction property or spectral projection used to pass from a matrix
process to the Dyson process can be used to pass from BM on the unitary group
to the Dyson circular process for which the invariant law is the Circular
Unitary Ensemble (CUE). This provides an upper bound for the cutoff
phenomenon. The cutoff for BM on the unitary group is known and holds at
critical time or order $\log(n)$, see for instance
\cite{MR2111426,saloff1994precise,meliot2014cut}.

More generally, we could ask about the cutoff phenomenon for a McKean--Vlasov
type interacting particle system ${(X^n_t)}_{t\geq0}$ in $(\mathbb{R}^d)^n$
solution of the stochastic differential equation of the form
\begin{equation}
  \mathrm{d}X^{n,i}_t
  =\sigma_{n,t}(X^n)\mathrm{d}B^n_t
  -\sum_{i=1}^n\nabla V_{n,t}(X^{n,i}_t)\mathrm{d}t
  -\sum_{j\neq i}\nabla W_{n,t}(X^{n,i}_t-X^{n,j}_t)\mathrm{d}t,\quad 1\leq i\leq n,
\end{equation}
for various types of confinement $V$ and interaction $W$ (convex, repulsive,
attractive, repulsive-attractive, etc), and discuss the relation with the
propagation of chaos. The case where $V$ and $W$ are both convex and constant
in time is already very well studied from the point of view of long-time
behavior and mean-field limit in relation with convexity, see for instance
\cite{MR2209130,MR2053570,zbMATH07327443}.

Regarding universality, it is worth noting that if $V=\left|\cdot\right|^2$
and if $W$ is convex then the proof by factorization of the optimal Poincaré
and logarithmic Sobolev inequalities and their extremal functions given in
\cite{zbMATH07238061} remains valid, paving the way to the generalization of
many of our results in this spirit. On the other hand, the convexity of the
limiting energy functional in the mean-field limit is of Bochner type and
suggests to take for $W$ a power, in other words a Riesz type interaction.

\subsection{Alternative parametrization}\label{se:altparam}

If ${(X^n_t)}_{t\geq0}$ is the process solution of the stochastic differential
equation \eqref{eq:DOU}, then for all real parameters $\alpha>0$ and
$\sigma>0$, the space scaled and time changed stochastic process
${(Y^n_t)}_{t\geq0}={(\sigma X^n_{\alpha t})}_{t\geq0}$ solves the stochastic
differential equation
\begin{equation}\label{eq:DOUalt}
  Y^n_0=\sigma x^n_0,\quad
  \mathrm{d}Y_t^{n,i}
  =\sqrt{\frac{2\alpha\sigma^2}{n}}\mathrm{d}B^i_t
  -\alpha Y_t^{n,i}\mathrm{d}t
  +\frac{\alpha\beta\sigma^2}{n}\sum_{j\neq i}\frac{\mathrm{d}t}{Y_t^{n,i}-Y_t^{n,j}},\quad 1\leq
  i\leq n,
\end{equation}
where ${(B_t)}_{t\geq0}$ is a standard $n$-dimensional BM. The
invariant law of 
${(Y^n_t)}_{t\geq0}$ is
\begin{equation}\label{eq:Palt}
  \frac{\mathrm{e}^{-\frac{n}{2\sigma^2}|y|^2}}{C_n^\beta}
  \prod_{i>j}(y_i-y_j)^\beta\mathbf{1}_{(y_1,\ldots,y_n)\in\overline{D}_n}
  \mathrm{d}y_1\cdots\mathrm{d}y_n
\end{equation}
where $C_n^\beta$ is the normalizing constant. This law and its normalization $C_n^\beta$
depend on the ``shape parameter'' $\beta$, the ``scale parameter'' $\sigma$,
and does not depend on the ``speed parameter'' $\alpha$. When $\beta>0$,
taking $\sigma^2=\beta^{-1}$, the stochastic differential equation
\eqref{eq:DOUalt} boils down to
\begin{equation}\label{eq:DOUaltspecial}
  Y^n_0=\frac{x^n_0}{\sqrt{\beta}},\quad
  \mathrm{d}Y_t^{n,i}
  =\sqrt{\frac{2\alpha}{n\beta}}\mathrm{d}B^i_t
  -\alpha Y_t^{n,i}\mathrm{d}t
  +\frac{\alpha}{n}\sum_{j\neq i}\frac{\mathrm{d}t}{Y_t^{n,i}-Y_t^{n,j}},\quad 1\leq
  i\leq n
\end{equation}
while the invariant law becomes 
\begin{equation}\label{eq:Paltspecial}
  \frac{\mathrm{e}^{-\frac{n\beta}{2}|y|^2}}{C_n^\beta}
  \prod_{i>j}(y_i-y_j)^\beta\mathbf{1}_{(y_1,\cdots,y_n)\in\overline{D}_n}
  \mathrm{d}y_1\cdots\mathrm{d}y_n.
\end{equation}
The equation \eqref{eq:DOUaltspecial} is the one considered in
\cite[Eq.~(12.4)]{MR3699468} and in \cite[Eq.~(1.1)]{zbMATH07109860}. The
advantage of \eqref{eq:DOUaltspecial} is that $\beta$ can be now truly
interpreted as an inverse temperature and the right-hand side in the analogue
of \eqref{eq:SC} does not depend on $\beta$, while the drawback is that we
cannot turn off the interaction by setting $\beta=0$ and recover the
OU process as in \eqref{eq:DOU}. It is worthwhile mentioning
that for instance Theorem \ref{th:DOUWTV} remains the same for the process
solving \eqref{eq:DOUaltspecial} in particular the cutoff threshold is at
critical time $\frac{c_n}{\alpha}$ and does not depend on $\beta$.

\subsection{Discrete models}

There are several discrete space Markov processes admitting the OU process as
a scaling limit, such as for instance the random walk on the discrete
hypercube, related to the Ehrenfest model, for which the cutoff has been
studied in \cite{zbMATH04001076,DiaSal}, and the M/M/$\infty$ queuing process,
for which a discrete Mehler formula is available \cite{zbMATH05216861}.
Certain discrete space Markov processes incorporate a singular repulsion
mechanism, such as for instance the exclusion process on the segment, for
which the study of the cutoff in \cite{MR3474475} shares similarities with our
proof of Theorem \ref{th:DOUWTV}. It is worthwhile noting that there are
discrete Coulomb gases, related to orthogonal polynomials for discrete
measures, suggesting to study discrete Dyson processes. More generally, it
could be natural to study the cutoff phenomenon for Markov processes on
infinite discrete state spaces, under curvature condition, even if the subject
is notoriously disappointing in terms of high-dimensional analysis. We refer
to the recent work \cite{salez2021cutoff} for the finite state space case.

\section{Cutoff phenomenon for the OU}
\label{se:OU}

In this section, we prove Theorems \ref{th:OU1} and \ref{th:OU2}: actually we only prove the latter since it implies the former. We start by recalling a
well-known fact.

\begin{lemma}[Mehler formula]\label{le:Mehler}
  If ${(Y_t)}_{t\geq0}$ is an OU process in $\mathbb{R}^d$
  solution of the stochastic differential equation $Y_0=y_0\in\mathbb{R}^d$
  and $\mathrm{d}Y_t=\sigma\mathrm{d}B_t-\mu Y_t\mathrm{d}t$ for parameters
  $\sigma>0$ and $\mu>0$ where $B$ is a standard $d$-dimensional Brownian
  motion then
  \[
    {(Y_t)}_{t\geq0}
    =
    {\Bigr(y_0 \mathrm{e}^{-\mu
        t}+\sigma\int_0^t\mathrm{e}^{\mu(s-t)}\mathrm{d}B_s\Bigr)}_{t\geq0}
    \text{ hence }
    Y_t
    \sim\mathcal{N}\Bigr(
    y_0\mathrm{e}^{-\mu t},
    \frac{\sigma^2}{2}\frac{1-\mathrm{e}^{-2\mu t}}{\mu}\mathrm{I}_d
    \Bigr)
    \text{ for all $t\geq0$.}
  \]
  Moreover its coordinates are independent one-dimensional OU processes with
  initial condition $y_0^i$ and invariant law
  $\mathcal{N}(0,\frac{\sigma^2}{2\mu})$, $1\leq i\leq d$.
\end{lemma}

\begin{proof}[Proof of Theorem \ref{th:OU1} and Theorem \ref{th:OU2}] 
  By using Lemma \ref{le:Mehler}, for all $n\geq1$ and $t\geq0$,
  \begin{equation}\label{eq:Mehler}
    Z^n_t
    \sim\mathcal{N}\Bigr(z^n_0\mathrm{e}^{-t},\frac{1-\mathrm{e}^{-2t}}{n}I_n\Bigr)
    =\otimes_{i=1}^n\mathcal{N}\Bigr(z^{n,i}_0\mathrm{e}^{-t},\frac{1-\mathrm{e}^{-2t}}{n}\Bigr),
    \ %
    P_n^0=\mathcal{N}\Bigr(0,\frac{I_n}{n}\Bigr)
    =\mathcal{N}\Bigr(0,\frac{1}{n}\Bigr)^{\otimes n}.
  \end{equation}
  
  \emph{Hellinger, Kullback, $\chi^2$, Fisher, and Wasserstein cutoffs.} A
  direct computation from \eqref{eq:Mehler} or Lemma \ref{le:gauss} either
  from multivariate Gaussian formulas or univariate via tensorization gives
  \begin{align}
    \mathrm{Hellinger}^2(\mathrm{Law}(Z^n_t),P_n^0)
    &=1-\exp\Bigr(-\frac{n}{4}\frac{|z^n_0|^2\mathrm{e}^{-2t}}{2-\mathrm{e}^{-2t}}+\frac{n}{4}\log\Bigr(4\frac{1-\mathrm{e}^{-2t}}{(2-\mathrm{e}^{-2t})^2}\Bigr)\Bigr),\label{eq:OUH}\\
    2\mathrm{Kullback}(\mathrm{Law}(Z^n_t)\mid P_n^0)
    &=n|z^n_0|^2\mathrm{e}^{-2t}-n\mathrm{e}^{-2t}-n\log(1-\mathrm{e}^{-2t}),\label{eq:OUK}\\    
    \chi^2(\mathrm{Law}(Z^n_t)\mid P_n^0)
    &=-1+\frac1{(1-\mathrm{e}^{-4t})^{n/2}} 
    \exp\Bigr( n|z_0^n|^2 \frac{\mathrm{e}^{-2t}}{1+\mathrm{e}^{-2t}} \Bigr),\label{eq:OUL2}\\
    \mathrm{Fisher}(\mathrm{Law}(Z^n_t)\mid P_n^0)
    &=n^2|z^n_0|^2\mathrm{e}^{-2t}+n^2\frac{\mathrm{e}^{-4t}}{1-\mathrm{e}^{-2t}},\label{eq:OUF}\\
    \mathrm{Wasserstein}^2(\mathrm{Law}(Z^n_t),P_n^0)
    &=|z^n_0|^2\mathrm{e}^{-2t}+2(1-\sqrt{1-\mathrm{e}^{-2t}}-\frac12 \mathrm{e}^{-2t}),\label{eq:OUW}
  \end{align}
  which gives the desired lower and upper bounds as before by using the
  hypothesis on $z^n_0$.

  \emph{Total variation cutoff.} By using the comparison between total
  variation and Hellinger distances (Lemma \ref{le:distineqs}) we deduce from
  \eqref{eq:OUH} the cutoff in total variation distance at the same critical
  time. The upper bound for the total variation distance can alternatively be
  obtained by using the $\mathrm{Kullback}$ estimate \eqref{eq:OUK} and the
  Pinsker--Csiszár--Kullback inequality (Lemma \ref{le:distineqs}). Since both
  distributions are tensor products, we could use alternatively the
  tensorization property of the total variation distance (Lemma
  \ref{le:tenso}) together with the one-dimensional version of the Gaussian
  formula for $\mathrm{Kullback}$ (Lemma \ref{le:distineqs}) to obtain the
  result for the total variation.
\end{proof}

\begin{remark}[Competition between bias and variance mixing]\label{remark:competition}
  From the computations of the proof of Theorem \ref{th:OU2}, we can show that
  for $\mathrm{dist} \in \{\mathrm{TV}, \mathrm{Hellinger}, \chi^2\}$
  \[
    A_t :=
    \mathrm{dist}(\mathrm{Law}(Z^n_t)\mid\mathrm{Law}(Z^n_t-z^n_0\mathrm{e}^{-t}))
  \]
  has a cutoff at time $c_n^{A} = \log(\sqrt{n} |z^n_0|)$, while
  \[
    B_t := \mathrm{dist}(\mathrm{Law}(Z^n_t-z^n_0\mathrm{e}^{-t})\mid P_n^0)
  \]
  admits a cutoff at time $c_n^{B} = \frac14 \log(n)$. The triangle inequality
  for $\mathrm{dist}$ yields
  \[
    | A_t - B_t | \le \mathrm{dist}(\mathrm{Law}(Z^n_t)\mid P_n^0)\le  A_t + B_t.
  \]
  Therefore the critical time of Theorem \ref{th:OU2} is dictated by either
  $A_t$ or $B_t$, according to whether $c_n^A \gg c_n^B$ or $c_n^A \ll c_n^B$.
  This can be seen as a competition between bias and variance mixing.
\end{remark}

\begin{remark}[Total variation discriminating event for small initial
  conditions]\label{remark:small initial}
  Let us introduce the random variable
  $Z^n_\infty\sim
  P_n^0=\mathcal{N}(0,\frac{1}{n}I_n)=\mathcal{N}(0,\frac{1}{n})^{\otimes n}$,
  in accordance with \eqref{eq:Mehler}. There holds
  \[
    S_t^n:=\sum_{i=1}^n (Z_t^{n,i} -z_0^{n,i}\mathrm{e}^{-t})^2
    \sim \mathrm{Gamma}\Bigr(\frac{n}{2},\frac{n}{2(1-\mathrm{e}^{-2t})}\Bigr)
    \quad\text{and}\quad
    |Z^n_\infty|^2\sim\mathrm{Gamma}\Bigr(\frac{n}{2},\frac{n}{2}\Bigr).
  \]
  We can check, using an explicit computation of Hellinger and Kullback
  between Gamma distributions and the comparison between total variation and
  Hellinger distances (Lemma \ref{le:distineqs}), that
  \[
    C_t := \mathrm{dist}(\mathrm{Law}(S^n_t)\mid \mathrm{Law}(|Z^n_\infty|^2))
  \]
  admits a cutoff at time $c_n^C = c_n^B = \frac14 \log(n)$.
  Moreover, one can exhibit a discriminating
  event for the TV distance. Namely, we can observe
  that
  \[
    \Bigr\Vert
    \mathrm{Gamma}
    \Bigr(\frac{n}{2},\frac{n}{2(1-\mathrm{e}^{-2t})}\Bigr)
    -\mathrm{Gamma}\Bigr(\frac{n}{2},\frac{n}{2}\Bigr)
    \Bigr\Vert_{\mathrm{TV}} %
    =\mathbb{P}( |Z^n_\infty|^2 \geq \alpha_{t})-\mathbb{P}( S_t^n \geq \alpha_{t})
  \]
  with $\alpha_t$ the unique point where the two densities meet, which happens to be
  \[
    \alpha_{t}=-\mathrm{e}^{2t}\log(1-\mathrm{e}^{-2t})(1-\mathrm{e}^{-2t}).
  \]
\end{remark}

From the explicit expressions \eqref{eq:OUH}, \eqref{eq:OUK}, \eqref{eq:OUL2},
\eqref{eq:OUF}, \eqref{eq:OUW}, it is immediate to extract the cutoff profile
associated to the convergence of $\mathrm{Law}(Z_t^n)$ to $P_n^0$ in
Hellinger, Kullback, $\chi^2$, Fisher and Wasserstein. For Wasserstein we
already know by Theorem \ref{th:OU2} that a cutoff occurs if and only if
$|z_0^n|\underset{n\to\infty}{\to} \infty$. In this case, regarding the
profile, we have
\begin{equation}
  \lim_{n\to\infty}\mathrm{Wasserstein}(\mathrm{Law}(Z_t^n),P_n^0)=\phi(b),
\end{equation}
where for all $b\in \mathbb{R}$,
\begin{equation}
    t_{n,b}=\log |z_0^n| + b\quad \text{and}\quad \phi(b)=\mathrm{e}^{-b}.
\end{equation}
For the other distances and divergences, let us assume that the following limit exists
\begin{equation}\label{eq:nzn2}
	a:= \lim_{n\to\infty}\sqrt{n}|z_0^n|^2\in[0,+\infty].
\end{equation}
This quantity can be related with 
  \begin{equation}\label{eq:cnAcnB}
	  c_n^A:=\log(|z_0^n|\sqrt{n})
    \quad\text{and}\quad
    c_n^B:=\frac{\log n}{4}
  \end{equation}
which were already introduced in Remark \ref{remark:competition}. Indeed
\[ a= 0 \Longleftrightarrow c_n^A \ll c_n^B,\quad a= +\infty
  \Longleftrightarrow c_n^A \gg c_n^B,\]
while $a\in (0,\infty)$ is equivalent to $c_n^A \asymp c_n^B$.\\
Then, for
$\mathrm{dist}\in \{\mathrm{Hellinger},\mathrm{Kullback},\chi^2,\mathrm{Fisher}
\}$, we have, for all $b\in\mathbb{R}$,
\begin{equation}\label{eq:OU:profile}
  \lim_{n\to\infty}\mathrm{dist}(\mathrm{Law}(Z_{t_{n,b}})\mid P_n^0)=\phi(b),
\end{equation}
where $t_{n,b}$ and $\phi(b)$ are as in Table \ref{tb:OU:profile}. The
cutoff window is always of size $1$.

\begin{table}[h]
  \centering
  \begin{tabular}{r||l|l|l}
    & $a=+\infty$
    & $a=0$
    & $a\in(0,+\infty)$\\[1em]
    $t_{n,b}$ &&\\[.5em]        
    Hellinger
    & $\log(|z_0^n|\sqrt{n} )+b$
    & $\frac{\log n}{4}+b$
    &$\frac{\log n}{4}+b$\\[.5em]
    Kullback
    & $\log(|z_0^n|\sqrt{n} )+b$
    & $\frac{\log n}{4}+b$
    &$\frac{\log n}{4}+b$\\[.5em]
    $\chi^2$
    & $\log(|z_0^n|\sqrt{n} )+b$
    & $\frac{\log n}{4}+b$
    & $\frac{\log n}{4}+b$
    \\[.5em]
    Fisher
    & $\log(|z_0^n|n )+b$ 
    & $\frac{\log n}{2}+b$
    & $\frac{\log n}{2}+b$\\[1em]
    $\phi(b)$ & & & \\[.5em]
    Hellinger
    & $\sqrt{1-\mathrm{e}^{-\frac{1}{8}\mathrm{e}^{-2b}}}$
    & $\sqrt{1-\mathrm{e}^{-\frac{1}{16}\mathrm{e}^{-4b}}}$
    & $\sqrt{1-\mathrm{e}^{-\frac{1}{8} {a\mathrm{e}^{-2b}} -\frac{1}{16}\mathrm{e}^{-4b} }}$\\[.5em]
    Kullback
    & $\frac{1}{2}\mathrm{e}^{-2b}$
    & $\frac{1}{4}\mathrm{e}^{-4b}$
    & $\frac{1}{2} a \mathrm{e}^{-2b}+\frac{1}{4}\mathrm{e}^{-4b}$\\[.5em]
    $\chi^2$
    & $\mathrm{e}^{\mathrm{e}^{-2b} }-1$
    & $\mathrm{e}^{\frac{1}{2}\mathrm{e}^{-4b} }-1$
    & $\mathrm{e}^{a\mathrm{e}^{-2b}+\frac{1}{2}\mathrm{e}^{-4b} }-1$\\[.5em]
    Fisher
    & $\mathrm{e}^{-2b}$
    & $\mathrm{e}^{-4b}$
    & $a\mathrm{e}^{-2b}+\mathrm{e}^{-4b}$
  \end{tabular}
  \bigskip
  \caption{\label{tb:OU:profile}Values of $t_{n,b}$ and $\phi(b)$ for
    the cutoff profile of the OU process in \eqref{eq:OU:profile}.}
\end{table}

Since the total variation distance is not expressed in a simple explicit
manner, further computations are needed to extract the precise cutoff profile,
which is given in the following lemma:
  
\begin{lemma}[Cutoff profile in $\mathrm{TV}$ for OU]\label{lemma:cutoff profile}
  Let $Z^n=(Z_t^n)_{t\geq 0}$ be the OU process \eqref{eq:OU}, started from
  $z_0^n\in \mathbb{R}^n$, and let $P_n^0$ be its invariant law. Assume as
  in \eqref{eq:nzn2} that
  $a:=\lim_{n\to\infty}|z_0^n|^2\sqrt{n}\in[0,+\infty]$, and let $t_{n,b}$
  be as in Table \eqref{tb:OU:profile} for Hellinger. Then, for all
  $b\in \mathbb{R}$, we have
  \begin{equation*}
    \lim_{n\to\infty}\Vert\mathrm{Law}({Z}_{t_{n,b}}^n)-P_n^0\Vert_{\mathrm{TV}}
    =\phi(b),
  \end{equation*}
  where
    \begin{equation*}
      \phi(b):=
      \begin{cases}
        \displaystyle\mathrm{erf}\Bigr(\frac{\mathrm{e}^{-b}}{2\sqrt{2}}\Bigr)
        & \text{if $a=+\infty$} \\[1em]
        \displaystyle\mathrm{erf}\Bigr(\frac{\mathrm{e}^{-2b}}{4}\Bigr)
        & \text{if $a=0$} \\[1em]
        \displaystyle\mathrm{erf}\Bigr(\frac{\sqrt{2a\mathrm{e}^{-2b}+\mathrm{e}^{-4b}}}{4}\Bigr)
        & \text{if $a\in(0,+\infty)$}
      \end{cases},
    \end{equation*}
    where $\displaystyle\mathrm{erf}(u):=\frac{1}{\sqrt{\pi}}\int_{|t|\leq
      u}\mathrm{e}^{-t^2}\mathrm{d}t=\mathbb{P}(|X|\leq\sqrt{2}u)$ with
    $X\sim\mathcal{N}(0,1)$ is the \emph{error function}.
  \end{lemma}
  
  \begin{proof}[Proof of Lemma \ref{lemma:cutoff profile}]
    The idea is to exploit the fact that we consider Gaussian product measures
    (the covariance matrices are multiple of the identity), which allows a
    finer analysis than for instance in \cite[Le.~3.1]{devroye}. We begin
    with a rather general step. Let $\mu$ and $\nu$ be two probability
    measures on $\mathbb{R}^n$ with densities $f$ and $g$ with respect to the
    Lebesgue measure $\mathrm{d}x$. We have then
  \begin{equation}
    \Vert \mu-\nu\Vert_{\mathrm{TV} }
    =\frac{1}{2}\int |f-g|\mathrm{d}x
    =\frac{1}{2}\Bigr(\int (f-g)\mathbf{1}_{g\leq f}\mathrm{d}x
    -\int(f-g)\mathbf{1}_{f\leq g}\mathrm{d}x\Bigr),
  \end{equation}
  and since
  \begin{equation*}
    -\int(f-g)\mathbf{1}_{f\leq g}\mathrm{d}x
    =-\int(f-g)(1-\mathbf{1}_{g<f})\mathrm{d}x
    =\int(f-g)\mathbf{1}_{g<f}\mathrm{d}x
    =\int(f-g)\mathbf{1}_{g\leq f}\mathrm{d}x,
  \end{equation*}
  we obtain
  \begin{equation}\label{eq:expression TV}
    \Vert \mu-\nu\Vert_{\mathrm{TV} }
    =\int (f-g)\mathbf{1}_{g\leq f}\mathrm{d}x   
    =\mu(g\leq f)-\nu(g\leq f).
  \end{equation}
  In particular, when $\mu=\mathcal{N}(m_1,\sigma_1^2I_n)$ and
  $\nu=\mathcal{N}(m_2,\sigma_2^2I_n)$ then $g(x)\leq f(x)$ is equivalent to
  \begin{equation}
    \psi(x):=\frac{|x-m_1|^2}{\sigma_1^2}-\frac{|x-m_2|^2}{\sigma_2^2}
    \leq n\log\frac{\sigma_2^2}{\sigma_1^2},
  \end{equation}
  for all $x\in\mathbb{R}^n$, and therefore, with $Z_1\sim\mu$ and
  $Z_2\sim\nu$, we get
  \begin{equation}\label{eq:TVpsi}
    \Vert \mu-\nu\Vert_{\mathrm{TV} }
    =\mathbb{P}\Bigr(\psi(Z_1)\leq n\log\frac{\sigma_2^2}{\sigma_1^2}\Bigr)
    -\mathbb{P}\Bigr(\psi(Z_2)\leq n\log\frac{\sigma_2^2}{\sigma_1^2}\Bigr).
  \end{equation}
  Let us assume from now on that $m_2=0$ and $\sigma_1\neq\sigma_2$. We can
  then gather the quadratic terms as
  \begin{equation}
    \psi(x)=\Bigr(1-\frac{\sigma_1^2}{\sigma_2^2}\Bigr)
    \frac{|x-\tilde{m}_1|^2}{\sigma_1^2}
    +\Bigr(\frac{1}{\sigma_1^2}-\frac{1}{(1-\frac{\sigma_1^2}{\sigma_2^2})\sigma_1^2}\Bigr)|m_1|^2
    \quad\text{where}\quad
    \tilde{m}_1:=\frac{1}{1-\frac{\sigma_1^2}{\sigma_2^2}}m_1.
  \end{equation}
  We observe at this step that the random variable $\frac{|Z_1-\tilde{m}_1|^2}{\sigma_1^2}$ follows a
  noncentral chi-squared distribution, which depends only on $n$ and on the
  noncentrality parameter
  \begin{equation}
    \lambda_1:=\frac{|m_1-\tilde{m}_1|^2}{\sigma_1^2}
    =\frac{\sigma_1^2}{(\sigma_2^2-\sigma_1^2)^2}|m_1|^2.
  \end{equation}
  Similarly, the random variable $\frac{|Z_2-\tilde{m}_1|^2}{\sigma_2^2}$ follows a
  noncentral chi-squared distribution, which depends only on $n$ and on the
  noncentrality parameter
  \begin{equation}
    \lambda_2:=\frac{|\tilde{m}_1|^2}{\sigma_1^2}
    =\frac{\sigma_2^4}{\sigma_1^2(\sigma_2^2-\sigma_1^2)^2}|m_1|^2.
  \end{equation}
  It follows that the law of $\psi(Z_1)$ and the law of $\psi(Z_2)$ depend
  over $m_1$ only via $|m_1|$. Hence
  \begin{equation}
    \psi(Z_1)\overset{\mathrm{d}}{=}X_1+\cdots+X_n
    \quad\text{and}\quad
    \psi(Z_2)\overset{\mathrm{d}}{=}Y_1+\cdots+Y_n
  \end{equation}
  where $X_1,\ldots,X_n$ and $Y_1,\ldots,Y_n$ are two sequences of i.i.d.\
  random variables for which the mean and variance depends only (and
  explicitly) on $|m_1|$, $\sigma_1$, $\sigma_2$. Note in particular that
  these means and variances are given by $\frac{1}{n}$ the ones of $\psi(Z_1)$
  and $\psi(Z_2)$. Now we specialize to the case where
  $\mu=\mathrm{Law}(Z^n_t)=\mathcal{N}(z_0^n\mathrm{e}^{-t},\frac{1-\mathrm{e}^{-2t}}{n}I_n)$
  and $\nu=\mathrm{Law}(Z^n_\infty)=\mathcal{N}(0,\frac{1}{n}I_n)=P_n^0$, and
  we find
      \begin{equation*}
  	\mathbb{E}[\psi(Z_1)]
  	=n\Bigr(1-\frac{\sigma_t^2}{\sigma_{\infty}^2}\Bigr)-\frac{|z_0^n|^2\mathrm{e}^{-2t} }{\sigma_{\infty}^2},\quad
  	\mathbb{E}[\psi(Z_2)]
  	=n\Bigr(\frac{\sigma_{\infty}^2}{\sigma_t^2}-1\Bigr)+\frac{|z_0^n|^2\mathrm{e}^{-2t}}{\sigma_t^2}
  \end{equation*}
  while
  \begin{equation*}
  	\mathrm{Var}[\psi(Z_1)]
  	=2n\Bigr(\frac{1}{\sigma_t^2}-\frac{1}{\sigma_{\infty}^2}\Bigr)^2\sigma_t^4+4\frac{\sigma_{t}^2}{\sigma_{\infty}^4}|z_0^n|^2\mathrm{e}^{-2t},
  	\quad
  	\mathrm{Var}[\psi(Z_2)]
  	=2n\Bigr(\frac{1}{\sigma_t^2}-\frac{1}{\sigma_{\infty}^2}\Bigr)^2\sigma_{\infty}^4+4\frac{\sigma_\infty^2}{\sigma_{t}^4}|z_0^n|^2\mathrm{e}^{-2t}.
  \end{equation*}
    Let $t=t_{n,b}$ be as in Table \ref{tb:OU:profile} for Hellinger. Using \eqref{eq:TVpsi} and the central limit theorem for the i.i.d.\ random variables $X_1,\ldots,X_n$ and
    $Y_1,\ldots,Y_n$, we get, with $Z\sim\mathcal{N}(0,1)$,
    \begin{equation*}
      \bigr\Vert\mathrm{Law}({Z}_t^n)-P_n^0\bigr\Vert_{\mathrm{TV}}
      = \mathbb{P}(Z\leq \gamma_{n,t})-\mathbb{P}(Z\leq \tilde{\gamma}_{n,t})+o_n(1), 
    \end{equation*}
    where
    \begin{equation*}
      \gamma_{n,t}:=\frac{-n\log(1-\mathrm{e}^{-2t})-\mathbb{E}[\psi(Z^n_t)]
      }{\sqrt{\mathrm{Var}[\psi(Z^n_t)]} },
      \quad
      \tilde{\gamma}_{n,t}:=\frac{-n\log(1-\mathrm{e}^{-2t})-\mathbb{E}[\psi(Z^n_\infty)]}{\sqrt{\mathrm{Var}[\psi(Z^n_\infty)]} }.
    \end{equation*}
    Expanding $\gamma_{n,t_{n,b} }$ gives the cutoff profile. Let us detail
    the computations in the most involved case
    $\lim_{n\to\infty}|z_0^n|^2\sqrt{n}=a\in(0,+\infty)$. For all
    $b\in \mathbb{R}$, recall $t_{n,b}=\frac{\log n}{4}+b$. One may check that
    \begin{equation*}
      -n\log (1-\mathrm{e}^{-2t_{n,b} })-\mathbb{E}[\psi(Z^n_{t_{n,b }}) ]=\frac{1}{2}\mathrm{e}^{-4b}+a \mathrm{e}^{-2b}+o_n(1),
    \end{equation*}
    \begin{equation*}
      -n\log (1-\mathrm{e}^{-2t_{n,b} })-\mathbb{E}[\psi(Z^n_\infty)]
      =-\frac{1}{2}\mathrm{e}^{-4b}-a \mathrm{e}^{-2b}+o_n(1),
    \end{equation*}
    \begin{equation*}
      \mathrm{Var}[\psi(Z^n_{t_{n,b}})]=2\mathrm{e}^{-4b}+4a
      \mathrm{e}^{-2b}+o_n(1), 
      \quad \mathrm{Var}[\psi(Z^n_\infty)]=2\mathrm{e}^{-4b}+4a \mathrm{e}^{-2b}+o_n(1).
    \end{equation*}
    It follows that
    \begin{equation*}
      \lim_{n\to\infty}\bigr\Vert\mathrm{Law}({Z}_{t_{n,b}}^n)-P_n^0\bigr\Vert_{\mathrm{TV}}=\mathbb{P}\Bigr(|Z|\leq \frac{1}{2\sqrt{2}}\sqrt{\mathrm{e}^{-4b }+2a \mathrm{e}^{-2b} }\Bigr)=\mathrm{erf}\Bigr(\frac{1}{4}\sqrt{\mathrm{e}^{-4b}+2a\mathrm{e}^{-2b} }\Bigr).
    \end{equation*}
    The other cases are similar.
  \end{proof}

\section{General exactly solvable aspects}
\label{se:proofs:DOUOU:DOUCIR:moments}

In this section, we prove Theorem \ref{th:DOUOU} and Corollary
\ref{cor:LowerBd}.

The proof of Theorem \ref{th:DOUOU} is based on the fact that the polynomial
functions $\pi(x)=x_1+\cdots+x_n$ and $|x|^2=x_1^2+\cdots+x_n^2$ are, up to an
additive constant for the second, eigenfunctions of the dynamics associated to
the spectral values $-1$ and $-2$ respectively, and that their ``carré du
champ'' is affine. In the matrix cases $\beta\in\{1,2\}$, these functions
correspond to the dynamics of the trace, the dynamics of the squared
Hilbert--Schmidt trace norm, and the dynamics of the squared trace. It is
remarkable that this phenomenon survives beyond these matrix cases, yet
another manifestation of the Gaussian ``ghosts'' concept due to Edelman, see
for instance \cite{zbMATH05938051}.

\begin{proof}[Proof of Theorem \ref{th:DOUOU}]
  The process $Y_t := \pi(X_t^n)$ solves
  \[
    \mathrm{d}Y_t = \sum_{i=1}^n \mathrm{d} X_t^{n,i} =  \sqrt{\frac{2}{n}} \sum_{i=1}^n\mathrm{d}B^i_t
    - \sum_{i=1}^n X_t^{n,i} \mathrm{d}t
    +\frac{\beta}{n}\sum_{j\neq
      i}\frac{\mathrm{d}t}{X_t^{n,i}-X_t^{n,j}}.
  \]
  By symmetry, the double sum vanishes. Note that the process
  $W_t := \frac1{\sqrt n} \sum_{i=1}^n B^i_t$ is a standard one dimensional
  BM, so that
  \(
    \mathrm{d}Y_t = \sqrt{2} \mathrm{d}W_t - Y_t \mathrm{d}t.
  \)
  This proves the first part of the statement.\\
  We turn to the second part. Recall that $X_t \in D_n$ for all $t>0$. By
  Itô's formula
  \[
    \mathrm{d} (X_t^{n,i})^2 = \sqrt{\frac{8}{n}} X_t^{n,i} \mathrm{d}B^i_t -
    2 (X_t^{n,i})^2 \mathrm{d}t + 2\frac{\beta}{n} X_t^{n,i} \sum_{j:j\neq
      i}\frac{\mathrm{d}t}{X_t^{n,i}-X_t^{n,j}} + \frac{2}{n} \mathrm{d}t.
  \]
  Set
  $W_t := \sum_{i=1}^n \int_0^t \frac{X_s^{n,i}}{|X^n_s|} \mathrm{d}B^i_s$.
  The process ${(W_t)}_{t\geq0}$ is a BM by the Lévy characterization since
  \[
    \langle W\rangle_t
    =\int_0^t\frac{\sum_{i=1}^n(X^{n,i}_s)^2}{|X^n_s|^2}\mathrm{d}s
    =t.
  \]
  Furthermore, a simple computation shows that
  \[
    \sum_{i=1}^n X_t^{n,i} \sum_{j:j\neq i}\frac{1}{X_t^{n,i}-X_t^{n,j}} =
    \frac{n(n-1)}{2}.
  \]
  Consequently the process $R_t := |X_t^n|^2$ solves
  \[
    \mathrm{d} R_t = \sqrt{\frac{8}{n} R_t} \mathrm{d}W_t + \Big(2 +
    \beta(n-1) - 2 R_t\Big) \mathrm{d}t,
  \]
  and is therefore a CIR process of parameters $a=2+\beta(n-1)$, $b=2$, and $\sigma = \sqrt{8/n}$.\\
  When $d=\frac{\beta}{2}n^2 + (1-\frac{\beta}{2})n$ is a positive integer,
  the last property of the statement follows from the connection between OU
  and CIR recalled right before the statement of the theorem.
\end{proof}

The last proof actually relies on the following general observation. Let $X$
be an $n$-dimensional continuous semi-martingale solution of
\[
  \mathrm{d}X_t=\sigma(X_t)\mathrm{d}B_t+b(X_t)\mathrm{d}t
\]
where $B$ is a $n$-dimensional standard BM, and where
\[
  x\in\mathbb{R}^n\mapsto\sigma(x)\in\mathcal{M}_{n,n}(\mathbb{R})
  \quad\text{and}\quad
  x\in\mathbb{R}^n\mapsto b(x)\in\mathbb{R}^n
\]
are Lipschitz. The infinitesimal generator of the Markov semigroup is given by
\[
  \G(f)(x)
  =\frac{1}{2}\sum_{i,j=1}^na_{i,j}(x)\partial_{i,j}f(x)
  +\sum_{i=1}^nb_i(x)\partial_if(x),
  \quad\text{where}\quad
  a(x)=\sigma(x)(\sigma(x))^\top,    
\]
for all $f\in\mathcal{C}^2(\mathbb{R}^n,\mathbb{R})$ and $x\in\mathbb{R}^n$.
Then, by Itô's formula, the process $M^f={(M^f_t)}_{t\geq0}$ given by
\[
  M^f_t=f(X_t)-f(X_0)-\int_0^t(\G f)(X_s)\mathrm{d}s
  =\sum_{i,k=1}^n\int_0^t\partial_if(X_s)\sigma_{i,k}(X_s)\mathrm{d}B_s^k
\]
is a local martingale, and moreover, for all $t\geq0$,
\[
  \langle M^f\rangle_t
  =\int_0^t\Gamma(f)(X_s)\mathrm{d}s
  \quad\text{where}\quad
  \Gamma(f)(x)=|\sigma(x)^\top\nabla f(x)|^2=a(x)\nabla f\cdot\nabla f.
\]
The functional quadratic form $\Gamma$ is known as the ``carré du champ''
operator.\\
If $f$ is an eigenfunction of $\G$ associated to the spectral value $\lambda$
in the sense that $\G f=\lambda f$ (note by the way that $\lambda\leq0$ since
$\G$ generates a Markov process), then we get
\[
  f(X_t)=f(X_0)+\lambda\int_0^tf(X_s)\mathrm{d}s
  +M_t^f,
  \quad\text{in other words}\quad
  \mathrm{d}f(X_t)=\mathrm{d}M^f_t+\lambda f(X_t)\mathrm{d}t.
\]
Now if $\Gamma(f) = c$ (as in the first part of the theorem), then by the Lévy
characterization of Brownian motion, the continuous local martingale
$W:=\frac{1}{\sqrt{c}}M^f$ starting from the origin is a standard BM and we
recover the result of the first part of the theorem. On the other hand, if
$\Gamma(f) = cf$ (as in the second part of the theorem), then by the Lévy
characterization of BM the local martingale
\[
  W := \int_0^t \frac{1}{\sqrt{cf(X_s)}}\mathrm{d}M_s^f
\]
is a standard BM and we recover the result of the second part.

\medskip

At this point, we observe that the infinitesimal generator of the CIR process
$R$ is the Laguerre partial differential operator
\begin{equation}\label{eq:LaguerreDOU}
  L(f)(x)=\frac{4}{n}xf''(x)+(2+\beta(n-1)-2x)f'(x).
\end{equation}
This operator leaves invariant the set of polynomials of degree less than or
equal to $k$, for all integer $k\geq0$, a property inherited from
\eqref{eq:G}. We will use this property in the following proof.

\subsection{Proof of Corollary \ref{cor:LowerBd}}
\label{se:DOUWTV:lb}

By Theorem \ref{th:DOUOU}, $Z = \pi(X^n)$ is an OU process
in $\mathbb{R}$ solution of the stochastic differential equation
\[
  Z_0 = \pi(X_0^n),
\quad
  \mathrm{d}Z_t = \sqrt 2 \mathrm{d}B_t - Z_t \mathrm{d}t,
\]
where $B$ is a standard one-dimensional BM. By Lemma
\ref{le:Mehler},
$Z_t \sim \mathcal{N}(Z_0 \mathrm{e}^{-t}, 1-\mathrm{e}^{-2t})$ for all
$t\ge 0$ and the equilibrium distribution is
$P_n^\beta \circ \pi^{-1} = \mathcal{N}(0, 1)$. Using the contraction property stated
in Lemma \ref{le:contraction}, the comparison between Hellinger and TV of
Lemma \ref{le:distineqs} and the explicit expressions for Gaussian
distributions of Lemma \ref{le:gauss}, we find
\begin{align*}
\|\mathrm{Law}(X^n_t)-P_n^\beta\|_{\mathrm{TV}}
      &\geq\|\mathrm{Law}(Z_t)-P_n^\beta\circ \pi^{-1}\|_{\mathrm{TV}}\\
      &\geq  \mathrm{Hellinger}^2(\mathrm{Law}(Z_t), P_n^\beta\circ \pi^{-1})\\
      &= 1 - 
        \frac{(1-\mathrm{e}^{-2t})^{1/4}}
        {(1-\frac{1}{2}\mathrm{e}^{-2t})^{1/2}}
        \exp\Bigr(-\frac{\pi(X^n_0)^2\mathrm{e}^{-2t}}{4(2-\mathrm{e}^{-2t})}\Bigr).
\end{align*}
Setting $c_n := \log(\vert\pi(X^n_0)\vert)$ and assuming that
$\lim_{n\to\infty}c_n=\infty$, we deduce that for all $\varepsilon \in (0,1)$
 \[
    \lim_{n\to\infty}
    \|\mathrm{Law}(X^n_{c_n(1-\varepsilon)})-P_n^\beta\|_{\mathrm{TV}}
    =1.\]
The comparison between $\mathrm{Hellinger}$ and $\mathrm{TV}$ of
Lemma \ref{le:distineqs} allows to deduce that this remains true for the Hellinger distance.

We turn to Kullback. The contraction property stated in Lemma
\ref{le:contraction} and the explicit expressions for Gaussian distributions
of Lemma \ref{le:gauss} yield
\begin{align*}
2\mathrm{Kullback}(\mathrm{Law}(X^n_t) \mid P_n^\beta )
      &\geq 2\mathrm{Kullback}(\mathrm{Law}(Z_t) \mid P_n^\beta\circ \pi^{-1})\\
      &= \pi(X^n_0)^2\mathrm{e}^{-2t} -\mathrm{e}^{-2t} - \log(1-\mathrm{e}^{-2t}).
\end{align*}
This is enough to deduce that
\[
    \lim_{n\to\infty}
    \mathrm{Kullback}(\mathrm{Law}(X^n_{(1-\varepsilon)c_n}) \mid P_n^\beta)
    =+\infty.\]

The situation is similar for $\chi^2$: the contraction property stated in Lemma \ref{le:contraction} and the explicit expressions for Gaussian
distributions of Lemma \ref{le:gauss} yield
\begin{align*}
	\chi^2(\mathrm{Law}(X^n_t) \mid P_n^\beta )
	&\geq \chi^2(\mathrm{Law}(Z_t) \mid P_n^\beta\circ \pi^{-1})\\
	&= -1+\frac1{\sqrt{1-\mathrm{e}^{-4t}}} \exp\left(\frac1{1+\mathrm{e}^{-2t}}(1-\pi(X^n_0) \mathrm{e}^{-t})^2\right),
\end{align*}
so that
\[
\lim_{n\to\infty}
\chi^2(\mathrm{Law}(X^n_{(1-\varepsilon)c_n}) \mid P_n^\beta)
=+\infty.\]

\medskip

Regarding the Wasserstein distance, we have
$\left\|\pi\right\|_{\mathrm{Lip}}:=\sup_{x\neq
  y}\frac{|\pi(x)-\pi(y)|}{|x-y|}\leq\sqrt{n}$ from the Cauchy--Schwarz
inequality, and by Lemma \ref{le:contraction}, for all probability measures
$\mu$ and $\nu$ on $\mathbb{R}^n$,
\begin{equation}\label{eq:W2contract}
  \mathrm{Wasserstein}(\mu\circ\pi^{-1},\nu\circ\pi^{-1})
  \leq\sqrt{n}\mathrm{Wasserstein}(\mu,\nu).
\end{equation}
Using the explicit expressions for Gaussian distributions of Lemma \ref{le:gauss}, we thus find 
\begin{align*}
  \mathrm{Wasserstein}^2(\mathrm{Law}(X_t^n), P_n^\beta)
  &\ge \frac1{n} \mathrm{Wasserstein}^2(\mathrm{Law}(Z_t), P_n^\beta\circ \pi^{-1})\\
  &= \frac1{n}\Bigr(\pi(X^n_0)^2 \mathrm{e}^{-2t} + 2 - \mathrm{e}^{-2t} - 2\sqrt{1-\mathrm{e}^{-2t}}\Bigr).
\end{align*}
Setting $c_n:= \log\Big(\frac{\vert\pi(x_0^n)\vert}{\sqrt{n}}\Big)$ and assuming $c_n \to\infty$ as $n\to\infty$, we thus deduce that for all $\varepsilon \in (0,1)$
\[
  \lim_{n\to\infty}
  \mathrm{Wasserstein}(\mathrm{Law}(X^n_{(1-\varepsilon)c_n}),P_n^\beta)
  =+\infty.
\]

\section{The random matrix cases}
\label{se:DOU12}

In this section, we prove Theorem \ref{th:DOU12} and Corollary \ref{cor:DOU12}
that cover the matrix cases $\beta \in \{1,2\}$. For these values of $\beta$,
the DOU process is the image by the spectral map of a matrix OU process,
connected to the random matrix models $\mathrm{GOE}$ and $\mathrm{GUE}$. We
could consider the case $\beta=4$ related to $\mathrm{GSE}$. Beyond these
three algebraic cases, it could be possible for an arbitrary $\beta\geq1$ to
use random tridiagonal matrices dynamics associated to $\beta$ Dyson
processes, see for instance \cite{HolcombPaquette01}.

The next two subsections are devoted to the proof of Theorem \ref{th:DOU12} in
the $\beta =2$ and $\beta = 1$ cases respectively. The third section provides
the proof of Corollary \ref{cor:DOU12}.

\subsection{Hermitian case ($\beta=2$)}

Let $\mathrm{Herm}_n$ be the set of $n\times n$ complex Hermitian matrices,
namely the set of $h\in\mathcal{M}_{n,n}(\mathbb{C})$ with
$h_{i,j}=\overline{h_{j,i}}$ for all $1\leq i,j\leq n$. An element
$h\in\mathrm{Herm}_n$ is parametrized by the $n^2$ real variables
$(h_{i,i})_{1\leq i\leq n}$, $(\Re h_{i,j})_{1\leq i<j\leq n}$,
$(\Im h_{i,j})_{1\leq i<j\leq n}$. We define, for $h\in\mathrm{Herm}_n$ and $1\leq i,j\leq n$,
\begin{equation}\label{eq:GUEij}
  \pi_{i,j}(h)
  =\begin{cases}
    h_{i,i} & \text{if $i=j$}\\
    \sqrt 2\, \Re h_{i,j} & \text{if $i<j$}\\
    \sqrt 2\, \Im h_{j,i} & \text{if $i>j$}
  \end{cases}.
\end{equation}
Note that
\[ \mathrm{Tr}(h^2)=\sum_{i,j=1}^n|h_{i,j}|^2=\sum_{i=1}^nh_{i,i}^2+2\sum_{i<j}(\Re
h_{i,j})^2+2\sum_{i<j}(\Im h_{i,j})^2 = \sum_{i,j} \pi_{i,j}(h)^2.\]
We thus identify $\mathrm{Herm}_n$ with $\mathbb{R}^n\times\mathbb{R}^{2\frac{n^2-n}{2}}=\mathbb{R}^{n^2}$, this identification is isometrical provided $\mathrm{Herm}_n$ is endowed with the norm $\sqrt{\mathrm{Tr}(h^2)}$ and $\mathbb{R}^{n^2}$ with the Euclidean norm.

\medskip

The Gaussian Unitary Ensemble $\mathrm{GUE}_n$ is the Gaussian law on
$\mathrm{Herm}_n$ with density
\begin{equation}\label{eq:GUE}
  h\in\mathrm{Herm}_n
  \mapsto
  \frac{\mathrm{e}^{-\frac{n}{2}\mathrm{Tr}(h^2)}}{C_n}
  \quad\text{where}\quad
  C_n:=\int_{\mathbb{R}^{n^2}}\mathrm{e}^{-\frac{n}{2}\mathrm{Tr}(h^2)}
  \prod_{i=1}^n\mathrm{d}h_{i,i}\prod_{i<j}\mathrm{d}\Re h_{i,j}\prod_{i<j}\mathrm{d}\Im h_{i,j}.
\end{equation}
If $H$ is a random $n\times n$
Hermitian matrix then $H\sim\mathrm{GUE}_n$ if and only if the $n^2$ real
random variables $\pi_{i,j}(H)$, $1\leq i,j\leq n$, are independent Gaussian
random variables with
\begin{equation}
  \pi_{i,j}(H)\sim
  \mathcal{N}\Bigr(0,\frac{1}{n}\Bigr),
  \quad 1\leq i,j\leq n.
\end{equation}
The law $\mathrm{GUE}_n$ is the unique invariant law of the Hermitian matrix
OU process ${(H_t)}_{t\geq0}$ on $\mathrm{Herm}_n$ solution
of the stochastic differential equation
\begin{equation}\label{Eq:Herm}
  H_0=h_0\in\mathrm{Herm}_n,\quad
  \mathrm{d}H_t=\sqrt{\frac{2}{n}}\mathrm{d}B_t-H_t\mathrm{d}t,
\end{equation}
where $B={(B_t)}_{t\geq0}$ is a Brownian motion on $\mathrm{Herm}_n$, in the
sense that the stochastic processes $(\pi_{i,j}(B_t))_{t\geq 0}$, $1\le i \ne j \le n$, are independent standard one-dimensional BM. The coordinates
stochastic processes ${(\pi_{i,j}(H_t))}_{t\geq0}$, $1\leq i,j\leq n$, are
independent real OU processes.

\smallskip

For any $h$ in $\mathrm{Herm}_n$, we denote by $\Lambda(h)$ the vector of the
eigenvalues of $h$ ordered in non-decreasing order. Lemma \ref{le:DOU} below
is an observation which dates back to the seminal work of Dyson
\cite{MR148397}, hence the name DOU for $X^n$. We refer to
\cite[Ch.~12]{MR3699468} and \cite[Sec.~4.3]{MR2760897} for a mathematical
approach using modern stochastic calculus.

\begin{lemma}[From matrix OU to DOU]\label{le:DOU}
  The image of $\mathrm{GUE}_n$ by the map $\Lambda$ is the Coulomb gas
  $P_n^\beta$ given by \eqref{eq:P} with $\beta=2$. Moreover the stochastic
  process $X^n={(X^n_t)}_{t\geq0}={(\Lambda(H_t))}_{t\geq0}$ is well-defined
  and solves the stochastic differential equation \eqref{eq:DOU} with
  $\beta=2$ and $x_0^n=\Lambda(h_0)$.
\end{lemma}

Let $\beta=2$. Let us assume from now on that the initial value
$h_0\in\mathrm{Herm}_n$ of ${(H_t)}_{t\geq0}$ has eigenvalues $x_0^n$ where
$x^n_0$ is as in Theorem \ref{th:DOU12}. We start by proving the upper bound
on the $\chi^2$ distance stated in Theorem \ref{th:DOU12}: it will be an
adaptation of the proof of the upper bound of Theorem \ref{th:OU1} applied to
the Hermitian matrix OU process ${(H_t)}_{t\geq0}$ combined with the
contraction property of the $\chi^2$ distance. Indeed, by Lemma \ref{le:DOU}
and the contraction property of Lemma \ref{le:contraction}
\begin{equation} \chi^2(\mathrm{Law}(X_t^n) \mid P_n^\beta) \leq
    \chi^2(\mathrm{Law}(H_t) \mid \mathrm{GUE}_n).
\end{equation}
We claim now that the right-hand side tends to $0$ as $n\to\infty$ when
$t=t_n$ is well chosen. Indeed, using the identification between
$\mathrm{Herm}_n$ and $\mathbb{R}^{n^2}$ mentioned earlier, we have
$\mathrm{GUE}_n=\mathcal{N}(m_2,\Sigma_2)$ where $m_2=0$ and where $\Sigma_2$
is an $n^2\times n^2$ diagonal matrix with
\begin{equation}
	(\Sigma_2)_{(i,j),(i,j)}=\frac{1}{n}.
\end{equation}
On the other hand, the Mehler formula (Lemma \ref{le:Mehler}) gives
$\mathrm{Law}(H_t)=\mathcal{N}(m_1,\Sigma_1)$
where $m_1=\mathrm{e}^{-t}h_0$ and where $\Sigma_1$ is an $n^2\times n^2$
diagonal matrix with
\begin{equation}
	(\Sigma_1)_{(i,j),(i,j)}
	=\frac{1-\mathrm{e}^{-2t}}{n}.
\end{equation}
Therefore, using Lemma \ref{le:gauss}, the analogue of \eqref{eq:OUK} reads
\begin{equation}\label{eq:OUKM}
	\chi^2(\mathrm{Law}(H_t)\mid \mathrm{GUE}_n)
	=-1 + \frac1{(1-\mathrm{e}^{-4t})^{n^2/2}} \exp\left(n|h_0|^2\frac{\mathrm{e}^{-2t}}{1+\mathrm{e}^{-2t}}\right).
\end{equation}
where
\begin{equation}
	|h_0|^2
	=\sum_{1\leq i,j\leq n}\pi_{i,j}(h_0)^2 
	=\sum_{1\leq i,j\leq n}|(h_0)_{i,j}|^2=\mathrm{Tr}(h_0^2)=|x_0^n|^2.
\end{equation}
Taking now $c_n := \log(\sqrt{n} |x_0^n|) \vee \log (\sqrt{n})$, for any
$\varepsilon\in(0,1)$, we get
\begin{equation}
	\chi^2(\mathrm{Law}(X_{(1+\varepsilon)c_n}^n) \mid P_n^\beta)
	\leq
	\chi^2(\mathrm{Law}(H_{(1+\varepsilon)c_n}) \mid\mathrm{GUE}_n)
	\underset{n\to\infty}{\longrightarrow}0.
\end{equation}
In the right-hand side of \eqref{eq:OUKM}, the factor $n^2$ is the dimension of
the $\mathbb{R}^{n^2}$ to which $\mathrm{Herm}_n$ is identified, while the
factor $n$ in the first term is due to the $1/n$ scaling in the stochastic
differential equation of the process. This explains the difference with the
analogue \eqref{eq:OUK} in dimension $n$.

From the comparison between TV, Hellinger, Kullback and $\chi^2$ stated in
Lemma \ref{le:distineqs}, we easily deduce that the previous convergence
remains true upon replacing $\chi^2$ by $\mathrm{TV}$,
$\mathrm{Hellinger}$ or $\mathrm{Kullback}$.

\medskip

It remains to cover the upper bound for the Wasserstein distance. This
distance is more sensitive to contraction arguments: according to Lemma
\ref{le:contraction}, one needs to control the Lipschitz norm of the
``contraction map'' at stake. It happens that the spectral map, restricted to
the set $\mathrm{Herm}_n$ of $n\times n$ Hermitian matrices, is $1$-Lipschitz:
more precisely, the Hoffman--Wielandt inequality, see~\cite{MR52379} and
\cite[Th.~6.3.5]{zbMATH06125590}, asserts that for any two such matrices $A$
and $B$, denoting $\Lambda(A)=(\lambda_i(A))_{1\le i \le n}$ and
$\Lambda(B)=(\lambda_i(B))_{1\le i \le n}$ the ordered sequences of their
eigenvalues, we have
\[
  \sum_{i=1}^{n} |\lambda_i(A) - \lambda_i(B)|^2 \le \sum_{i,j} |A_{i,j} -
  B_{i,j}|^2.
\]
Applying Lemma \ref{le:contraction}, we thus deduce that
\begin{equation}
  \mathrm{Wasserstein}(\mathrm{Law}(X_t^n) , P_n^\beta)
  \leq
  \mathrm{Wasserstein}(\mathrm{Law}(H_t) , \mathrm{GUE}_n).
\end{equation}
Following the Gaussian computations in the proof of Theorem \ref{th:OU2}, we obtain
\begin{equation}\label{eq:OUKMW}
\mathrm{Wasserstein}^2(\mathrm{Law}(H_t) , \mathrm{GUE}_n) = |x_0^n|^2 \mathrm{e}^{-2t} + 2 - \mathrm{e}^{-2t} - 2\sqrt{1-\mathrm{e}^{-2t}}.
\end{equation}
Set $c_n := \log(|x_0^n|)$. If $c_n\to \infty$ as $n\to\infty$ then for all $\varepsilon\in(0,1)$ we find
\[
  \mathrm{Wasserstein}(\mathrm{Law}(X_{(1+\varepsilon)c_n}^n) , P_n^\beta)
  \underset{n\to\infty}{\longrightarrow}0.
\]
This completes the proof of Theorem \ref{th:DOU12}.

\subsection{Symmetric case ($\beta=1$)}

The method is similar to the case $\beta=2$. Let us focus only on the
differences. Let $\mathrm{Sym}_n$ be the set of $n\times n$ real symmetric
matrices, namely the set of $s\in\mathcal{M}_{n,n}(\mathbb{R})$ with
$s_{i,j}=s_{j,i}$ for all $1\leq i,j\leq n$. An element $s\in\mathrm{Sym}_n$
is parametrized by the $n+\frac{n^2-n}{2}=\frac{n(n+1)}{2}$ real variables
$(s_{i,j})_{1\leq i\leq j\leq n}$. We define, for $s\in\mathrm{Sym}_n$ and $1\leq i\leq j\leq n$,
\begin{equation}\label{eq:GOEij}
  \pi_{i,j}(s)=\begin{cases}
    s_{i,i} & \text{if $i=j$}\\
    \sqrt 2\, s_{i,j} & \text{if $i < j$}
  \end{cases}.
\end{equation}
Note that
\[
  \mathrm{Tr}(s^2)=\sum_{i,j=1}^ns_{i,j}^2=\sum_{i=1}^ns_{i,i}^2+2\sum_{i<j}s_{i,j}^2
  = \sum_{1\le i\le j \le n}  \pi_{i,j}(s)^2.
\]
We thus identify isometrically $\mathrm{Sym}_n$, endowed with the norm $\sqrt{\mathrm{Tr}(h^2)}$, with $\mathbb{R}^n\times\mathbb{R}^{\frac{n^2-n}{2}}=\mathbb{R}^{\frac{n(n+1)}{2}}$ endowed with the Euclidean norm.

\medskip

The Gaussian Orthogonal Ensemble $\mathrm{GOE}_n$ is the Gaussian law on
$\mathrm{Sym}_n$ with density 
\begin{equation}\label{eq:GOE}
  s\in\mathrm{Sym}_n
  \mapsto
  \frac{\mathrm{e}^{-\frac{n}{2}\mathrm{Tr}(s^2)}}{C_n}
  \quad\text{where}\quad
  C_n:=\int_{\mathbb{R}^{\frac{n(n+1)}{2}}}\mathrm{e}^{-\frac{n}{2}\mathrm{Tr}(s^2)}
  \prod_{1\leq i\leq j\leq n}\mathrm{d}s_{i,j}.
\end{equation}

If $S$ is a random $n\times n$ real symmetric matrix then
$S\sim\mathrm{GOE}_n$ if and only if the $\frac{n(n+1)}{2}$ real random
variables $\pi_{i,j}(S)$, $1\leq i\leq j\leq n$, are independent Gaussian
random variables with
\begin{equation}
  \pi_{i,j}(S)\sim
  \mathcal{N}\Bigr(0,\frac{1}{n}\Bigr),
  \quad 1\leq i\leq j\leq n.
\end{equation}
The law $\mathrm{GOE}_n$ is the unique invariant law of the real symmetric
matrix OU process ${(S_t)}_{t\geq0}$ on $\mathrm{Sym}_n$
solution of the stochastic differential equation
\begin{equation}\label{Eq:Sym}
  S_0=s_0\in\mathrm{Sym}_n,\quad
  \mathrm{d}S_t=\sqrt{\frac{2}{n}}\mathrm{d}B_t-S_t\mathrm{d}t
\end{equation}
where $B={(B_t)}_{t\geq0}$ is a Brownian motion on $\mathrm{Sym}_n$, in the
sense that the stochastic processes $(\pi_{i,j}(B_t))_{t\geq0}$, $1\le i\le j \le n$, are independent standard one-dimensional BM. The coordinates
stochastic processes ${(\pi_{i,j}(S_t))}_{t\geq0}$, $1\leq i\leq j\leq n$, are
independent real OU processes.

For any $s$ in $\mathrm{Sym}_n$, we denote by $\Lambda(s)$ the vector of the
eigenvalues of $s$ ordered in non-decreasing order. Lemma \ref{le:DOUS}
below is the real symmetric analogue of Lemma \ref{le:DOU}.

\begin{lemma}[From matrix OU to DOU]\label{le:DOUS}
  The image of $\mathrm{GOE}_n$ by the map $\Lambda$ is the Coulomb gas
  $P_n^\beta$ given by \eqref{eq:P} with $\beta=1$. Moreover the stochastic
  process $X^n={(X^n_t)}_{t\geq0}={(\Lambda(S_t))}_{t\geq0}$ is well-defined
  and solves the stochastic differential equation \eqref{eq:DOU} with
  $\beta=1$ and $x_0^n=\Lambda(s_0)$.
\end{lemma}

As for the case $\beta=2$, the idea now is that the DOU process is sandwiched
between a real OU process and a matrix OU process.

By similar computations to the case $\beta=2$, the analogue of \eqref{eq:OUKM} becomes
\begin{equation}\label{eq:OUKM1}
  \chi^2(\mathrm{Law}(H_t)\mid \mathrm{GOE}_n)
  =-1 + \frac1{(1-\mathrm{e}^{-4t})^{\frac{(n(n+1))^2}{8}}} \exp\left(n|h_0|^2\frac{\mathrm{e}^{-2t}}{1+\mathrm{e}^{-2t}}\right).
\end{equation}
This allows to deduce the upper bound for TV, Hellinger, Kullback and $\chi^2$. Regarding
the Wasserstein distance, the analogue of \eqref{eq:OUKMW} reads
\begin{equation}\label{eq:OUKMW1}
  \mathrm{Wasserstein}^2(\mathrm{Law}(S_t) , \mathrm{GOE}_n)
  = |x_0^n|^2 \mathrm{e}^{-2t} + 2 - \mathrm{e}^{-2t} - 2\sqrt{1-\mathrm{e}^{-2t}}.
\end{equation}
If $\lim_{n\to\infty}\log(|x_0^n|)=\infty$ then we deduce the asserted result,
concluding the proof of Theorem \ref{th:DOU12}.

\subsection{Proof of Corollary \ref{cor:DOU12}}

Let $\beta \in \{1,2\}$. Recall the definitions of $a_n$ and $c_n$ from the
statement. Take $x_0^{n,i} = a_n$ for all $i$, and note that
$\pi(x_0^n) = n a_n$. Given our assumptions on $a_n$, Corollary
\ref{cor:LowerBd} yields for this particular choice of initial condition and
for any $\varepsilon \in (0,1)$
\[
  \lim_{n\to\infty}\mathrm{dist}(\mathrm{Law}(X^n_{(1-\varepsilon)c_n})\mid P_n^\beta)
  =\max.
\]
On the other hand, in the proof of Theorem \ref{th:DOU12} we saw that
\[
  \chi^2(\mathrm{Law}(X_t^n) \mid P_n^\beta) \le -1 + \frac1{(1-\mathrm{e}^{-4t})^{b_n/2}} \exp\left(n|x_0^n|^2\frac{\mathrm{e}^{-2t}}{1+\mathrm{e}^{-2t}}\right),
\]
where $b_n = n^2$ for $\beta =2$ and $b_n = (n(n+1)/2)^2$ for $\beta =1$.
Since $|x_0^n| \le \sqrt{n} a_n$ for all $x_0^n \in [-a_n,a_n]^n$, and given
the comparison between TV, Hellinger, Kullback and $\chi^2$ stated in Lemma
\ref{le:distineqs} we obtain for
$\mathrm{dist} \in \{\mathrm{TV}, \mathrm{Hellinger}, \mathrm{Kullback}, \chi^2\}$ and for all $\varepsilon \in (0,1)$
\[
  \lim_{n\to\infty}\sup_{x_0^n \in [-a_n,a_n]^n}\mathrm{dist}(\mathrm{Law}(X^n_{(1+\varepsilon)c_n})\mid P_n^\beta)
  =0,
\]
thus concluding the proof of Corollary \ref{cor:DOU12} regarding theses distances.\\
Concerning Wasserstein, the proof of Theorem \ref{th:DOU12} shows that for any
$x_0^n \in [-a_n,a_n]^n$ we have
\begin{align*}
  \mathrm{Wasserstein}^2(\mathrm{Law}(X_t^n) , P_n^\beta)
  &\le |x_0^n|^2 \mathrm{e}^{-2t} + 2 - \mathrm{e}^{-2t} - 2\sqrt{1-\mathrm{e}^{-2t}}\\
  &\le n a_n^2 \mathrm{e}^{-2t} + 2 - \mathrm{e}^{-2t} - 2\sqrt{1-\mathrm{e}^{-2t}}.
\end{align*}
If $\sqrt{n} a_n \to \infty$, then for $c_n = \log(\sqrt{n}a_n)$ we deduce
that for all $\varepsilon \in (0,1)$
\[
  \lim_{n\to\infty}\sup_{x_0^n \in [-a_n,a_n]^n}\mathrm{dist}(\mathrm{Law}(X^n_{(1+\varepsilon)c_n})\mid P_n^\beta)
  =0.
\]
    
\section{Cutoff phenomenon for the DOU in TV and Hellinger}
\label{se:DOUWTV:ub}

In this section, we prove Theorem \ref{th:DOUWTV} and Corollary
\ref{cor:DOUWTV} for the TV and Hellinger distances. We only consider the case
$\beta \ge 1$, although the arguments could be adapted \emph{mutatis mutandis}
to cover the case $\beta = 0$: note that the result of Theorem \ref{th:DOUWTV}
and Corollary \ref{cor:DOUWTV} for $\beta = 0$ can be deduced from Theorem
\ref{th:OU2}. At the end of this section, we also provide the proof of Theorem
\ref{th:DOUK}.

\subsection{Proof of Theorem \ref{th:DOUWTV} in TV and Hellinger}

By the comparison between TV and Hellinger stated in Lemma \ref{le:distineqs},
it suffices to prove the result for the TV distance, so we concentrate on this
distance until the end of this section. Our proof is based on the exponential
decay of the relative entropy at an explicit rate given by the optimal
logarithmic Sobolev constant. However, this requires the relative entropy of
the initial condition to be \emph{finite}. Consequently, we proceed in three
steps. First, given an arbitrary initial condition $x_0^n \in \overline{D}_n$,
we build an absolutely continuous probability measure $\mu_{x_0^n}$ on
$D_n$ that approximates $\delta_{x_0^n}$ and whose relative entropy is not too
large. Second, we derive a decay estimate starting from this regularized
initial condition. Third, we control the
total variation distance between the two processes starting respectively from
$\delta_{x_0^n}$ and $\mu_{x_0^n}$.

\subsubsection{Regularization}

In order to have a finite relative entropy at time $0$, we first regularize
the initial condition by smearing out each particle in a ball of radius
bounded below by $n^{-(\kappa+1)}$, for some $\kappa>0$. Let us first
introduce the regularization at scale $\eta$ of a Dirac distribution
$\delta_{z}$, $z\in \mathbb{R}$ by
\[
  \delta_z^{(\eta)}(\mathrm{d}u)=\mathrm{Uniform}([z,z+\eta])(\mathrm{d}u)=\eta^{-1} \mathbf{1}_{[z,z+\eta]} \mathrm{d}u.
\]
Given $x \in \overline{D}_n$ and $\kappa > 0$, we define a regularized version
of $\delta_{x}$ at scale $n^{-\kappa}$, that we denote
$\mu_x$, by setting
\begin{equation}\label{eq:muin}
	\mu_x=\otimes_{i=1}^n \delta_{x_i+3i\eta}^{(\eta)},
\end{equation}
where $\eta := n^{-(\kappa +1 )}$. The parameters have been tuned in such a
way that, independently of the choice of $x\in \overline{D}_n$, the following
properties hold. The supports of the Dirac masses
$\delta_{x_i+3i\eta}^{(\eta)}$, $i\in \{1,\ldots,n\}$, lie at distance at
least $\eta$ from each other. The volume of the support of $\mu_x$ is equal to
$\eta^n$, and therefore the relative entropy of $\mu_x$ with respect to the
Lebesgue measure is not too large. Finally, provided $X_0^n = x$ and $Y_0^n$
is distributed according to $\mu_x$, almost surely
$\vert X_0^n - Y_0^n\vert_\infty \le (3n+1) \eta$.

\subsubsection{Convergence of the regularized process to equilibrium}

\begin{lemma}[Convergence of regularized process]\label{le:convergence regul}
  Let $(Y_t^n)_{t\geq0}$ be a DOU process solution of
  \eqref{eq:DOU}, $\beta\geq 1$, and let $P_n^\beta$ be its invariant law. Assume
  that $\mathrm{Law}(Y^n_0)$ is the
  regularized measure $\mu_{x_0^n}$ in \eqref{eq:muin} associated to some initial condition $x_0^n \in \overline{D}_n$. Then there exists a constant $C > 0$, only depending on $\kappa$, such that for all $t\ge 0$, all $n\ge 2$ and all $x_0^n \in \overline{D}_n$
  \begin{equation*}
    \mathrm{Kullback}(\mathrm{Law}(Y_t^n)\mid P_n^\beta)
    \leq C (n|x_0^n|^2 + n^2\log(n))\mathrm{e}^{-2t}.
  \end{equation*}
\end{lemma}

\begin{proof}[Proof of Lemma \ref{le:convergence regul}]
  By Lemma \ref{le:expdec} and since
  $\mathrm{Law}(Y^n_0) = \mu_{x_0^n}$, for all $t\geq 0$, there holds
  \begin{equation}\label{eq:decay kull}
    \mathrm{Kullback}(\mathrm{Law}(Y_t^n)\mid P_n^\beta)
    \leq  \mathrm{Kullback}(\mu_{x_0^n}\mid P_n^\beta)\mathrm{e}^{-2t}.
  \end{equation}
  Now we have
  \begin{equation*}
    \mathrm{Kullback}(\mu_{x_0^n}\mid P_n^\beta) %
    = \mathbb{E}_{\mu_{x_0^n} }\left[\log
      \frac{\mathrm{d}\mu_{x_0^n} }{\mathrm{d} P_n^\beta} \right].  
  \end{equation*}
  Recall the definition of $S$ in \eqref{eq:S}. As $P_n^\beta$ has density
  $\frac{\mathrm{e}^{- E}}{C_n^\beta}$, we may re-write this as
  \begin{equation}\label{eq:splitting kull}
    \mathrm{Kullback}(\mu_{x_0^n}\mid P_n^\beta)=
    S(\mu_{x_0^n})
    +\mathbb{E}_{\mu_{x_0^n}}[E]
    +\log C_n^\beta.
  \end{equation}
  Recall the partition function $C_{*n}^\beta = n! C_n^\beta$ from Subsection
  \ref{Subsec:Analysis}. It is proved in \cite{MR1465640}, using explicit
  expressions involving Gamma functions via a Selberg integral, that for some
  constant $C>0$
  \begin{equation}\label{eq:logZn}
    \log C_n^\beta \le \log C_{*n}^\beta\leq Cn^2.
  \end{equation}
  
  Next, we claim that
  $S(\mu_{x_0^n})\leq n
  \log(n^{1+\kappa})$. Indeed since $\mu_{x_0^n}$ is a product
  measure, the tensorization property of entropy recalled in Lemma
  \ref{le:tenso} gives
  \begin{equation*}
    \mathrm{Kullback}(\mu_{x_0^n} \mid \mathrm{d}x ) %
    = \sum_{i=1}^n \mathrm{Kullback}( \delta_{0}^{(\eta)} \mid \mathrm{d}x).  
  \end{equation*}
  Moreover an immediate computation yields
  $\mathrm{Kullback}( \delta_{0}^{(\eta)} \mid \mathrm{d}x)=\log(\eta^{-1})$ so that, given the definition of $\eta$,  we get
  \begin{equation}\label{eq:in entropy}
    \mathrm{Kullback}(\mu_{x_0^n} \mid \mathrm{d}x) = n \log(n^{\kappa+1}).
  \end{equation}
  We turn to the estimation of the term
  $\mathbb{E}_{\mu_{x_0^n} }[E].$ The confinement term can be easily bounded:
  \begin{equation*}
    \mathbb{E}_{\mu_{x_0^n} }\left[\frac{n}2 \sum_{i=1}^n {x_i^2} \right]\leq  (n|x_0^n|^2 + n^2 \eta^2).
  \end{equation*}
  Let us now estimate the logarithmic energy of $\mu_{x_0^n}$. Using the fact
  that the logarithmic function is increasing, together with the fact the
  supports of $\delta_{x_i+3i\eta}^{(\eta)}$ lie at distance at least $\eta$
  from each other, we notice that for any $i > j$ there holds
  \begin{align*}
    \mathbb{E}_{\mu_{x_0^n} }\left[\log |x_i-x_j| \right]
    &= \iint \log |x-y|\delta_{x_i+3i\eta}^{(\eta)}(\mathrm{d}x)\delta_{x_j+3j\eta}^{(\eta)}(\mathrm{d}y)\\
    &\ge \iint \log |x-y|\delta_{3\eta}^{(\eta)}(\mathrm{d}x)\delta_{0}^{(\eta)}(\mathrm{d}y)\\
    &\ge \log \eta.
  \end{align*}
  It follows that the initial logarithmic energy cannot be much larger than
  $n^2\log n$:
  \begin{equation*}
    \mathbb{E}_{\mu_{x_0^n} }\left[\sum_{i > j}\log \frac1{|x_i-x_j|} \right] %
    \leq \frac{n(n-1)}{2} \log n^{\kappa +1}.
  \end{equation*}
  This implies that there exists a constant $C>0$, only depending on $\kappa$,
  such that for all $n\ge 2$
  \begin{equation}\label{eq:in energy}
    \mathbb{E}_{\mu_{x_0^n}  }[E] %
    = \mathbb{E}_{\mu_{x_0^n}  }\left[ \frac{n}2 \sum_{i=1}^n |x_i|^2+{\beta}\sum_{i > j}\log\frac{1}{|x_i-x_j|}\right]\leq C \big(n|x_0^n|^2 + n^2\log n\big).
  \end{equation}
  Inserting \eqref{eq:logZn}, \eqref{eq:in entropy} and \eqref{eq:in
    energy} into \eqref{eq:splitting kull} we obtain (for a different constant $C>0$)
  \begin{equation*}
    \mathrm{Kullback}(\mu_{x_0^n} \mid P_n^\beta)\leq C \big(n|x_0^n|^2 + n^2\log n\big).
  \end{equation*}
  This bound, combined with \eqref{eq:decay kull}, concludes the proof of Lemma \ref{le:convergence regul}.
  \end{proof}

\subsubsection{Convergence to the regularized process in total variation distance}

Let $(X_t^n)_{t\geq 0}$ and $(Y_t^n)_{t\geq 0}$ be two DOU processes with
$X_0^n=x_0^n$ and $\mathrm{Law}(Y_0^n)=\mu_{x_0^n}$, where the measure
$\mu_{x_0^n}$ is defined in \eqref{eq:muin}. Below we prove that, as soon as
the parameter $\kappa$ is large enough, the total variation distance between
$\mathrm{Law}(X_t^n)$ and $\mathrm{Law}(Y_t^n)$ tends to $0$, for any fixed
$t> 0$.

Note that at time $0$, almost surely, there holds $X_0^{n,i} \le Y_0^{n,i}$,
for every $i\in\{1,\ldots,n\}$. We now introduce a coupling of the processes
$(X_t^n)_{t\geq 0}$ and $(Y_t^n)_{t\geq 0}$ that preserves this ordering at
all times. Consider two independent standard BM $B^n$ and $W^n$ in
$\mathbb{R}^n$. Let $X^n$ be the solution of \eqref{eq:DOU} driven by $B^n$,
and let $Y^n$ be the solution of
\[
  \mathrm{d}Y_t^{n,i} =\sqrt{\frac{2}{n}} \Big( \mathbf{1}_{\{Y_t^{n,i} \ne
    X_t^{n,i} \}} \mathrm{d}W^i_t + \mathbf{1}_{\{Y_t^{n,i} = X_t^{n,i} \}}
  \mathrm{d}B^i_t \Big) -Y_t^{n,i}\mathrm{d}t +\frac{\beta}{n}\sum_{j\neq
    i}\frac{\mathrm{d}t}{Y_t^{n,i}-Y_t^{n,j}},\quad 1\leq i\leq n.
\]
We denote by $\mathbb{P}$ the probability measure under which these two
processes are coupled. Let us comment on the driving noise in the equation
satisfied by $Y^n$. When the $i$-th coordinates of $X^n$ and $Y^n$ equal, we
take the same driving Brownian motion and the difference $Y^{n,i}-X^{n,i}$
remains non-negative due to the convexity of $-\log$ defined in \eqref{eq:E},
see the monotoncity result stated in Lemma \ref{lemma:monotonicity}. On the
other hand, when these two coordinates differ, we take independent driving
Brownian motions in order for their difference to have non-zero quadratic
variation (this allows to increase their merging probability). Under this
coupling, the ordering of $X^n$ and $Y^n$ is thus preserved at all times, and
if $X_s^n = Y_s^n$ for some $s\ge 0$, then it remains true at all times
$t\ge s$. Note however that if $X_s^{n,i} = Y_s^{n,i}$, then this equality
does not remain true at all times except if all the coordinates match.

As in \eqref{eq:TVW}, the total variation distance between the laws of $X_t^n$
and $Y_t^n$ may be bounded by
\[
  \Vert \mathrm{Law}(Y_t^n) -\mathrm{Law}(X_t^n) \Vert_{\mathrm{TV}}\leq \mathbb{P}( X_t^n \neq Y_t^n),
\]
for all $t\geq 0$.
We wish to establish that for any given $t>0$,
\begin{equation*}
    \lim_{n\to \infty}\mathbb{P}(X_t^n\neq Y_t^n)=0.
\end{equation*}
To do so, we work with the \textit{area} between the two processes $X^n$ and $Y^n$, defined by
\[
  A^n_t := \sum_{i=1}^n \big( Y^{n,i}_t - X^{n,i}_t \big) %
  = \pi(Y^n_t) - \pi(X^n_t),\quad t\ge 0.
\]
As the two processes are ordered at any time, this is nothing but the
geometric area between the two discrete interfaces $i\mapsto X_t^{n,i}$ and
$i\mapsto Y_t^{n,i}$ associated to the configurations $X_t^n$ and $Y_t^n$. We
deduce that the merging time of the two processes coincide with the hitting
time of $0$ by this area, that we denote by $\tau=\inf \{t\geq 0:A_t^n=0\}$.

The process $A^n$ has a very simple structure: it is a semimartingale that
behaves like an OU process with a randomly varying quadratic variation. Let
$N_t$ be the number of coordinates that do not coincide at time $t$, that is
\[
  N_t := \#\big\{i\in\{1,\ldots,n\}: X^{n,i}_t \ne Y^{n,i}_t\big\}.
\]
Then $A^n$ satisfies 
\begin{equation*}
  \mathrm{d}A^n_t = -A^n_t \mathrm{d}t + \mathrm{d}M_t,
\end{equation*}
where $M$ is a centered martingale with quadratic variation
\begin{equation}\label{eq:quadra varia M}
  \mathrm{d}\langle M\rangle_t = \frac2{n} N_t \mathrm{d}t.
 \end{equation}
Note that whenever $t< \tau$ we have
\[
  \mathrm{d}\langle M\rangle_t \ge \frac2{n}.
\]
This \emph{a priori} lower bound on the quadratic variation of $M$, combined
with the Dubins--Schwarz theorem, allows to check that $\tau < \infty$ almost
surely. Note that in view of the coupling between $X_t^n$ and $Y_t^n$, we have
$X_t^n=Y_t^n$ for all $t\geq \tau$.\\

Recall the following informal fact: with large probability, a Brownian motion starting from
$a$ hits $b$ by a time of order $(a-b)^2$. For a continuous martingale, this
becomes: with large probability, a continuous martingale starting from $a$
accumulates a quadratic variation of order $(a-b)^2$ up to its first hitting
time of $b$. Our next lemma states such a bound on the supermartingale $A^n$.

\begin{lemma}\label{Lemma:supermgale}
 Let $a>b\ge 0$. Let $\tau_b=\inf\{t>0:A_t=b\}<\infty$ almost surely. Then, for all $u\ge 1$,
  \[
    \mathbb{P}(\langle A\rangle_{\tau_b} \ge (a-b)^2 u \mid A_0 = a) \le 4 u^{-1/2}.
  \]
\end{lemma}

\begin{proof}
  Without loss of generality one can assume that $A_0 = a$ almost surely.\\
  By Itô's formula, for all $\lambda \ge 0$, the process
  \[
    S_t = \exp\Bigr(-\lambda A_t - \frac{\lambda^2}{2} \langle A\rangle_t\Bigr),
  \]
  defines a submartingale (taking its values in $[0,1]$). Doob's stopping
  theorem yields
  \[
    \mathbb{E}[\mathrm{e}^{-\frac{\lambda^2}{2} \langle A\rangle_{\tau_b}}]
    =\mathrm{e}^{\lambda b} \mathbb{E}[S_{\tau_b}] \ge \mathrm{e}^{\lambda b}
    \mathbb{E}[S_0] = \mathrm{e}^{-\lambda(a-b)}.
  \]
  On the other hand, for $\lambda = 2 (a-b)^{-1} u^{-1/2}$, there holds
  \begin{align*}
    \mathbb{E}[\mathrm{e}^{-\frac{\lambda^2}{2} \langle A\rangle_{\tau_b}}]
    &\le \mathbb{P}\big(\langle
    A\rangle_{\tau_b} < (a-b)^2 u\big) + \mathrm{e}^{-\frac{\lambda^2}{2}(a-b)^2u}\, \mathbb{P}\big(\langle
    A\rangle_{\tau_b} \ge (a-b)^2 u\big)\\
    &\le 1 - (1-\mathrm{e}^{-\frac{\lambda^2}{2}(a-b)^2u})\mathbb{P}\big(\langle
    A\rangle_{\tau_b} \ge (a-b)^2 u\big)\\
    &\le 1 - \frac12 \mathbb{P}\big(\langle
    A\rangle_{\tau_b} \ge (a-b)^2 u\big).
  \end{align*}
  Consequently one deduces that
  \[
    \mathbb{P}(\langle A\rangle_{\tau_b} \ge (a-b)^2 u) \le
    2(1-\mathrm{e}^{-\lambda(a-b)}) \le 4u^{-1/2}.
  \]
\end{proof}

We are now ready to prove the following lemma:
\begin{lemma}\label{le:distance X Y}
  If $\kappa>\frac{3}{2}$, then for every sequence of times ${(t_n)}_n$ with
  $\varliminf_{n\to \infty}t_n>0$, we have
  \[
    \lim_{n\to \infty}%
    \sup_{x_0^n \in \overline{D}_n}%
    \Vert \mathrm{Law}(Y_{t_n}^n) -\mathrm{Law}(X_{t_n}^n) \Vert_{\mathrm{TV}}=0.
  \]
\end{lemma}

\begin{proof}[Proof of Lemma \ref{le:distance X Y}]
  Let ${(t_n)}_n$ be a sequence of times such that
  $\varliminf_{n\to \infty}t_n>0$. In view of the definition of
  $\mu_{{x_0^n}}$ and $\eta$, the initial area satisfies almost surely
  \[
    A_0^n\leq 4n^{1-\kappa}.
  \]
  According to Lemma \ref{Lemma:supermgale}, with a probability that goes to $1$, one has
  \[
    \langle A^n\rangle_{\tau}-\langle A^n\rangle_{0} < 16n^{2-2\kappa} \log
    n.
  \]
  On the other hand, by \eqref{eq:quadra varia M}, we have the following control on the
  quadratic variation:
  \[
    \langle A\rangle_{\tau}-\langle A\rangle_{0} \ge
    \frac2{n}\tau.
  \]
  One deduces that, with a probability that goes to $1$,
  \[
    \tau \le \frac{16}2 n^{3-2\kappa} \log n,
  \]
  and this quantity goes to $0$ as $n\to\infty$, whenever
  $\kappa>\frac{3}{2}$. Therefore for $\kappa>\frac{3}{2}$, there holds
  \begin{equation*}
    \lim_{n\to \infty} \sup_{x_0^n \in \overline{D}_n} \mathbb{P}(X_{t_n}^n\neq Y_{t_n}^n)=0,
  \end{equation*}
  thus concluding the proof of Lemma \ref{le:distance X Y}.
\end{proof}

\begin{proof}[Proof of Theorem \ref{th:DOUWTV} in TV and Hellinger]
  Let $\kappa>\frac{3}{2}$ and fix some initial condition
  $x_0^n \in \overline{D}_n$. By the triangle inequality for $\mathrm{TV}$,
  there holds
  \begin{equation}\label{eq:triangular ine}
    \Vert\mathrm{Law}(X_{t}^n)-P_n^\beta\Vert_{\mathrm{TV}}\leq \Vert\mathrm{Law}(Y_{t}^n)-P_n^\beta\Vert_{\mathrm{TV}}+ \Vert\mathrm{Law}(X_{t}^n)-\mathrm{Law}(Y_{t}^n)\Vert_{\mathrm{TV}}.
  \end{equation}
  Taking $t=t_n(1+\varepsilon)$ with
  $t_n=\log(\sqrt{n} |x_0^n|) \vee \log(n)$, one deduces from Lemma
  \ref{le:convergence regul} and the Pinsker inequality stated in Lemma
  \ref{le:distineqs} that the first term in the right-hand side of
  \eqref{eq:triangular ine} vanishes as $n$ tends to infinity. Meanwhile Lemma
  \ref{le:distance X Y} guaranties that the second term tends to $0$ as $n$
  tends to infinity. We also conclude using the comparison between TV and
  Hellinger (see Lemma \ref{le:distineqs}) that
  \[
    \lim_{n\to \infty}\mathrm{Hellinger}(\mathrm{Law}(X_{t_n}^n),P_n^\beta)=0.
  \]
\end{proof}

\subsection{Proof of Corollary \ref{cor:DOUWTV} in TV and Hellinger}

\begin{proof}[Proof of Corollary \ref{cor:DOUWTV} in TV and Hellinger]
  By Lemma \ref{le:distineqs} and the triangle inequality for $\mathrm{TV}$,
  we have
  \begin{align*}
    \sup_{x_0^n \in [-a_n,a_n]^n}\Vert\mathrm{Law}(X_{t}^n)-P_n^\beta\Vert_{\mathrm{TV}}
    &\le \sup_{x_0^n \in [-a_n,a_n]^n}\Vert\mathrm{Law}(Y_{t}^n)-\mathrm{Law}(X_{t}^n)\Vert_{\mathrm{TV}}\\
    &\quad+ \sup_{x_0^n \in [-a_n,a_n]^n} \sqrt{2\, \mathrm{Kullback}(\mathrm{Law}(Y_t^n)\mid P_n^\beta)}.
  \end{align*}
  Take $t=(1+\varepsilon) c_n$ with $c_n = \log(na_n)$. Lemmas
  \ref{le:convergence regul} and \ref{le:distance X Y}, combined with the assumption made on $(a_n)$, show that the two terms
  on the right-hand side vanish as $n\to\infty$. Using Lemma
  \ref{le:distineqs}, the same result holds for $\mathrm{Hellinger}$.\\
  On the other hand, take $x_0^{n,i} = a_n$ for all $i$ and note that $\pi(x_0^n) = na_n$ goes to $+\infty$ as $n\to\infty$. By Corollary \ref{cor:LowerBd} we find
  \[\lim_{n\to\infty}%
    \sup_{x^n_0 \in [-a_n,a_n]^n}
    \mathrm{dist}(\mathrm{Law}(X^n_{(1-\varepsilon)c_n})\mid P_n^\beta)
    = 1\;\]
    whenever $\mathrm{dist} \in \{\mathrm{TV},\mathrm{Hellinger}\}$.
\end{proof}

\subsection{Proof of Theorem \ref{th:DOUK}}
\label{se:DOUK}

\begin{proof}[Proof of Theorem \ref{th:DOUK}]

  \emph{Lower bound.} The contraction property provided by Lemma
  \ref{le:contraction} gives
  \begin{equation*}
    \mathrm{Kullback}(\mathrm{Law}(X_t^n)\mid P_n^\beta)
    \geq \mathrm{Kullback}( \mathrm{Law}(\pi(X_t^n))\mid P_n^\beta \circ \pi^{-1}).
  \end{equation*}
  By Theorem \ref{th:DOUOU} $P_n\circ\pi^{-1}=\mathcal{N}(0,1)$ and
  $Y=\pi(X^n)$ is an OU process weak solution of
  $Y_0=\pi(X^n_0)$ and $\mathrm{d}Y_t=\sqrt{2}\mathrm{d}B_t-Y_t\mathrm{d}t$.
  In particular for all $t\geq0$, $\mathrm{Law}(Y_t)$ is a mixture of Gaussian
  laws in the sense that for any measurable test function $g$ with polynomial
  growth,
  \[
    \mathbb{E}_{\mathrm{Law}(Y_t)}[g]
    =\mathbb{E}[g(Y_t)]=\mathbb{E}[G_t(Y_0)]
    \quad\text{where}\quad
    G_t(y)=\mathbb{E}_{\mathcal{N}(y\mathrm{e}^{-t},1-\mathrm{e}^{-2t})}[g].
  \]  
  Now we use (again) the variational formula used in the proof of Lemma
  \ref{le:contraction} to get
  \begin{equation*}
    \mathrm{Kullback}( \mathrm{Law}(\pi(X_t^n))\mid P_n^\beta \circ \pi^{-1})
    =\sup_{g}\{\mathbb{E}_{\mathrm{Law}(\pi(X_t^n))}[g]-\log\mathbb{E}_{\mathcal{N}(0,1)}[\mathrm{e}^g] \},
  \end{equation*}
  and taking for $g$ the linear function defined by $g(x)=\lambda x$ for all
  $x\in\mathbb{R}$ and for some $\lambda\neq 0$ yields
  \begin{equation*}
    \mathrm{Kullback}( \mathrm{Law}(\pi(X_t^n))\mid P_n^\beta \circ \pi^{-1})
    \geq\lambda\mathrm{e}^{-t}\sum_{i=1}^n \int x\mu_i(\mathrm{d}x)-\frac{\lambda^2}{2}.
  \end{equation*}
  Finally, by using the assumption on first moment and taking $\lambda$ small
  enough we get, for all $\varepsilon\in(0,1)$,
  \begin{equation*}
    \lim_{n\to \infty}\mathrm{Kullback}(
    \mathrm{Law}(\pi(X_{(1-\varepsilon)\log(n)}^n)\mid P_n^\beta \circ \pi^{-1})
    =+\infty,
  \end{equation*}

  \emph{Upper bound.} From Lemma \ref{le:expdec} we have, for all $t\geq0$,
  \begin{equation*}
    \mathrm{Kullback}(\mathrm{Law}(X_t^n)\mid P_n^\beta)
    \leq \mathrm{Kullback}(\mathrm{Law}(X_0^n)\mid P_n^\beta)\mathrm{e}^{-2t}.
  \end{equation*}
  Arguing like in the proof of Lemma \ref{le:convergence regul} and using the
  contraction property of $\mathrm{Kullback}$ provided by Lemma
  \ref{le:contraction} for the map $\Psi$ defined in \eqref{eq:Psi}, we can write the following
  decomposition
  \begin{align*}
    \mathrm{Kullback}(\mathrm{Law}(X_0^n)\mid P_n^\beta)
    &\leq\mathrm{Kullback}(\otimes_{i=1}^n\mu_i\mid P_{*n}^\beta)\\
    &=S(\otimes_{i=1}^n\mu_i)+\mathbb{E}_{\otimes_{i=1}^n\mu_i}[E]+\log C_{*n}^\beta\\
    &\leq\sum_{i=1}^nS(\mu_i)+\sum_{i\ne j}\iint\Phi \mathrm{d}\mu_i \otimes \mathrm{d}\mu_j
      +Cn^2.
  \end{align*}
  Combining \eqref{eq:logZn} with the assumptions on the $\mu_i$'s
  yields for some constant $C>0$
  \begin{equation*}
    \mathrm{Kullback}(\mathrm{Law}(X_0^n)\mid P_n^\beta)\leq Cn^2
  \end{equation*}
  and it follows finally that for all $\varepsilon\in(0,1)$,
  \begin{equation*}
    \lim_{n\to \infty}\mathrm{Kullback}(\mathrm{Law}(X_{(1+\varepsilon)\log(n)})\mid P_n^\beta)=0.
  \end{equation*}
\end{proof}

\section{Cutoff phenomenon for the DOU in Wasserstein}
\label{se:DOUW}

\subsection{Proofs of Theorem \ref{th:DOUWTV} and Corollary \ref{cor:DOUWTV}
  in Wasserstein}

Let ${(X_t)}_{t\geq0}$ be the DOU process. By Lemma \ref{le:expdec}, for all
$t\geq0$ and all initial conditions $X_0 \in \overline{D}_n$,
\begin{equation*}
  \mathrm{Wasserstein}^2(
  \mathrm{Law}(X_t),P_n^\beta)
  \leq\mathrm{e}^{-2t}
  \mathrm{Wasserstein}^2(\mathrm{Law}(X_0),P_n^\beta).
\end{equation*}
Suppose now that $\mathrm{Law}(X^n_0)=\delta_{x^n_0}$. Then the triangle inequality for the Wasserstein distance gives
\begin{equation*}
  \mathrm{Wasserstein}^2(\delta_{x^n_0},P_n^\beta) %
  = \int\left|x^n_0 - x \right|^2 P_n^\beta(\mathrm{d}x) %
  \le 2 |x^n_0|^2 +2 \int\left|x\right|^2P_n^\beta(\mathrm{d}x).
\end{equation*}
By Theorem \ref{th:DOUOU}, the mean at equilibrium of $|X_t^n|^2$ equals $1 + \frac{\beta}{2}(n-1)$ and therefore
\[
  \int\left|x\right|^2P_n^\beta(\mathrm{d}x)
  =1+\frac{\beta}{2}(n-1).
\]
We thus get
\[
  \mathrm{Wasserstein}^2(\mathrm{Law}(X^n_t),P_n^\beta)
  \leq 2(|x^n_0|^2+1+\frac{\beta}{2}(n-1))\mathrm{e}^{-2t}.
\]
Set $c_n := \log(|x_0^n|) \vee \log(\sqrt{n})$. For any $\varepsilon \in (0,1)$, we have
\[
  \lim_{n\to\infty} \mathrm{Wasserstein}(\mathrm{Law}(X^n_{(1+\varepsilon)c_n}),P_n^\beta)
  =0
\]
and this concludes the proof of Theorem \ref{th:DOUWTV} in the Wasserstein distance.\\
Regarding the proof of Corollary \ref{cor:DOUWTV}, if $x_0^n \in [-a_n,a_n]^n$
then $|x_0^n| \le \sqrt{n} a_n$. Therefore if $\inf_n a_n > 0$, setting
$c_n = \log(\sqrt{n} a_n)$ we find, as required,
\[
  \lim_{n\to\infty}%
  \sup_{x_0^n \in [-a_n,a_n]^n}%
  \mathrm{Wasserstein}(\mathrm{Law}(X^n_{(1+\varepsilon)c_n}),P_n^\beta)%
  =0.
\]

\subsection{Proof of Theorem \ref{th:DOUW}}

This is an adaptation of the previous proof. We compute
\begin{align*}
  \mathrm{Wasserstein}^2(\delta_{x_0^n},P_n^\beta)
  &= \int\left|x^n_0 - x \right|^2P_n^\beta(\mathrm{d}x)\\
  &\le 2\left|x^n_0 - \rho_n\right|^2 + 2\int\left|\rho_n -x \right|^2P_n^\beta(\mathrm{d}x),
\end{align*}
where $\rho_n \in D_n$ is the vector of the quantiles of order $1/n$ of the
semi-circle law as in \eqref{eq:qn}. The rigidity estimates established in
\cite[Th.~2.4]{MR3253704} justify that
\begin{equation*}\label{Eq:Rigidity}
  \lim_{n\to\infty}\int\left|\rho_n -x \right|^2P_n^\beta(\mathrm{d}x)=0.
\end{equation*}
If $|x^n_0-\rho_n|$ diverges with $n$, we deduce that for all
$\varepsilon\in (0,1)$, with $t_n = \log(|x^n_0-\rho_n|)$,
\[
  \lim_{n\to\infty}
  \mathrm{Wasserstein}(\mathrm{Law}(X^n_{(1+\varepsilon)t_n}),P_n^\beta)
  =0.
\]
On the other hand, if $|x^n_0-\rho_n|$ converges to some limit $\alpha$ then
we easily get, for any $t\ge 0$,
\[
  \varlimsup_{n\to\infty} \mathrm{Wasserstein}^2(\mathrm{Law}(X^n_{t}),P_n^\beta) %
  \le \alpha^2\mathrm{e}^{-2t}.
\]

\begin{remark}[High-dimensional phenomena]
  With $X_n\sim P_n^\beta$, in the bias-variance decomposition
  \[
    \int\left|\rho_n -x \right|^2P_n^\beta(\mathrm{d}x)
    =|\mathbb{E}X_n-\rho_n|^2+\mathbb{E}(|X_n-\mathbb{E}X_n|^2),
  \]
  the second term of the right hand side is a variance term that measures the
  concentration of the log-concave random vector $X_n$ around its mean
  $\mathbb{E}X_n$, while the first term in the right hand side is a bias term
  that measures the distance of the mean $\mathbb{E}X_n$ to the mean-field
  limit $\rho_n$. Note also that
  $\mathbb{E}(|X_n-\mathbb{E}X_n|^2)=\mathbb{E}(|X_n|^2)-|\mathbb{E}X_n|^2=1+\frac{\beta}{2}(n-1)-|\mathbb{E}X_n|^2$,
	reducing the problem to the mean. We refer to \cite{zbMATH02165102} for a fine asymptotic analysis in the determinantal case $\beta=2$.
\end{remark}

\appendix

\section{Distances and divergences}
\label{ap:distances}

We use the following standard distances and divergences to quantify the trend
to equilibrium of Markov processes and to formulate the cutoff phenomena.

The \emph{Wasserstein--Kantorovich--Monge transportation distance} of order
$2$ and with respect to the underlying Euclidean distance is defined for all
probability measures $\mu$ and $\nu$ on $\mathbb{R}^n$ by
\begin{equation}
  \mathrm{Wasserstein}(\mu,\nu)
  =\Bigr(\inf_{(X,Y)}\mathbb{E}[\left|X-Y\right|^2]\Bigr)^{1/2}
  \in[0,+\infty]
\end{equation}
where $|x|=\sqrt{x_1^2+\cdots+x_n^2}$ and where the inf runs over all couples
$(X,Y)$ with $X\sim\mu$ and $Y\sim\nu$.

\medskip

The \emph{total variation distance} between probability measures $\mu$ and
$\nu$ on the same space is
\begin{equation}
  \left\Vert\mu-\nu\right\Vert_{\mathrm{TV}}
  =\sup_A|\mu(A)-\nu(A)|\in[0,1]
\end{equation}
where the supremum runs over Borel subsets. If $\mu$ and $\nu$ are absolutely
continuous with respect to a reference measure $\lambda$ with densities
$f_\mu$ and $f_\nu$ then
$\|\mu-\nu\|_{\mathrm{TV}}=\frac{1}{2}\int|f_\mu-f_\nu|\mathrm{d}\lambda=\frac{1}{2}\|f_\mu-f_\nu\|_{L^1(\lambda)}$.

\medskip

The \emph{Hellinger distance} between probability measures $\mu$ and $\nu$
with densities $f_\mu$ and $f_\nu$ with respect to the same reference measure
$\lambda$ is
\begin{equation}
  \mathrm{Hellinger}(\mu,\nu)
  =\Bigr(\int\frac{1}{2}(\sqrt{f_\mu}-\sqrt{f_\nu})^2\mathrm{d}\lambda\Bigr)^{1/2}
  =\Bigr(1-\int\sqrt{f_\mu f_\nu}\mathrm{d}\lambda\Bigr)^{1/2}
  \in[0,1].
\end{equation}
This quantity does not depend on the choice of $\lambda$. We have
$\mathrm{Hellinger}(\mu,\nu)=\frac{1}{\sqrt{2}}\Vert\sqrt{f_\mu}-\sqrt{f_\nu}\Vert_{L^2(\lambda)}$. Note that an alternative normalization is sometimes considered in the literature, making the maximal value of the Hellinger distance equal $\sqrt{2}$.

\medskip

The \emph{Kullback--Leibler divergence or relative entropy} is defined by
\begin{equation}
  \mathrm{Kullback}(\nu\mid\mu)
  =\int\log{\textstyle\frac{\mathrm{d}\nu}{\mathrm{d}\mu}}\mathrm{d}\nu
  =\int{\textstyle\frac{\mathrm{d}\nu}{\mathrm{d}\mu}}
  \log{\textstyle\frac{\mathrm{d}\nu}{\mathrm{d}\mu}}\mathrm{d}\mu
  \in[0,+\infty]
\end{equation}
if $\nu$ is absolutely continuous with respect to $\mu$, and $\mathrm{Kullback}(\nu\mid\mu)=+\infty$
otherwise.

\medskip

The \emph{$\chi^2$ divergence} is given by
\begin{equation}
  \chi^2(\nu\mid\mu)=
  \left\| \textstyle\frac{\mathrm{d}\nu}{\mathrm{d}\mu}
    -1\right\|_{L^2(\mu)}^2
  = \int \left|\frac{\mathrm{d}\nu}{\mathrm{d}\mu} -1\right|^2\mathrm{d}\mu
  =\|\frac{\mathrm{d}\nu}{\mathrm{d}\mu}\|_{L^2(\mu)}^2-1
  \in[0,+\infty].
\end{equation}
We set it to $+\infty$ if $\nu$ is not absolutely continuous with respect to
$\mu$. If $\mu$ and $\nu$ have densities $f_\mu$ and $f_\nu$ with respect to a
reference measure $\lambda$ then
$\chi^2(\nu\mid\mu)=\int(f_\nu^2/f_\mu)\mathrm{d}\lambda-1$.

\medskip

The (logarithmic) \emph{Fisher information or divergence} is defined by 
\begin{equation}
  \mathrm{Fisher}(\nu\mid\mu)
  =\int\left|\nabla\log\textstyle\frac{\mathrm{d}\nu}{\mathrm{d}\mu}\right|^2\mathrm{d}\nu
  =\int\frac{|\nabla\frac{\mathrm{d}\nu}{\mathrm{d}\mu}|^2}{\frac{\mathrm{d}\nu}{\mathrm{d}\mu}}\mathrm{d}\mu
  =4\int\left|\nabla\sqrt{\textstyle\frac{\mathrm{d}\nu}{\mathrm{d}\mu}}\right|^2\mathrm{d}\mu
  \in[0,+\infty]
\end{equation}
if $\nu$ is absolutely continuous with respect to $\mu$, and
$\mathrm{Fisher}(\nu\mid\mu)=+\infty$ otherwise.

Each of these distances or divergences has its advantages and drawbacks. In
some sense, the most sensitive is Fisher due to its Sobolev nature, then
$\chi^2$, then Kullback which can be seen as a sort of $L^{1+}=L\log L$ norm,
then TV and Hellinger which are comparable, then Wasserstein, but this rough
hierarchy misses some subtleties related to some scales and nature of the
arguments.

Some of these distances or divergences can generically be compared as the
following result shows.

\begin{lemma}[Inequalities]\label{le:distineqs}
  For any probability measures $\mu$ and $\nu$ on the same space,
  \begin{align*}
    \left\Vert\mu-\nu\right\Vert_{\mathrm{TV}}^2
    &\leq2\mathrm{Kullback}(\nu\mid\mu)\\
    2\mathrm{Hellinger}^2(\mu,\nu)
    &\leq\mathrm{Kullback}(\nu\mid\mu)\\
    \mathrm{Kullback}(\nu\mid\mu)
    &\leq 2\chi(\nu\mid\mu)+\chi^2(\nu\mid\mu)\\
    \mathrm{Hellinger}^2(\mu,\nu)\leq\|\mu-\nu\|_{\mathrm{TV}}
    &\leq \mathrm{Hellinger}(\mu,\nu)\sqrt{2-\mathrm{Hellinger}(\mu,\nu)^2}.
  \end{align*}
\end{lemma}

We refer to \cite[p.~61-62]{MR1873379} for a proof. The inequality between the
total variation distance and the relative entropy is known as the Pinsker
or Csiszár--Kullback inequality, while the inequalities between the total
variation distance and the Hellinger distance are due to Kraft. There are
many other metrics between probability measures, see for instance
\cite{zbMATH00049698,zbMATH02124714} for a discussion.

\medskip

The total variation distance can also be seen as a special Wasserstein
distance of order $1$ with respect to the atomic distance, namely
\begin{equation}\label{eq:TVW}
  \left\|\mu-\nu\right\|_{\mathrm{TV}}=\inf_{(X,Y)}\mathbb{P}(X\neq
  Y)=\inf_{(X,Y)}\mathbb{E}[\mathbf{1}_{X\neq Y}]\in[0,1]
\end{equation}
where the infimum runs over all couplings $X\sim\mu$ and $Y\sim\nu$. This
explains in particular why $\mathrm{TV}$ is more sensitive than
$\mathrm{Wasserstein}$ at short scales but less sensitive at large scales, a
consequence of the sensitivity difference between the underlying atomic and
Euclidean distances. The probabilistic representations of $\mathrm{TV}$ and
$\mathrm{Wasserstein}$ make them compatible with techniques
of coupling, which play an important role in the literature on convergence to
equilibrium of Markov processes.

We gather now useful results on distances and divergences.

\begin{lemma}[Contraction properties]\label{le:contraction}
  Let $\mu$ and $\nu$ be two probability measures on a same measurable space $S$. Let $f:S\mapsto T$ be a measurable function, where $T$ is another measurable space.
  \begin{itemize}
  \item If $\mathrm{dist}\in\{\mathrm{TV},\mathrm{Kullback},\chi^2\}$ then
  \[
  	\mathrm{dist}(\nu \circ f^{-1}\mid \mu \circ f^{-1})
  	\leq \mathrm{dist}(\nu \mid \mu).
  \]
  \item If $S=\mathbb{R}^n$, $T=\mathbb{R}^k$ then, denoting
    $\left\|f\right\|_{\mathrm{Lip}}=\sup_{x\neq y}\frac{|f(x)-f(y)|}{|x-y|}$,
    \[
      \mathrm{Wasserstein}(\mu\circ f^{-1},\nu\circ f^{-1})
      \leq\left\|f\right\|_{\mathrm{Lip}}\mathrm{Wasserstein}(\mu,\nu).
    \]
  \end{itemize}
\end{lemma}

The notation $f^{-1}$ stands for the reciprocal map
$f^{-1}(A)=\{y\in S:f(x)\in A\}$ and $\mu\circ f^{-1}$ is the image measure or
push-forward of $\mu$ by the map $f$, defined by
$(\mu\circ f^{-1})(A)=\mu(f^{-1}(A))$. In terms of random variables we have
$Y\sim\mu\circ f^{-1}$ if and only $Y=f(X)$ where $X\sim\mu$.

The proof of the contraction properties of Lemma \ref{le:contraction} are all
based on variational formulas. Note that following
\cite[Ex.~22.20~p.~588]{zbMATH05306371}, there is a variational formula for
$\mathrm{Fisher}$ that comes from its dual representation as an inverse Sobolev
norm. We do not develop this idea in this work.

\begin{proof}
  The proof of the contraction property for Wasserstein comes from the fact
  that every coupling of $\mu$ and $\nu$ produces a coupling for
  $\mu\circ f^{-1}$ and $\nu\circ f^{-1}$. Regarding TV, the contraction
  property is a consequence of the definition of this distance and of
  measurability. In the case of Kullback, the property can be proved using the
  following well known variational formula:
  \[
    \mathrm{Kullback}(\nu\mid\mu)
    =\sup_{g}\{\mathbb{E}_{\nu }[g]-\log\mathbb{E}_{\mu}[\mathrm{e}^g] \}
  \]
  where the supremum runs over all $g\in L^1(\nu)$, or by approximation when
  the supremum runs over all bounded measurable $g$. This variational formula
  can be derived for instance by applying Jensen's inequality to
  $- \log \mathbb{E}_{\nu}[\mathrm{e}^g \frac{\mathrm{d}\mu}{\mathrm{d}\nu}]$.
  Equality is achieved for
  $g=\log(\mathrm{d}\nu/\mathrm{d}\mu)$. Now, taking $g=h\circ f$ gives
  \[
    \mathrm{Kullback}(\nu \mid \mu)
    \geq \mathbb{E}_{\nu \circ f^{-1} }[h]-\log\mathbb{E}_{\mu \circ f^{-1} }[\mathrm{e}^h],  
  \]
  and it remains to take the supremum over $h$ to get 
  \[
    \mathrm{Kullback}(\nu\mid\mu)
    \geq \mathrm{Kullback}(\nu \circ f^{-1} \mid \mu \circ f^{-1}).
  \]
  The variational formula for $\mathrm{Kullback}(\cdot\mid\mu)$ is a
  manifestation of its convexity, it expresses this functional as the envelope
  of its tangents, its Fenchel--Legendre transform or convex dual is the
  log-Laplace transform. Such a variational formula is equivalent to
  tensorization, and is available for all $\Phi$-entropies such that
  $(u,v)\mapsto\Phi''(u)v^2$ is convex, see \cite[Th.~4.4]{zbMATH05216861}. In
  particular, the analogous variational formula as well as the consequence in
  terms of contraction are also available for $\chi^2$ which corresponds to
  the $\Phi$-entropy with $\Phi(u)=u^2$ (variance as a $\Phi$-entropy).
\end{proof}

\begin{lemma}[Scale invariance versus homogeneity]\label{le:scale}
  The total variation distance is scale invariant while the Wasserstein
  distance is homogeneous just like a norm, namely for all probability
  measures $\mu$ and $\nu$ on $\mathbb{R}^n$ and all scaling factor
  $\sigma\in(0,\infty)$, denoting $\mu_\sigma=\mathrm{Law}(\sigma X)$ where
  $X\sim\mu$, we have
  \[
    \|\mu_\sigma-\nu_\sigma\Vert_{\mathrm{TV}}=\|\mu-\nu\Vert_{\mathrm{TV}}%
    \quad\text{while}\quad%
    \mathrm{Wasserstein}(\mu_\sigma,\nu_\sigma)=\sigma\mathrm{Wasserstein}(\mu,\nu).
  \]
\end{lemma}

\begin{proof}
  For the Wasserstein distance, the result follows from
  \[
    \mathrm{Wasserstein}(\mu_\sigma,\nu_\sigma) = \Bigr(\inf_{(X,Y)}\mathbb{E}[\left|\sigma X-\sigma Y\right|^2]\Bigr)^{1/2} = \sigma \mathrm{Wasserstein}(\mu,\nu)\;,
  \]
  while for the $\mathrm{TV}$ distance, it comes from the fact that $A\mapsto A_\sigma := \{\sigma x: x\in A\}$ is a bijection.
\end{proof}
We turn to the behavior of the distances/divergences under tensorization.

\begin{lemma}[Tensorization]\label{le:tenso}
  For all probability measures $\mu_1,\ldots,\mu_n$ and $\nu_1,\ldots,\nu_n$
  on $\mathbb{R}$, we have
  \begin{align*}
    \mathrm{Hellinger}^2(\otimes_{i=1}^n\mu_i,\otimes_{i=1}^n\nu_i)
    &=1-\prod_{i=1}^n\Bigr(1-\mathrm{Hellinger}^2(\mu_i,\nu_i)\Bigr),\\
    \mathrm{Kullback}(\otimes_{i=1}^n\nu_i\mid\otimes_{i=1}^n\mu_i)
    &=\sum_{i=1}^n\mathrm{Kullback}(\nu_i\mid\mu_i),\\
    \chi^2(\otimes_{i=1}^n\mu_i\mid\otimes_{i=1}^n\nu_i)
    &=-1+\prod_{i=1}^n(\chi^2(\mu_i,\nu_i)+1),\\
    \mathrm{Fisher}(\otimes_{i=1}^n\nu_i\mid\otimes_{i=1}^n\mu_i)
    &=\sum_{i=1}^n\mathrm{Fisher}(\nu_i\mid\mu_i),\\
    \mathrm{Wasserstein}^2(\otimes_{i=1}^n\mu_i,\otimes_{i=1}^n\nu_i)
    &=\sum_{i=1}^n\mathrm{Wasserstein}^2(\mu_i,\nu_i),\\
    \max_{1\leq i\leq n}\Vert\mu_i-\nu_i\Vert_{\mathrm{TV}}
    \leq \Vert\otimes_{i=1}^n\mu_i-\otimes_{i=1}^n\nu_i\Vert_{\mathrm{TV}}%
    &\leq \sum_{i=1}^n\Vert\mu_i-\nu_i\Vert_{\mathrm{TV}}.
  \end{align*}
\end{lemma}

The equality for the Wasserstein distance comes by taking the product of
optimal couplings. The first inequality for the total variation distance comes
from its contraction property (Lemma \ref{le:contraction}), while the second
comes from
$|(a_1\cdots a_n)-(b_1\cdots b_n)|\leq\sum_{i=1}^n|a_i-b_i|(a_1\cdots
a_{i-1})(b_{i+1}\cdots b_n)$, $a_1,\ldots,a_n,b_1,\ldots,b_n\in[0,+\infty)$,
which comes itself from the triangle inequality on the telescoping sum
$\sum_{i=1}^n(c_i-c_{i-1})$ where $c_i=(a_1\cdots a_i)(b_{i+1}\cdots b_n)$ via
$c_i-c_{i-1}=(a_i-b_i)(a_1\cdots a_{i-1})(b_{i+1}\cdots b_n)$.

\begin{lemma}[Explicit formulas for Gaussian distributions]\label{le:gauss}
  For all $n\geq1$, $m_1,m_2\in\mathbb{R}^n$, and all $n\times n$ covariance
  matrices $\Sigma_1,\Sigma_2$, denoting $\Gamma_1=\mathcal{N}(\mu_1,\Sigma_1)$
  and $\Gamma_2=\mathcal{N}(\mu_2,\Sigma_2)$, we have
  \begin{align*}
    \mathrm{Hellinger}^2(\Gamma_1,\Gamma_2)
    &=1-\frac{\det(\Sigma_1\Sigma_2)^{1/4}}{\det(\frac{\Sigma_1+\Sigma_2}{2})^{1/2}}
      \mathrm{e}^{-\frac{1}{4}(\Sigma_1+\Sigma_2)^{-1}(m_2-m_1)\cdot(m_2-m_1)},\\
    2\mathrm{Kullback}(\Gamma_1\mid\Gamma_2)
    &=\Sigma_2^{-1}(m_1-m_2)\cdot(m_1-m_2)+\mathrm{Tr}(\Sigma_2^{-1}\Sigma_1-\mathrm{I}_n)+\log\det(\Sigma_2\Sigma_1^{-1}),\\
    \chi^2(\Gamma_1\mid\Gamma_2)
    &=-1+
      \frac{\det(\Sigma_2)}{\sqrt{\det(\Sigma_1)\det(2\Sigma_2-\Sigma_1)}}%
      \mathrm{e}^{\frac{1}{2}(\Sigma_2^{-1}+(2\Sigma_2\Sigma_1^{-1}\Sigma_2-\Sigma_2)^{-1})(m_2-m_1)\cdot(m_2-m_1)},\\
    \mathrm{Fisher}(\Gamma_1\mid\Gamma_2)
    &=|\Sigma_2^{-1}(m_1-m_2)|^2+\mathrm{Tr}(\Sigma_2^{-2}\Sigma_1-2\Sigma_2^{-1}+\Sigma_1^{-1})\\
    \mathrm{Wasserstein}^2(\Gamma_1,\Gamma_2)
    &=|m_1-m_2|^2+\mathrm{Tr}\Bigr(\Sigma_1+\Sigma_2-2\sqrt{\sqrt{\Sigma_1}\Sigma_2\sqrt{\Sigma_1}}\Bigr),
  \end{align*}
  where the formula for $\chi^2(\Gamma_1\mid\Gamma_2)$ holds if
  $2\Sigma_2>\Sigma_1$, and $\chi^2(\Gamma_1\mid\Gamma_2)=+\infty$ otherwise.
  Moreover the formulas for Fisher and Wasserstein rewrite, if $\Sigma_1$ and
  $\Sigma_2$ commute, $\Sigma_1\Sigma_2=\Sigma_2\Sigma_1$, to
    \begin{align*}
      \mathrm{Fisher}(\Gamma_1\mid\Gamma_2)
      &=|\Sigma_2^{-1}(m_1-m_2)|^2+\mathrm{Tr}(\Sigma_2^{-2}(\Sigma_2-\Sigma_1)^2\Sigma_1^{-1})\\
      \mathrm{Wasserstein}^2(\Gamma_1,\Gamma_2)
      &=|m_1-m_2|^2+\mathrm{Tr}((\sqrt{\Sigma_1}-\sqrt{\Sigma_2})^2).
    \end{align*}
\end{lemma}

Regarding the total variation distance, there is no general simple formula for
Gaussian laws, but we can use for instance the comparisons with
$\mathrm{Kullback}$ and $\mathrm{Hellinger}$ (Lemma \ref{le:distineqs}), see
\cite{devroye} for a discussion.

\begin{proof}[Proof of Lemma \ref{le:gauss}]
  We refer to \cite[p.~47~and~p.~51]{MR2183173} for Kullback and Hellinger,
  and to \cite{MR752258} for Wasserstein, a far more subtle case. The formula
  for $\chi^2(\Gamma_1\mid\Gamma_2)$ follows easily from a direct computation.
  We have not found in the literature a formula for Fisher. Let us give it
  here for the sake of completeness. Using
  $\mathbb{E}[X_iX_j]=\Sigma_{ij}+m_im_j$ when $X\sim\mathcal{N}(m,\Sigma)$ we
  get, for all $n\times n$ symmetric matrices $A$ and $B$
  \begin{equation*}
    \mathbb{E}[AX\cdot BX]
    =\sum_{i,j,k=1}^nA_{ij}B_{ik}\mathbb{E}[X_jX_k]
    =\sum_{i,j,k=1}^nA_{ij}B_{ik}(\Sigma_{jk}+m_jm_k)
    =\mathrm{Trace}(A\Sigma B)+Am\cdot Bm
  \end{equation*}
  and thus for all $n$-dimensional vectors $a$ and $b$, 
  \begin{align*}
    \mathbb{E}[A(X-a)\cdot B(X-b)]
    &=\mathbb{E}[AX\cdot BX]+A(m-a)\cdot B(m-b)-Am\cdot Bm\\
    &=\mathrm{Trace}(A\Sigma B)+A(m-a)\cdot B(m-b).
  \end{align*}
  Now, using the notation $q_i(x)=\Sigma_i^{-1}(x-m_i)\cdot(x-m_i)$ and
  $|\Sigma_i|=\det(\Sigma_i)$,
  \begin{align*}
    \mathrm{Fisher}(\Gamma_1\mid\Gamma_2)
    &=4\frac{\sqrt{|\Sigma_2|}}{\sqrt{|\Sigma_1|}}
      \int\Bigr|\nabla\mathrm{e}^{-\frac{q_1(x)}{4}+\frac{q_2(x)}{4}}\Bigr|^2\frac{\mathrm{e}^{-\frac{q_2(x)}{2}}}{\sqrt{2\pi|\Sigma_2|}}\mathrm{d}x\\
    &=\int|\Sigma_2^{-1}(x-m_2)-\Sigma_1^{-1}(x-m_1)|^2
      \frac{\mathrm{e}^{-\frac{q_1(x)}{2}}}{\sqrt{2\pi|\Sigma_1|}}\mathrm{d}x\\
    &=\int(|\Sigma_2^{-1}(x-m_2)|^2-2\Sigma_2^{-1}(x-m_2)\cdot\Sigma_1^{-1}(x-m_1)+|\Sigma_1^{-1}(x-m_1)|^2)
      \frac{\mathrm{e}^{-\frac{q_1(x)}{2}}}{\sqrt{2\pi|\Sigma_1|}}\mathrm{d}x\\  
    &=\mathrm{Trace}(\Sigma_2^{-1}\Sigma_1\Sigma_2^{-1})+|\Sigma_2^{-1}(m_1-m_2)|^2-2\mathrm{Trace}(\Sigma_2^{-1})+\mathrm{Trace}(\Sigma_1^{-1})\\
    &=\mathrm{Trace}(\Sigma_2^{-2}\Sigma_1-2\Sigma_2^{-1}+\Sigma_1^{-1})+|\Sigma_2^{-1}(m_1-m_2)|^2.    
  \end{align*}
  The formula when $\Sigma_1\Sigma_2=\Sigma_2\Sigma_1$ follows immediately. 
\end{proof}

\section{Convexity and its dynamical consequences}
\label{ap:convexity}

We gather useful dynamical consequences of convexity. We start with
functional inequalities.

\begin{lemma}[Logarithmic Sobolev inequality]\label{le:lsi}
  Let $P_n^\beta$ be the invariant law of the DOU process
  solving \eqref{eq:DOU}. Then, for all law $\nu$ on $\mathbb{R}^n$, we have
  \[
    \mathrm{Kullback}(\nu\mid P_n^\beta)
    \leq\frac{1}{2n}\mathrm{Fisher}(\nu\mid P_n^\beta).
  \]
  Moreover the constant $\frac{1}{2n}$ is optimal.\\
  Furthermore, finite equality is achieved if and only if
  $\mathrm{d}\nu/\mathrm{d}P_n^\beta$ is of the form
  $\mathrm{e}^{\lambda(x_1+\cdots+x_n)}$, $\lambda\in\mathbb{R}$.\\
\end{lemma}

Linearizing the log-Sobolev inequality above with
$\mathrm{d}\nu/\mathrm{d}P_n^\beta=1+\varepsilon f$ gives the Poincaré
inequality
\begin{equation}\label{eq:PI}
  \mathrm{Var}_{P^\beta_n}(f)\leq-\int f\G f\mathrm{d}P^\beta_n.
\end{equation}
It can be extended by truncation and regularization from the case where $f$
is smooth and compactly supported to the case where $f$ is in the Sobolev
space $H^1(P^\beta_n)$. Finite equality is achieved when $f$ is an
eigenfunction associated to the eigenvalue $-1$ of $\G$, namely
$f(x)=a(x_1+\cdots+x_n)+b$, $a,b\in\mathbb{R}$, hence the other name
\emph{spectral gap inequality}. It rewrites in terms of $\chi^2$ divergence
as
\begin{equation}\label{eq:PIchi2}
  \chi^2(\nu\mid P_n^\beta)\leq\frac{1}{n}\int\Bigr|\nabla\frac{\mathrm{d}\nu}{\mathrm{d}P^\beta_n}\Bigr|^2\,\mathrm{d}P_n^\beta.
\end{equation}
The right-hand side plays for the $\chi^2$ divergence the role played by
Fisher for Kullback. 

We refer to \cite{MR3699468,zbMATH07238061} for a proof of Lemma \ref{le:lsi}.
This logarithmic Sobolev inequality is a consequence of the log-concavity of
$P_n^\beta$ with respect to $\mathcal{N}(0,\frac{1}{n}\mathrm{I}_n)$. A slightly delicate aspect
lies in the presence of the restriction to $D_n$, which can be circumvented by
using a regularization procedure.

There are many other functional inequalities which are a consequence of this
log-concavity, for instance the Talagrand transportation inequality that
states that when $\nu$ has finite second moment,
\[
  \mathrm{Wasserstein}^2(\nu,P_n^\beta)
  \leq\frac{1}{n}\mathrm{Kullback}(\nu\mid P_n^\beta)
\]
and the HWI inequality\footnote{Here ``H'' is the capital $\eta$ used by
  Boltzmann for entropy, ``W'' is for Wasserstein, ``I'' is for Fisher
  information.} that states that when $\nu$ has finite second moment,
\[
  \mathrm{Kullback}(\nu\mid P_n^\beta)
  \leq\mathrm{Wasserstein}(\nu, P_n^\beta)\sqrt{\mathrm{Fisher}(\nu\mid P_n^\beta)}
  -\frac{n}{2}\mathrm{Wasserstein}^2(\nu\mid P_n^\beta),
\]
and we refer to \cite{zbMATH05306371} 
for this couple of functional inequalities, that we do not use here.

\begin{lemma}[Sub-exponential convergence to equilibrium]\label{le:expdec}
  Let ${(X^n_t)}_{t\geq0}$ be the DOU process solution
  of \eqref{eq:DOU} with $\beta=0$ or $\beta\geq1$, and let $P_n^\beta$ be its
  invariant law. Then for all $t\geq0$, we have the sub-exponential
  convergences
  \begin{align*}
    \chi^2(\mathrm{Law}(X^n_t)\mid P_n^\beta)
    &\leq\mathrm{e}^{-2t}\chi^2(\mathrm{Law}(X^n_0)\mid P_n^\beta),\\
    \mathrm{Kullback}(\mathrm{Law}(X^n_t)\mid P_n^\beta)
    &\leq\mathrm{e}^{-2t}\mathrm{Kullback}(\mathrm{Law}(X^n_0)\mid P_n^\beta),\\
    \mathrm{Fisher}(\mathrm{Law}(X^n_t)\mid P_n^\beta)
    &\leq\mathrm{e}^{-2t}\mathrm{Fisher}(\mathrm{Law}(X^n_0)\mid P_n^\beta),\\
    \mathrm{Wasserstein}^2(\mathrm{Law}(X^n_t),P_n^\beta)
    &\leq\mathrm{e}^{-2t}\mathrm{Wasserstein}^2(\mathrm{Law}(X^n_0),P_n^\beta).
  \end{align*}
\end{lemma}

Recall that when $\beta>0$ the initial condition $X^n_0$ is always taken in
$D_n$.

For each inequality, if the right-hand side is infinite then the inequality is
trivially satisfied. This is in particular the case for $\mathrm{Kullback}$ and
$\mathrm{Fisher}$ when $\mathrm{Law}(X^n_0)$ is not absolutely continuous with
respect to the Lebesgue measure, and for Wasserstein when
$\mathrm{Law}(X^n_0)$ has infinite second moment.

\begin{proof}[Elements of proof of Lemma \ref{le:expdec}]
  The idea is that an exponential decay for $\mathrm{Kullback}$, $\chi^2$,
  $\mathrm{Fisher}$, and $\mathrm{Wasserstein}$ can be established by taking
  the derivative, using a functional inequality, and using the Grönwall lemma.
  More precisely, for $\mathrm{Kullback}$ it is a log-Sobolev inequality, for
  $\chi^2$ a Poincaré inequality, for $\mathrm{Wasserstein}$ a transportation
  type inequality, and for $\mathrm{Fisher}$ a Bakry\,--\,Émery $\Gamma_2$
  inequality, see for instance
  \cite{zbMATH01633816,zbMATH06175511,zbMATH05306371}. It is a rather standard
  piece of probabilistic functional analysis, related to the log-concavity of
  $P_n^\beta$. We recall the crucial steps for the reader convenience. Let us
  set $\mu_t=\mathrm{Law}(X^n_t)$ and $\mu=P_n^\beta$. For $t>0$ the density
  $p_t=\mathrm{d}\mu_t/\mathrm{d}\mu$ exists and solves the evolution equation
  $\partial_tp_t=\G p_t$ where $\G$ is as in \eqref{eq:G}. We have the
  integration by parts
  \[
    \int f\G g\mathrm{d}\mu
    =\int g\G f\mathrm{d}\mu
    =-\frac{1}{n}\int\nabla f\cdot\nabla g\mathrm{d}\mu.
  \]
  For $\mathrm{Kullback}$, we find using these tools, for all $t>0$, denoting
  $\Phi(u):=u\log(u)$,
    \begin{multline}\label{eq:debruijn}
      \partial_t\mathrm{Kullback}(\mu_t\mid\mu)
      =\int\Phi'(f_t)\G f_t\mathrm{d}\mu
      =-\frac{1}{n}\int\Phi''(f_t)|\nabla f_t|^2\mathrm{d}\mu\\
      =-\frac{1}{n}\mathrm{Fisher}(\mu_t\mid\mu)
      \leq-2\mathrm{Kullback}(\mu_t\mid\mu),
    \end{multline}
  where the inequality comes from the logarithmic Sobolev inequality of Lemma
  \ref{le:lsi}. It remains to use the Grönwall lemma to get the
  exponential decay of $\mathrm{Kullback}$.

  The derivation of the exponential decay of the Fisher divergence follows the
  same lines by differentiating again with respect to time. Indeed, after a
  sequence of differential computations and integration by parts, we find, see
  for instance \cite[Ch.~5]{zbMATH01633816}, \cite{zbMATH06175511}, or
  \cite{zbMATH05306371},
  \begin{equation}\label{eq:dtF}
    \partial_t\mathrm{Fisher}(\mu_t\mid\mu)
    =-2n\int f_t\Gamma_{\!2}(\log(f_t))\mathrm{d}\mu,
  \end{equation}
  where $\Gamma_{\!2}(f):=\frac{1}{n^2}f''^2+\frac{1}{n}V''f'^2$ is the
  Bakry\,--\,Émery ``Gamma-two'' operator of the dynamics. Now using the
  convexity of $V$, we get, by the Grönwall lemma, for all $t>0$,
  \begin{equation}\label{Eq:DecayFisher}
  	    \partial_t\mathrm{Fisher}(\mu_t\mid\mu)
    \leq-2\mathrm{Fisher}(\mu_t\mid\mu).
  \end{equation}
  This can be used to prove the log-Sobolev inequality, see
  \cite[Ch.~5]{zbMATH01633816}, \cite{zbMATH06175511}, and
  \cite{zbMATH05306371}. This differential approach goes back at least to
  Boltzmann (statistical physics) and Stam (information theory) and was
  notably extensively developed later on by Bakry, Ledoux, Villani and their
  followers.

  For the Wasserstein distance, we proceed by coupling. Indeed, since the
  diffusion coefficient is constant in space, we can simply use a
  \emph{parallel coupling}. Namely, let ${(X'_t)}_{t\geq0}$ be the process
  started from another possibly random initial condition $X'_0$, and
  satisfying to the same stochastic differential equation, with the same
  BM. We get
  \[
    \mathrm{d}(X_t-X'_t)=-\frac1{n}(\nabla E(X_t)-\nabla E(X'_t))\mathrm{d}t,
  \]
  hence
  \begin{equation}\label{eq:dtW}
    \mathrm{d}(X_t-X'_t)\cdot(X_t-X'_t)
    =-\frac1n((\nabla E(X_t)-\nabla E(X'_t))\cdot(X_t-X'_t))\mathrm{d}t.
  \end{equation}
  Now since $E$ is uniformly convex with $\nabla^2 E\geq n I_n$,
  we get, for all $x,y\in\mathbb{R}^n$,
  \[
    (\nabla E(x)-\nabla E(y))\cdot(x-y)\geq n|x-y|^2,
  \]
which gives
  \[
    \mathrm{d}|X_t-X'_t|^2
    \leq-2|X_t-X'_t|^2\mathrm{d}t
  \]
  and by the Grönwall lemma,
  \[
    |X_t-X'_t|^2\leq\mathrm{e}^{-2 t}|X_0-X'_0|^2.
  \]
 It follows that
  \[
    \mathrm{Wasserstein}^2(
    \mathrm{Law}(X_t),\mathrm{Law}(X_t'))
    \leq\mathrm{e}^{-2 t}
    \mathbb{E}[|X_0-X'_0|^2].
  \]
  By taking the infimum over all couplings of $X_0$ and $X_0'$ we get
  \[
    \mathrm{Wasserstein}^2(
    \mathrm{Law}(X_t),\mathrm{Law}(X_t'))
    \leq\mathrm{e}^{-2 t}
    \mathrm{Wasserstein}^2(\mathrm{Law}(X_0),\mathrm{Law}(X'_0)).
  \]
  Taking $X_0'\sim P_n^\beta$ we get, by invariance, for all $t\geq0$,
  \[
    \mathrm{Wasserstein}^2(
    \mathrm{Law}(X_t),P_n^\beta)
    \leq\mathrm{e}^{-2 t}
    \mathrm{Wasserstein}^2(\mathrm{Law}(X_0),P_n^\beta).
  \]
\end{proof}

\begin{lemma}[Monotonicity]\label{le:monot}
  Let ${(X^n_t)}_{t\geq0}$ be the DOU process \eqref{eq:DOU}, with $\beta=0$ or
  $\beta\geq1$ and invariant law $P^\beta_n$. Then for all
  $\mathrm{dist}\in\{\mathrm{TV},\mathrm{Hellinger},\mathrm{Kullback},\chi^2,\mathrm{Fisher},\mathrm{Wasserstein}\}$,
  the function
  $t\geq0\mapsto\mathrm{dist}(\mathrm{Law}(X^n_t)\mid P^\beta_n)$ is
  non-increasing.
\end{lemma}

\begin{proof}[Elements of proof of Lemma \ref{le:monot}]
  The monotonicity for
  $\mathrm{TV},\mathrm{Hellinger},\mathrm{Kullback},\chi^2$ comes from the
  Markov nature of the process and the convexity of
  \[
    u\mapsto
    \Phi(u)=
    \begin{cases}
      \frac{1}{2}|u-1| & \text{if $\mathrm{dist}=\mathrm{TV}$}\\
      1-\sqrt{u} & \text{if $\mathrm{dist}=\mathrm{Hellinger}$}\\
      u\log(u) & \text{if $\mathrm{dist}=\mathrm{Kullback}$}\\
      u^2-1 & \text{if $\mathrm{dist}=\chi^2$}
    \end{cases}.
  \]
  This is known as the $\Phi$-entropy
  dissipation of Markov processes, see
  \cite{MR2081075,zbMATH05306371,zbMATH06175511}. This can also be seen from
  \eqref{eq:debruijn}. The monotonicity for $\mathrm{TV}$ follows also from
  the contraction property of the total variation with respect to general
  Markov kernels, see \cite[Ex.~4.2]{zbMATH06813269}.
  
  The monotonicity for $\mathrm{Fisher}$ comes from the identity
  \eqref{eq:dtF} and the convexity of $V$. By \eqref{eq:debruijn} this
  monotonicity is also equivalent to the convexity of $\mathrm{Kullback}$
  along the dynamics. The monotonicity for $\mathrm{Wasserstein}$ can be
  obtained by computing the derivative along the dynamics starting from
  \eqref{eq:dtW}, but this is more subtle due to the variational nature of this
  distance and involves the convexity of $V$, see for instance \cite[Bottom
  of p.~2442 and Lem.~3.2]{MR2964689}.
  
  The monotonicities can also be extracted from the exponential decays of
  Lemma \ref{le:expdec} thanks to the Markov property and the profile
  $\mathrm{e}^{-t}=1-t+o(t)$ of the prefactor in the right hand side.
\end{proof}

The convexity of the interaction $-\log$ as well as the
constant nature of the diffusion coefficient in the evolution equation
\eqref{eq:DOU} allows to use simple ``maximum principle'' type arguments to
prove that the dynamic exhibits a monotonous behavior and an exponential
decay.

\begin{lemma}[Monotonicity and exponential decay]\label{lemma:monotonicity}
  Let ${(X_t^n)}_{t\geq0}$ and ${(Y_t^n)}_{t\geq0}$ be a pair of DOU processes
  solving \eqref{eq:DOU}, $\beta\geq1$, driven by the same Brownian motion
  $(B_t)_{t\geq 0}$ on $\mathbb{R}^n$ and with respective initial conditions
  $X_0^n\in \overline{D}_n$ and $Y_0^n\in \overline{D}_n$. If for all
  $i\in \{1,\ldots,n\}$
  \[
    X_0^{n,i}\leq Y_0^{n,i}
  \]
  then the following properties hold true:
  \begin{itemize}
  \item (Monotonicity property) for all $t\geq0$ and $i \in \{1,\ldots,n\}$,
    \[
      X_t^{n,i}\leq Y_t^{n,i},
    \]
  \item (Decay estimate) for all $t\geq0$,
    \[
      \max_{i\in \{1,\ldots,n\}} (Y_t^{n,i}-X_t^{n,i})
      \leq \max_{i\in \{1,\ldots,n\}} (Y_0^{n,i}-X_0^{n,i})\mathrm{e}^{-t}.
    \]
  \end{itemize}
\end{lemma}

\begin{proof}[Proof of Lemma \ref{lemma:monotonicity}]
  The difference of $Y_t^n - X_t^n$ satisfies
  \begin{equation}\label{eq:partial t}
      \partial_t (Y_t^{n,i}-X_t^{n,i})=\frac{\beta}{n}\sum_{j:j\neq i}\frac{(Y_t^{n,j}-X_t^{n,j}) - (Y_t^{n,i}-X_t^{n,i}) }{ (Y_t^{n,j}-Y_t^{n,i})(X_t^{n,j}-X_t^{n,i}) }-(Y_t^{n,i}-X_t^{n,i}).
  \end{equation}
  Since there are almost surely no collisions between the coordinates of
  $X^n$, resp.~of $Y^n$, the right-hand side is almost surely finite for all
  $t > 0$ and every process $Y_t^{n,i}-X_t^{n,i}$ is $\mathcal{C}^1$ on
  $(0,\infty)$. Note that at time $0$ some derivatives may blow up as two
  coordinates of $X^n$ or $Y^n$ may coincide.
  
  Let us define
  \[
    M(t) = \max_{i \in \{1,\ldots,N\}} \,(Y_t^{n,i}-X_t^{n,i})
    \quad\text{and}\quad
    m(t) =\min_{i \in \{1,\ldots,N\}}\, (Y_t^{n,i}-X_t^{n,i}).
  \]
  Elementary considerations imply that $M$ and $m$ are themselves
  $\mathcal{C}^1$ on $(0,\infty)$ and that at all times $t>0$, there exist
  $i,j$ such that
  \[
    \partial_t M(t) = \partial_t (Y_t^{n,i}-X_t^{n,i})
    \quad\text{and}\quad
    \partial_t m(t) = \partial_t (Y_t^{n,j}-X_t^{n,j}).
  \]
  This would not be true if there were infinitely many processes of course.
  Now observe that if at time $t>0$ we have $Y_t^{n,i}-X_t^{n,i} = M(t)$, then
  \[
    \partial_t (Y_t^{n,i}-X_t^{n,i}) \le -(Y_t^{n,i}-X_t^{n,i}).
  \]
  This implies that $\partial_t M(t) \le - M(t)$. Similarly, we can deduce
  that $\partial_t m(t) \ge - m(t)$. Integrating these differential equations,
  we get for all $t\ge t_0 > 0$
  \[
    M(t) \le\mathrm{e}^{-(t-t_0)} M(t_0),\quad m(t) \ge\mathrm{e}^{-(t-t_0)} m(t_0).
  \]
  Since all processes are continuous on $[0,\infty)$, we can pass to the limit
  $t_0\downarrow 0$ and get for all $t\ge 0$,
  \[
    \min_{i \in \{1,\ldots,N\}}  (Y_t^{n,i}-X_t^{n,i}) \ge 0,\quad
    \max_{i \in \{1,\ldots,N\}} (Y_t^{n,i}-X_t^{n,i}) \le \mathrm{e}^{-t} \max_{i \in \{1,\ldots,N\}}
    (Y^{n,i}_0-X^{n,i}_0).
  \]
\end{proof}

\begin{remark}[Beyond the DOU dynamics]
  The monotonicity property of Lemma \ref{lemma:monotonicity} relies on the
  convexity of the interaction $-\log$, and has nothing to do with the
  long-time behavior and the strength of $V$. In particular, this monotonicity
  property remains valid for the process solving \eqref{eq:DOU} with an
  arbitrary $V$ provided that it is $\mathcal{C}^1$ and there is no explosion,
  even in the situation where $V$ is not strong enough to ensure that the
  process has an invariant law. If $V$ is $\mathcal{C}^2$ then the decay
  estimate of Lemma \ref{lemma:monotonicity} survives in the following decay
  or growth form:
  \[
    \max_{i\in \{1,\ldots,n\}}(Y_t^{n,i}-X_t^{n,i})
    \leq \max_{i\in \{1,\ldots,n\}}(Y_0^{n,i}-X_0^{n,i})
    \mathrm{e}^{t(-\inf_{\mathbb{R}}V'') },
    \quad t\geq0.
  \]
\end{remark}


%

\section*{Acknowledgements}

JB is supported by a grant from the ``Fondation CFM pour la Recherche''.

DC is supported by the project EFI ANR-17-CE40-0030.

CL is supported by the project SINGULAR ANR-16-CE40-0020-01.

\bibliographystyle{abbrv}
\bibliography{cdou.bib}

\begin{figure}[htbp]
  \centering
  \includegraphics[width=.6\textwidth]{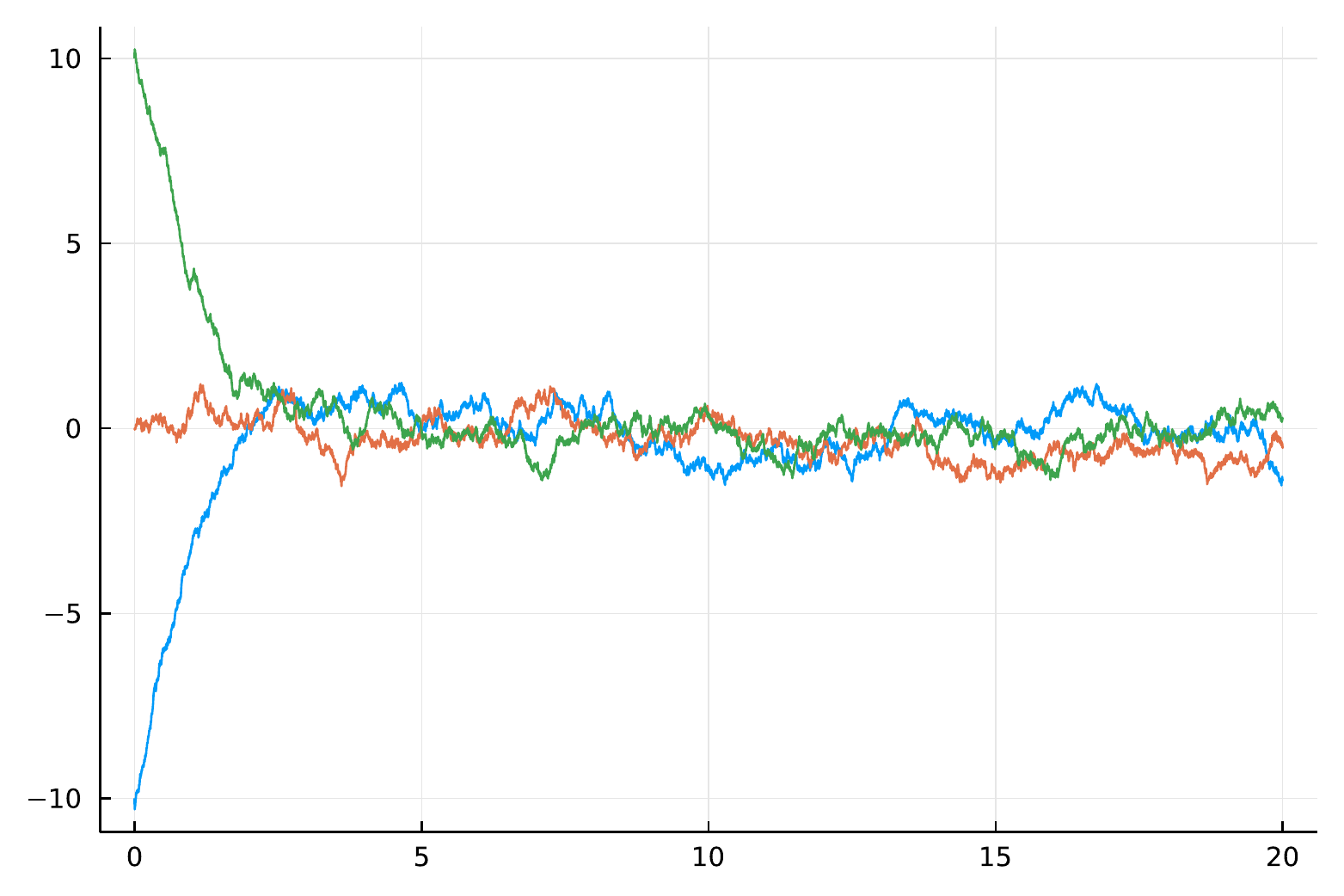}
  \includegraphics[width=.6\textwidth]{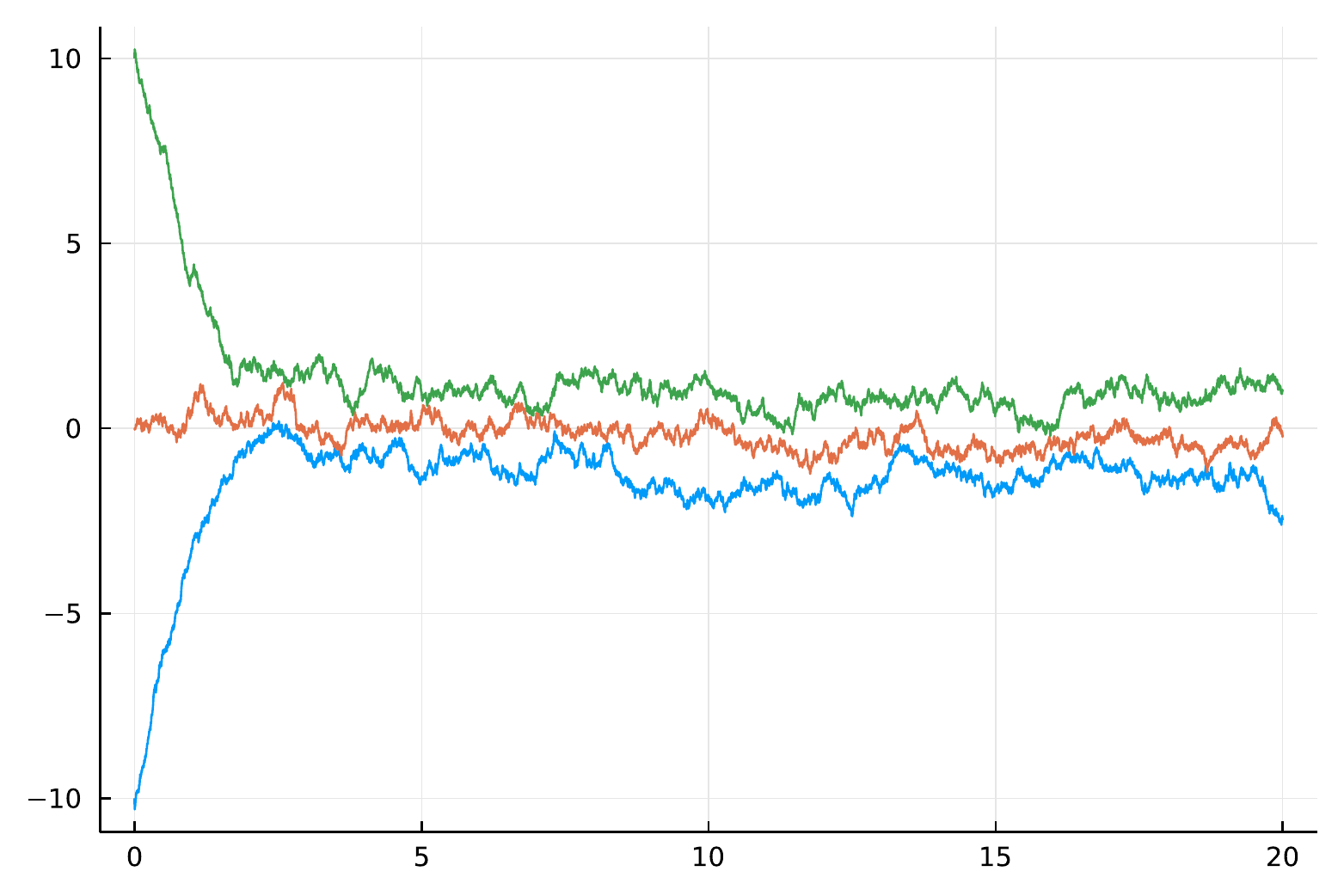}
  \caption{\label{fi:oudou}A trajectory of a single DOU with $n=3$ and
    $x^n_0=(-10,0,10)$, $\beta=0$ on top and $\beta=2$ on bottom. The driving
    Brownian motions are the same.}
\end{figure}

\begin{figure}[htbp]
  \centering
  \includegraphics[width=.6\textwidth]{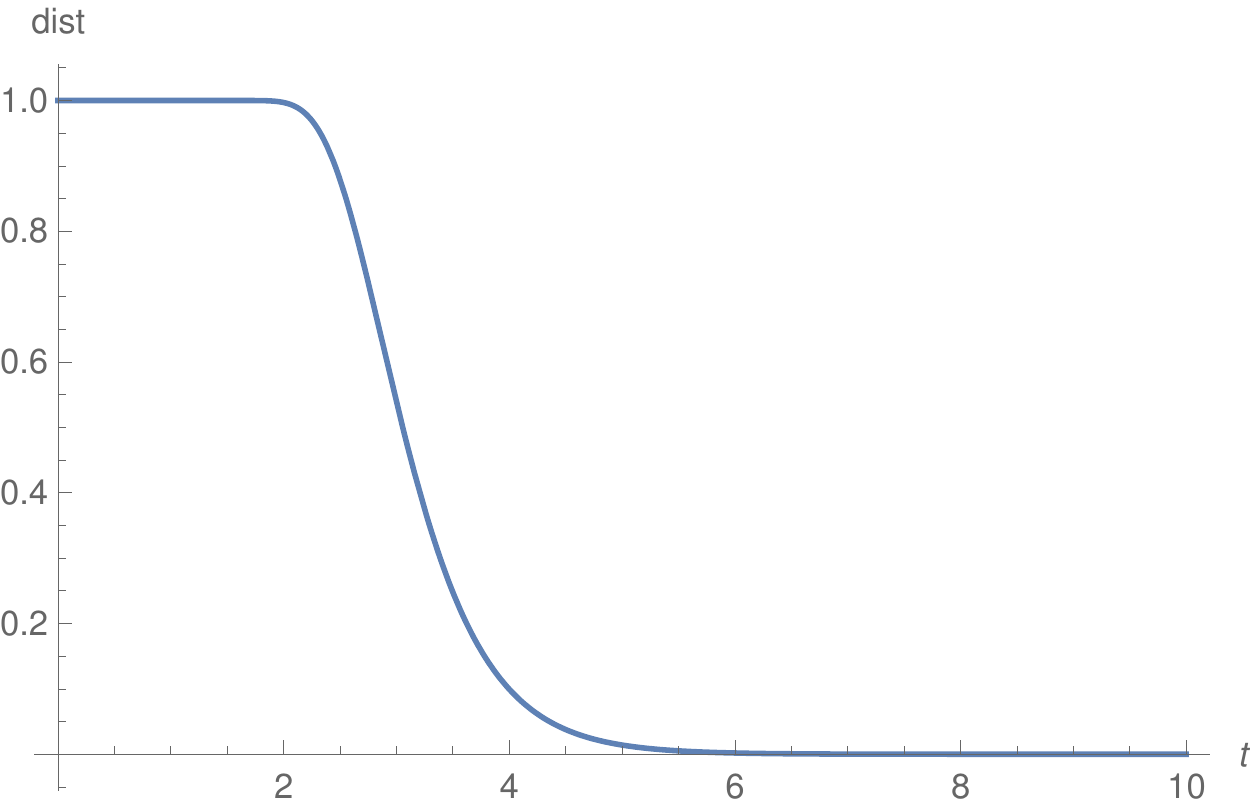}
  \caption{\label{fi:hellinger}Plot of the function
    $t\mapsto\mathrm{Hellinger}(\mathrm{Law}(X^n_t)\mid P^\beta_n)$ (see
    \eqref{eq:OUH} for the explicit formula) with $n=50$, $\beta=0$, and
    $\tfrac{|x^n_0|^2}{n}=1$. Note that $\log(50)\approx 3.9$.}
\end{figure}


\end{document}